\documentclass[10pt]{amsart}

\usepackage{mathtools,amssymb}
\usepackage{booktabs}
\usepackage{enumitem}
\usepackage{needspace}
\usepackage{setspace}
\usepackage[a4paper,left=3.2cm,right=3.2cm,top=3cm,bottom=3cm]{geometry}
\usepackage{xcolor}
\usepackage[hidelinks]{hyperref}
\usepackage[scr=boondox]{mathalfa}

\hypersetup{
  pdftitle={ON THE INDUCTIVE BLOCKWISE ALPERIN WEIGHT CONDITION FOR TYPE B AND TYPE C},
  pdfauthor={Baoyu Zhang}
}

\allowdisplaybreaks

\newtheorem{theorem}{Theorem}[section]
\newtheorem{proposition}[theorem]{Proposition}
\newtheorem{lemma}[theorem]{Lemma}
\newtheorem{corollary}[theorem]{Corollary}
\theoremstyle{definition}
\newtheorem{definition}[theorem]{Definition}
\newtheorem{remark}[theorem]{Remark}
\newtheorem*{remark*}{Remark}

\makeatletter

\renewcommand{\leq}{\leqslant}
\renewcommand{\geq}{\geqslant}
\expandafter\let\expandafter\oldproof\csname\string\proof\endcsname
\let\oldendproof\endproof
\renewenvironment{proof}[1][\proofname]{%
  \oldproof[\bfseries #1]%
}{\oldendproof}
\makeatother

\DeclareMathOperator{\Aut}{Aut}
\DeclareMathOperator{\Out}{Out}
\DeclareMathOperator{\Inn}{Inn}
\DeclareMathOperator{\IBr}{IBr}
\DeclareMathOperator{\Irr}{Irr}
\DeclareMathOperator{\Alp}{Alp}
\DeclareMathOperator{\Rad}{Rad}
\DeclareMathOperator{\dz}{dz}
\DeclareMathOperator{\Sp}{Sp}
\DeclareMathOperator{\PSp}{PSp}
\DeclareMathOperator{\CSp}{CSp}
\DeclareMathOperator{\Syl}{Syl}
\DeclareMathOperator{\bl}{bl}

\newcommand{\Z}{\mathbb Z}
\newcommand{\F}{\mathbb F}
\newcommand{\Baby}{\mathbb B}
\newcommand{\Monster}{\mathbb M}
\newcommand{\GAP}{\textsf{GAP}}

\numberwithin{equation}{section}
\numberwithin{table}{section}

\title[On the iBAW condition for Types $\mathsf B$ and $\mathsf C$]
{ON THE INDUCTIVE BLOCKWISE ALPERIN WEIGHT CONDITION
FOR TYPE $\mathsf B$ AND TYPE $\mathsf C$}

\author{Baoyu Zhang}
\address{School of Mathematics, University of Birmingham, Birmingham B15 2TT, UK}
\email{baoyuzhang.math@outlook.com}
\date{19 September 2026}


\begin{document}

\begin{abstract}
The purpose of this paper is to prove that every finite simple group of type
$\mathsf B$ or $\mathsf C$ satisfies the inductive blockwise Alperin
weight condition at every prime $\ell$ dividing its order.  For
$\PSp_{2n}(q)$, where $q$ is odd and $n\geq3$, a permutation lattice argument based on Conlon's induction theorem removes the
unitriangularity assumption from earlier results on the inductive condition at
odd nondefining primes.  At $\ell=2$ in odd defining characteristic,
errors and omissions in the underlying classification of radical
$2$-subgroups affect an earlier parametrisation of weights for arbitrary blocks.  We verify the required weight parametrisation for principal blocks
and use Jordan reduction to establish the inductive condition for all blocks.  At odd $\ell$ in even defining characteristic and rank at least four, we prove
the inductive condition using generic weights and Jordan reduction, without
assuming unitriangularity.
For $\Omega_{2n+1}(q)$, where $q$ is odd and $n\geq3$, we remove the
unitriangularity assumption at odd nondefining primes using Conlon's
induction theorem. At $\ell=2$, we prove the stabiliser and extension property for Brauer
characters of $\operatorname{Spin}_{2n+1}(q)$ that was assumed in earlier work
on the inductive condition.
We also give proofs of the inductive condition for the remaining sporadic
groups $J_4$, $Fi'_{24}$,
the Baby Monster and the Monster, without making a priority claim.
Together with the previously established cases, these results imply that
the blockwise Alperin weight conjecture holds at $\ell$
for every finite group each of whose nonabelian simple sections of order
divisible by $\ell$ is of type $\mathsf B$, of type $\mathsf C$, or sporadic.
\end{abstract}

\maketitle

\section{Introduction}

Let $b$ be an $\ell$-block of a finite group $G$.  Alperin's weight conjecture
asserts that the number of isomorphism classes of simple modules in $b$ equals
the number of $G$-conjugacy classes of weights belonging to $b$
\cite{Alperin1987}.  The inductive blockwise Alperin weight condition
(\emph{iBAW condition}), introduced by Sp\"ath
\cite[Definition~4.1]{Spath2013}, refines this numerical equality by requiring
compatible equivariant bijections.  The purpose of this paper is to prove the
inductive blockwise Alperin weight condition for the finite simple groups of
types $\mathsf B$ and $\mathsf C$.

Combining the results proved here with the previously established cases,
we obtain the following theorem.

\begin{theorem}\label{thm:main}
Let $S$ be a nonabelian finite simple group of type $\mathsf B$, of type
$\mathsf C$, or sporadic. Then $S$ satisfies the inductive blockwise
Alperin weight condition at every prime dividing its order.
\end{theorem}

Let $q=p^f$ be odd and let $\ell$ be an odd prime not dividing $q$.
Feng, Li, and Zhang \cite[Theorem~1.5]{FengLiZhang2019} proved the
inductive condition for unipotent blocks of $\Sp_{2n}(q)$ and
$\operatorname{Spin}_{2n+1}(q)$, where $n\geq2$. When $f$ is odd and $\ell$ is
\emph{linear}, meaning that the multiplicative order of $q$ modulo
$\ell$ is odd, they also proved the inductive condition for
$\PSp_{2n}(q)$ and for the blocks of $\operatorname{Spin}_{2n+1}(q)$
that dominate blocks of $\Omega_{2n+1}(q)$
\cite[Theorems~1.2 and~1.4]{FengLiZhang2019}.
Li's result for type~$\mathsf C$
\cite[Theorem~2]{Li2021} and Feng--Li--Zhang's result for type~$\mathsf B$
\cite[Theorem~1]{FengLiZhang2022TypeB} treat arbitrary blocks
under unitriangularity assumptions. We remove these
assumptions in rank at least three. Integral basic sets for the conformal
symplectic and finite special Clifford groups identify the permutation
lattices spanned by these basic sets with those spanned by the irreducible
Brauer characters of the corresponding blocks. Conlon's induction theorem \cite{Conlon1968},
applied at the prime $2$, then yields equivariant bijections between
irreducible Brauer characters and the ordinary characters in these basic
sets. Applying the same argument to the symplectic and spin groups
gives the same stabiliser factorisations for Brauer characters as for
ordinary characters, allowing us to verify the inductive condition.

For $\Sp_{2n}(q)$ with even $q$ and odd $\ell$, Li
\cite[Section~6]{Li2021} obtains the required descriptions of blocks
and weights by adapting arguments for odd characteristic of
An \cite{An1994} and Fong--Srinivasan \cite{FongSrinivasan1989}.
Li proves the inductive condition under a unitriangularity assumption
\cite[Theorem~3]{Li2021}. For $n\geq4$, we give an independent proof
of the inductive condition without assuming unitriangularity, using
Geck's basic set theorem \cite{Geck1993} and the generic weight
correspondences of Feng, Malle, and Zhang \cite{FengMalleZhang2026}.
These give bijections between the irreducible Brauer characters and
weights of unipotent blocks. The resulting bijections are equivariant
because field automorphisms fix the unipotent characters and the conjugacy
classes of generic weights involved, as explained in
Lemma~\ref{lem:even-field-fixed}. After verifying the remaining parts of the
inductive condition for unipotent blocks, we obtain the result for all blocks by
Jordan reduction \cite{FengLiZhang2022Jordan}.

For odd $q$ at the prime $2$, we give an alternative proof of Feng and
Malle's result for $\PSp_{2n}(q)$ \cite[Theorem~1]{FengMalle2022}.
The classification of radical $2$-subgroups \cite[p.~190]{An1993}
used in Feng and Malle's weight parametrisation omits a family with even
multiplicity, as noted by Feng, Yu, and Zhang
\cite[Remark~3.36]{FengYuZhang2024}. Using the corrected classification,
we verify the principal block weight calculation needed for Feng and
Malle's argument and obtain the result for all blocks by Jordan reduction
\cite{FengLiZhang2022Jordan}.

For type $\mathsf B$ at the prime $2$, Feng, Yu, and Zhang's proof
assumes a Brauer character stabiliser and extension property
\cite[Theorem~1 and Assumption~5.1]{FengYuZhang2024}. We establish this
property for $\operatorname{Spin}_{2n+1}(q)$, where $q$ is odd and
$n\geq3$. For the principal block, we use generalised Gelfand--Graev
characters to prove that field automorphisms fix every irreducible Brauer
character. For nonprincipal blocks, we obtain the required stabiliser
factorisations and extensions through the equivariant Jordan
correspondence and proper Levi subgroups. Jordan reduction then gives the
inductive condition for $\Omega_{2n+1}(q)$ when $(n,q)\neq(3,3)$.
We treat the exceptional covering group $3.\Omega_7(3)$ separately by
explicit calculations.

We also give proofs of the inductive condition for $J_4$, $Fi'_{24}$,
$\Baby$ and $\Monster$. The proofs combine published weight counts
with character table calculations and, for $Fi'_{24}$, a comparison of
the actions of the outer automorphism on Brauer characters and weights.
As of 19 September 2026, the website of the \GAP~package CTBlocks
\cite{BreuerSporadicOverview} still listed the verification for these
four groups as incomplete. In a 2023 paper, Feng, Li, and Zhang cite
Breuer and Feng's work on sporadic simple groups as \emph{in preparation}
\cite[p.~6518]{FengLiZhang2023Morita}.
We have not found a publicly available manuscript of that project and
therefore cannot compare its scope with the present work. Accordingly,
we make no priority claim for the sporadic theorem.

Sp\"ath's reduction theorem \cite[Theorem~A]{Spath2013} gives the following
consequence.  A group $K$ is said to be \emph{involved} in a finite group
$H$ if $K$ is isomorphic to a section of $H$.

\begin{corollary}\label{cor:reduction}
Let $H$ be a finite group and let $\ell$ be a prime.  Suppose that every
nonabelian simple group involved in $H$ whose order is divisible by $\ell$
is of type $\mathsf B$, of type $\mathsf C$, or sporadic.  Then the blockwise
Alperin weight conjecture holds for $H$ at $\ell$.
\end{corollary}

\section{Preliminaries}\label{sec:preliminaries}

For modular representations at a prime $\ell$, we use a splitting
$\ell$-modular system for the finite groups under consideration, with
residue field $k$. If $b$ is an $\ell$-block, then
$\Irr(b)$ and $\IBr(b)$ denote its sets of ordinary irreducible characters
and irreducible Brauer characters, respectively.  For an ordinary irreducible
character or irreducible Brauer character $\psi$, let $\bl(\psi)$ denote the
block containing $\psi$, and let
$\dz(G)$ denote the set of ordinary
irreducible characters of $G$ \emph{of defect zero}, namely those whose
degrees have the same $\ell$-part as $|G|$.
For $\chi\in\Irr(G)$, $\chi^0$ denotes its
restriction to the \emph{$\ell$-regular elements}, namely the elements of order
prime to $\ell$.  The \emph{decomposition map}
$d_G:\Z[\Irr(G)]\longrightarrow\Z[\IBr(G)]$ is the $\Z$-linear map
determined by $d_G(\chi)=\chi^0$. For an $\ell$-block $b$ of $G$, write
$d_b$ for the restriction of $d_G$ from
$\Z[\Irr(b)]$ to $\Z[\IBr(b)]$.
We write automorphisms on the right. For a group $H$, a subgroup
$Q\leq H$, and $\alpha\in\Aut(H)$, set $Q^\alpha=\alpha^{-1}(Q)$.
If $Q$ is finite and $\chi$ is a class function of $Q$, define the class
function $\chi^\alpha$ of $Q^\alpha$ by
$\chi^\alpha(h)=\chi(\alpha(h))$ for $h\in Q^\alpha$.
For Brauer characters, this formula is evaluated on the $\ell$-regular
elements of $Q^\alpha$.
If a group $A$ acts on sets $\Omega_1,\ldots,\Omega_m$ and
$x_i\in\Omega_i$, we write $A_{x_i}$ for the stabiliser of $x_i$ and
$A_{x_1,\ldots,x_m}=A_{x_1}\cap\cdots\cap A_{x_m}$.
For a finite group $G$, let $O_\ell(G)$ denote its \emph{$\ell$-core}, namely
its largest normal $\ell$-subgroup.  A subgroup $Q\leq G$ is
\emph{$\ell$-radical} if
$Q=O_\ell(N_G(Q))$, and $\Rad_\ell(G)$ denotes the set of
$\ell$-radical subgroups of $G$.
An \emph{$\ell$-weight} of $G$ is a pair $(Q,\varphi)$, where
$Q\in\Rad_\ell(G)$ and $\varphi\in\Irr(N_G(Q)/Q)$ has defect zero.  We call
$\varphi$ the \emph{weight character} and, when convenient, identify it with
its inflation to $N_G(Q)$.  The weight \emph{belongs to}
an $\ell$-block $b$ of $G$ if the block of $N_G(Q)$ containing the inflated
weight character induces to $b$.  We write $\Alp(b)$ for the set of $G$-conjugacy classes of weights
belonging to $b$, and $\Alp(G)$ for the set of $G$-conjugacy classes of
all $\ell$-weights of $G$.
Blocks and weights carry the induced right actions of automorphisms.  In
particular, $(Q,\varphi)^\alpha=(Q^\alpha,\varphi^\alpha)$.

Let $S$ be nonabelian simple and let $\widehat S$ be its
\emph{universal central extension}.  For a prime $\ell$, set
$X=X_\ell(S)=\widehat S/O_\ell(Z(\widehat S))$ and $Z=Z(X)$.
The group $X$ is the \emph{universal $\ell'$-covering group} used in
Sp\"ath's inductive condition \cite[Definition~4.1]{Spath2013}.  Its centre
$Z$ is an $\ell'$-group and $X/Z\cong S$.  The quotient map
$X\longrightarrow S$ induces an isomorphism
$\Aut(X)\longrightarrow\Aut(S)$ \cite[p.~184]{Spath2013}, and
consequently $\Out(X)\cong\Out(S)$.

Since $Z$ is an $\ell'$-group, the decomposition map $d_Z$ identifies
$\Irr(Z)$ with $\IBr(Z)$. We use \emph{sector} as shorthand in the calculations for covering groups,
where we compare Brauer character and weight counts over each central
character and describe the action of outer automorphisms.

\begin{definition}\label{def:central-sector}
For $\nu\in\Irr(Z)$, an ordinary irreducible or irreducible Brauer character
$\psi$ of $X$ \emph{lies over} $\nu$ if its restriction to $Z$ is a positive
integral multiple of $\nu$. A weight $(Q,\theta)$ lies over $\nu$ if the
inflation of $\theta$ to $N_X(Q)$ does. We say that such characters and
$X$-conjugacy classes of weights belong to the \emph{$\nu$-sector}. The sector
is called \emph{trivial} when $\nu=1_Z$, and \emph{faithful} when $\nu$ is
faithful.
\end{definition}

Under the action defined above, $\alpha\in\Aut(X)$ sends the $\nu$-sector onto
the $\nu^\alpha$-sector. Since $Z$ is a central $\ell'$-subgroup,
each block $b$ lies over a unique $\nu\in\Irr(Z)$: its block idempotent
lies under exactly one primitive idempotent of $kZ$.
All ordinary and Brauer characters in $b$ lie over $\nu$.
Block induction preserves this central character
\cite[Definition~2.1]{Spath2013}, so every weight belonging to $b$
also lies over $\nu$. We say that $b$ lies in the $\nu$-sector.

\begin{lemma}\label{lem:linear-block-stabilizer}
Let $N\lhd G$ be finite groups, let $b$ be an $\ell$-block of $G$,
and let $C$ be the group of ordinary linear characters of $G/N$
of $\ell'$-order, acting on blocks by tensoring with their inflations.
Then every character in $C_b$ factors through $G/NZ(G)$, and
$|C_b|\leq |G:NZ(G)|$. The same conclusions hold when $C$ is
identified by reduction with the group of linear Brauer characters of $G/N$.
\end{lemma}

\begin{proof}
Let $Z_{\ell'}$ be the subgroup of $Z(G)$ consisting of its
$\ell$-regular elements. Since $Z_{\ell'}$ is a central
$\ell'$-subgroup, all characters in $\IBr(b)$ lie over the same character
$\nu\in\IBr(Z_{\ell'})$. For $\lambda\in C_b$, tensoring preserves
$b$, so $(\lambda^0|_{Z_{\ell'}})\nu=\nu$. Thus $\lambda$ is
trivial on $Z_{\ell'}$. Its $\ell'$-order also makes it trivial on
$O_\ell(Z(G))$, and hence on $Z(G)$. Since $\lambda$ is already trivial on $N$, it factors through
$G/NZ(G)$. Thus $|C_b|\leq |G:NZ(G)|$.
Reduction identifies ordinary linear characters of $\ell'$-order
with linear Brauer characters and commutes with tensoring, so the
same argument applies to their action on blocks.
\end{proof}

In Sp\"ath's iBAW condition \cite[Definition~4.1(i)--(iv)]{Spath2013},
the local bijections give an $\Aut(X)_b$-equivariant bijection
$\IBr(b)\longrightarrow\Alp(b)$ for each $\ell$-block $b$ of $X$.
The condition also requires compatibility with central characters,
character extensions and blocks of intermediate groups. For weights with
subgroup $Q=1$, the character $\chi^0$ is matched with $\chi\in\dz(X)$,
with identical global and local extensions.
We say that $S$ satisfies iBAW at $\ell$ when the full condition holds.

Let $G\lhd A$ be finite groups, and let $b$ be an $A$-stable
$\ell$-block of $G$. An $A$-equivariant bijection
$\Omega:\IBr(b)\longrightarrow\Alp(b)$ is an
\emph{iBAW bijection compatible with $A$} if, for every
$\chi\in\IBr(b)$ and every representative $(Q,\varphi)$ of
$\Omega(\chi)$,
\[
 (A_\chi,G,\chi)\succeq_b
 (N_A(Q)_\varphi,N_G(Q),\varphi^0).
\]
Here $\succeq_b$ denotes block isomorphism of modular character triples
\cite[Definition~3.4]{MartinezRizoRossi2026}, and $\varphi^0$ is
the Brauer reduction of the weight character, inflated to $N_G(Q)$.
Taking $A=G\rtimes\Aut(G)_b$ gives an \emph{iBAW bijection} for $b$
in the sense of \cite[Definition~3.5]{FengLiZhang2022Jordan}.
For the block form of the inductive BAW condition, let $b$ be an
$\ell$-block of $X=X_\ell(S)$ with central character $\nu$, and let
$\overline b$ be the block it dominates on $X/\ker(\nu)$.
The characters in $\IBr(\overline b)$ are faithful on the centre of
this quotient. Following \cite[Definition~3.2]{KoshitaniSpath2016},
we say that $b$ satisfies the inductive BAW condition when $\overline b$ does.

We recall Sp\"ath's character triple formulation of the extension and block
compatibility requirements in the inductive BAW condition.

\begin{theorem}[Sp\"ath {\cite[Theorem~4.4]{Spath2017Triples}}]\label{thm:spath-triples}
Let $\widehat G$ be the universal covering group of a nonabelian simple
group $S$, and let $\ell$ be a prime.  Let
$\widehat\theta\in\IBr(\widehat G)$,
$\widehat Q\in\Rad_\ell(\widehat G)$ and
$\widehat M=N_{\widehat G}(\widehat Q)$.  Suppose that
$\widehat\theta'$ is the inflation to $\widehat M$ of a character in
$\dz(\widehat M/\widehat Q)$ and that
$\bl(\widehat\theta')^{\widehat G}=\bl(\widehat\theta)$.
Then the following are equivalent.
\begin{enumerate}[label=\textup{(\roman*)}]
\item The characters $\widehat\theta$ and $\widehat\theta'$ satisfy
part~\textup{(iii)} of the inductive blockwise Alperin weight condition
in \cite[Definition~4.1]{Spath2013}.
\item The modular character triples satisfy
\begin{equation}\label{eq:spath-block-triples}
 (G\rtimes\Aut(G)_{\widehat\theta},G,\widehat\theta)
 \succ_{F,b}
 \bigl(M(G\rtimes\Aut(G))_{\widehat M,\widehat\theta'},
       M,(\widehat\theta')^0\bigr),
\end{equation}
where $G=\widehat G/\widehat Z$ and $M=\widehat M/\widehat Z$ for
$\widehat Z=\ker(\widehat\theta_{Z(\widehat G)})$.
\end{enumerate}
\end{theorem}

Here $F$ is an algebraically closed field of characteristic $\ell$.
Sp\"ath writes $\succ_{F,b}$ for block isomorphism of modular character
triples, denoted by $\succeq_b$ elsewhere in this paper. In~\textup{(ii)},
$\widehat\theta$ and $\widehat\theta'$ are regarded as characters of $G$
and $M$, respectively.

Assume that $S$ satisfies iBAW at $\ell$, and let
$\Omega:\IBr(b)\to\Alp(b)$ be the resulting bijection for a block $b$
of $X=X_\ell(S)$. For each $\psi\in\IBr(b)$, choose a representative
$(Q,\varphi)$ of $\Omega(\psi)$. Replace $Q$ by its full preimage in
the universal covering group, and inflate $\psi$ and $\varphi$ to the
cover and the corresponding normaliser, respectively. The normaliser
quotient is unchanged, and the lifted weight belongs to the block of the
inflated Brauer character. The full inductive condition
supplies part~\textup{(i)} of Theorem~\ref{thm:spath-triples}, so
part~\textup{(ii)} gives the character triple condition on
$X/(\ker(\psi)\cap Z(X))$.
Since $\ker(\psi)\cap Z(X)$ is a central $\ell'$-subgroup, inflation gives
the condition on $X$ by \cite[Remark~4.4]{BroughSpath2022}. Thus $\Omega$ is an iBAW
bijection for $b$.

\begin{remark}\label{rem:block-aggregation}
\textup{(a)} By \cite[Lemma~3.3]{KoshitaniSpath2016} and
\cite[Remark~5.18]{Spath2013}, it is enough to verify the inductive BAW
condition for one representative of each $\Aut(X_\ell(S))$-orbit of
$\ell$-blocks of $X_\ell(S)$ in order to prove that $S$ satisfies iBAW at
$\ell$.

\textup{(b)} We will also combine bijections on individual blocks by the following
elementary observation. Suppose that a group $A$ acts on $\mathcal B$ and on
disjoint unions $\mathcal U=\coprod_{b\in\mathcal B}\mathcal U_b$ and
$\mathcal V=\coprod_{b\in\mathcal B}\mathcal V_b$, with
$\mathcal U_b^a=\mathcal U_{b^a}$ and
$\mathcal V_b^a=\mathcal V_{b^a}$ for $a\in A$.
For one representative $b$ of each $A$-orbit, choose an
$A_b$-equivariant bijection $\Phi_b:\mathcal U_b\to\mathcal V_b$.
Then $\Phi(x^a)=\Phi_b(x)^a$, for $x\in\mathcal U_b$, defines an
$A$-equivariant bijection $\Phi:\mathcal U\to\mathcal V$, with
$\Phi(\mathcal U_b)=\mathcal V_b$ for every $b\in\mathcal B$.
The $A_b$-equivariance makes this definition independent of the choice of $a$.
\end{remark}

\begin{corollary}[Sp\"ath]\label{cor:fixed-point-descent}
Let $S$ be a nonabelian simple group.  If $S$ satisfies iBAW at $\ell$, then
every $\ell$-block of $S$ admits an iBAW
bijection.
\end{corollary}

\begin{proof}
Apply \cite[Proposition~4.6(a)--(c)]{Spath2013} with $K=S$,
$G=\Aut(S)$ and $r=1$. The resulting blockwise bijections satisfy
the character triple condition by
\cite[Proposition~3.6(b) and Theorem~3.5(b)]{Spath2017Triples}.
\end{proof}

The following lemma uses the argument in the proof of
\cite[Lemma~6.3]{MartinezRizoRossi2026} for a central $\ell$-subgroup.

\begin{lemma}\label{lem:type-b-normal-core}
Let $G\lhd A$ be finite groups, and suppose that
$P=O_\ell(G)\leq Z(G)$. Set $\overline G=G/P$ and
$\overline A=A/P$. Let $b$ be an $A$-stable $\ell$-block of $G$,
and let $\overline b$ be the block of $\overline G$ dominated by $b$.
If $\overline b$ admits an iBAW bijection compatible with
$\overline A$, then $b$ admits an iBAW bijection compatible with $A$.
\end{lemma}

\begin{proof}
Since $P$ is central in $G$, the block correspondence for
$G\longrightarrow\overline G$ identifies the irreducible Brauer
characters of $b$ and $\overline b$ under inflation
\cite[Theorem~9.10]{Navarro1998}.
Every $\ell$-radical subgroup $Q$ of $G$ contains $P$. Writing
$\overline Q=Q/P$, the map $Q\longmapsto\overline Q$ gives a bijection
between the $\ell$-radical subgroups of $G$ and $\overline G$, with
\[
 N_G(Q)/Q\cong N_{\overline G}(\overline Q)/\overline Q.
\]
By \cite[Lemma~2.1]{FengLiZhang2019} and the block correspondence above,
these identifications also give a bijection between $\Alp(b)$ and
$\Alp(\overline b)$. All these maps commute with the action of $A$.

Composing these correspondences with the given iBAW bijection for
$\overline b$ gives an $A$-equivariant bijection
$\IBr(b)\longrightarrow\Alp(b)$. For $\chi\in\IBr(b)$, let
$(Q,\varphi)$ represent its image. Then $P$ is
normal in $A_\chi$, lies in $N_G(Q)$, and is contained
in the kernels of $\chi$ and $\varphi^0$.
The quotient maps identify $A_\chi/P$ and $N_A(Q)_\varphi/P$ with
the corresponding stabilisers in $\overline A$.
Hence \cite[Lemma~3.14]{MartinezRizoRossi2026} lifts the required
block isomorphism of modular character triples, proving the assertion.
\end{proof}

We use \emph{BAW-good} in the sense of
\cite[Section~3.5]{FengLiZhang2022Jordan} for blocks of universal
$\ell'$-covering groups. For these groups,
Theorem~\ref{thm:spath-triples} identifies BAW-goodness with the
inductive BAW condition, which supplies an iBAW bijection.
For a full universal covering group $G$, set
$P=O_\ell(G)=O_\ell(Z(G))$. Following
\cite[Theorems~3.18 and~5.7]{FengLiZhang2022Jordan}, we also call a
block $b$ of $G$ BAW-good if the block $\overline b$ it dominates on
$G/P$ is BAW-good. Such a block $b$ also admits an iBAW bijection.
Set $\overline G=G/P$. The quotient map induces an isomorphism
$\Aut(G)\longrightarrow\Aut(\overline G)$, which identifies
$\Aut(G)_b$ with $\Aut(\overline G)_{\overline b}$. The block
$\overline b$ admits an iBAW bijection, so
Lemma~\ref{lem:type-b-normal-core}, applied with
$A=G\rtimes\Aut(G)_b$, gives an iBAW bijection for $b$.

For the permutation arguments below, actions on $A$-sets and $\Z A$-lattices
are written on the left.
If a set $Y$ already carries a right action, as do the character sets above,
we pass to the left action $a\cdot y=y^{a^{-1}}$.

For a finite group $A$, let $B(A)$ denote its Burnside ring, and let
$\mathcal S$ be a set of representatives of the $A$-conjugacy classes of
subgroups of $A$.  For a finite $A$-set $Y$, write $[Y]$ for the element of $B(A)$
represented by $Y$, and, for $U\leq A$, let $Y^U$ be the set of $U$-fixed points of
$Y$.  The \emph{mark homomorphism}
\[
        B(A)\longrightarrow \prod_{U\in\mathcal S}\Z,\qquad
        [Y]\longmapsto (|Y^U|)_{U\in\mathcal S},
\]
is injective
\cite[Theorem~2.3.2 and Section~3.2]{Bouc2000}.

Let $r$ be a prime. A finite group $A$ is
\emph{$r$-hypoelementary} if $A/O_r(A)$ is cyclic.
Every subgroup of an $r$-hypoelementary group is $r$-hypoelementary.

Let $\Z_r$ denote the ring of $r$-adic integers. For a finite $A$-set
$Y$, write $\Z_r[Y]$ for the \emph{permutation lattice} with basis $Y$.
Conlon's theorem
\cite{Conlon1968}, in the form stated in
\cite[Theorem~3.5.5]{Bouc2000}, says that
$\Z_r[Y]\cong\Z_r[Y']$ as $\Z_rA$-lattices if and only if
$|Y^U|=|Y'^U|$ for every $r$-hypoelementary subgroup $U\leq A$.
We will use the following consequence with $r=2$.

\begin{corollary}[Conlon]\label{cor:conlon-mark}
Let $A$ be a finite $r$-hypoelementary group and let $Y,Y'$ be finite
$A$-sets.  If the permutation lattices
$\Z_r[Y]$ and $\Z_r[Y']$ are isomorphic as $\Z_rA$-lattices, then
$Y\cong Y'$ as $A$-sets.
\end{corollary}

\begin{proof}
Every subgroup of $A$ is $r$-hypoelementary. Conlon's theorem therefore
gives $|Y^U|=|Y'^U|$ for every $U\leq A$. The injectivity of the mark
homomorphism gives $Y\cong Y'$ as $A$-sets.
\end{proof}

For an $\ell$-block $b$, a \emph{basic set} of Brauer characters is a
$\Z$-basis of $\Z[\IBr(b)]$ consisting of Brauer characters
\cite[(1.1)--(1.2)]{GeckHiss1991}.
A subset $\mathcal X\subseteq\Irr(b)$ is an \emph{integral basic set}
for $b$ if the decomposition map restricts to an isomorphism
\[
  d_b:\Z[\mathcal X]\xrightarrow{\sim}\Z[\IBr(b)].
\]

\begin{lemma}\label{lem:basicset-bridge}
Let $G$ be a finite group, let $b$ be an $\ell$-block of $G$, and let $A$ be
a finite group acting on $\Irr(b)$ and $\IBr(b)$ such that $d_b$ is
$A$-equivariant.  Let $\mathcal X\subseteq\Irr(b)$ be an $A$-stable integral
basic set.  Then $d_b$ restricts to an isomorphism
$\Z[\mathcal X]\cong\Z[\IBr(b)]$ of $\Z A$-permutation lattices.
Consequently the underlying finite sets are $A$-isomorphic if $A$ is
$2$-hypoelementary.
\end{lemma}

\begin{proof}
By the definition of an integral basic set, this restriction is an
isomorphism of $\Z$-modules.  Its $A$-equivariance makes it an isomorphism
of $\Z A$-permutation lattices.
If $A$ is $2$-hypoelementary, extend scalars to $\Z_2$ and apply
Corollary~\ref{cor:conlon-mark}.
\end{proof}

We shall repeatedly use the following standard result.

\begin{lemma}\label{lem:cyclic-extension}
Let $H$ be a finite group, let $N\triangleleft H$, and suppose that $H/N$ is
cyclic. Then every $H$-invariant character $\theta\in\Irr(N)\cup\IBr(N)$ extends to $H$.
\end{lemma}

\begin{proof}
Apply \cite[Corollary~11.22]{Isaacs1976} for ordinary characters and
\cite[Theorem~8.12]{Navarro1998} for Brauer characters.
\end{proof}

The following lemma is a special case of Brough--Sp\"ath's criterion
\cite[Theorem~3.18]{FengLiZhang2022Jordan}, with the required extensions
supplied by Lemma~\ref{lem:cyclic-extension}.

\begin{lemma}\label{lem:cyclic-outer-baw}
Let $G$ be a nonabelian simple group that is its own universal covering
group, let $\ell$ be a prime dividing $|G|$, and suppose that
$\Aut(G)=G\rtimes E$ for a cyclic group $E$.  If $b$ is an $E$-stable
$\ell$-block of $G$ and there is an $E$-equivariant bijection
\[
  \Omega_b:\IBr(b)\longrightarrow\Alp(b),
\]
then $b$ is BAW-good.
\end{lemma}

\begin{proof}
Apply Brough--Sp\"ath's criterion
\cite[Theorem~3.18]{FengLiZhang2022Jordan} with
$\widetilde G=G$, $\mathcal B=\{b\}$, and
$\Omega=\widetilde\Omega=\Omega_b$.
The group $G$ is perfect, $E$ is abelian, and
$C_{G\rtimes E}(G)=1=Z(G)$, so conditions~\textup{(i)(a)} and
\textup{(i)(b)} of the criterion hold.
Inner automorphisms act trivially on $\IBr(b)$ and on the $G$-classes
of weights, so $\Omega_b$ is $G\rtimes E$-equivariant.
For $\varphi\in\IBr(b)$, the quotient
$(G\rtimes E)_\varphi/G$ is cyclic, so
Lemma~\ref{lem:cyclic-extension} extends $\varphi$ to its stabiliser.
Choose a representative $(Q,\theta)$ of $\Omega_b(\varphi)$.
Identifying $\theta$ with its inflation to $N_{G}(Q)$, we have
$\bl(\varphi)=\bl(\theta)^{G}=b$ because $(Q,\theta)$ belongs to $b$.
This verifies the block induction condition~\textup{(iii)(c)} of the
criterion.

Set $D=(G\rtimes E)_{Q,\theta}$.
Then $D\cap G=N_{G}(Q)$ and $D/N_{G}(Q)$ embeds in $E$.
The irreducible Brauer reduction $\theta^0$ of $\theta$, viewed as a
character of $N_{G}(Q)/Q$, is $D/Q$-invariant.
Lemma~\ref{lem:cyclic-extension} extends it to $D/Q$, and inflation
gives the required extension to $D$.
Since $G=\widetilde G$ and $Z(G)=1$, the remaining conditions on
restriction, covering, central characters and stabiliser products are
immediate, and the extensions within $\widetilde G$ are given by the
characters themselves.  Hence $b$ is BAW-good.
\end{proof}

Let $\mathbf Y$ be a connected reductive algebraic group with Frobenius
endomorphism $F_Y$, and set $Y=\mathbf Y^{F_Y}$. Let $\mathbf Y^*$ be
the dual group, equipped with the dual Frobenius endomorphism, also
denoted by $F_Y$, and set $Y^*=(\mathbf Y^*)^{F_Y}$. For a semisimple
element $u\in Y^*$, write $\mathcal E(Y,u)$ for the \emph{rational
Lusztig series} indexed by the $Y^*$-conjugacy class of $u$.
For $\chi\in\Irr(Y)$, write $\chi^*$ for the irreducible character
equal, up to sign, to its Alvis--Curtis dual $D_Y(\chi)$.

For an irreducible character $\chi$ of a finite reductive group,
let $n_\chi$ denote its \emph{generic denominator}, as defined in \cite[Section~3(B), p.~443]{GeckMalle2000}.
Equivalently, $n_\chi$ is the reciprocal of the leading coefficient of
its degree polynomial \cite[Remark~2.3.28]{GeckMalle2020}.

The proofs of the following two lemmas are given in
Appendix~\ref{app:gggr-proofs}.

\begin{lemma}\label{lem:type-b-isolated-coverage}
Let $\mathbf H$ be a connected reductive algebraic group with connected
centre and simply connected derived subgroup of type $\mathsf B_n$,
where $n\geq3$. Suppose that $\mathbf H$ is defined over
$\mathbb F_q$, where $q$ is a power of an odd prime $p$, and let $F$
be the corresponding Frobenius endomorphism. Set $H=\mathbf H^F$.
Choose an $F$-stable Borel subgroup $\mathbf B^*$ of $\mathbf H^*$
and an $F$-stable maximal torus $\mathbf T^*\subseteq\mathbf B^*$.
Let $s\in(\mathbf T^*)^F$ be an isolated semisimple element.
Let $W_s$ be the Weyl group of $C_{\mathbf H^*}(s)$ relative to
$\mathbf T^*$, and let $\mathcal F\subseteq\Irr(W_s)$ be a Lusztig
family \cite[Definition~4.1.17]{GeckMalle2020}.
Denote by $\mathcal G_{\mathcal F}$ the finite group attached to
$\mathcal F$ in Lusztig's parametrisation, and by
$\Irr_{s,\mathcal F}(H)$ the corresponding family of irreducible
characters in $\mathcal E(H,s)$.
Let $\mathcal C$ be the $F$-stable unipotent class associated with
$(s,\mathcal F)$. Suppose that, for some $u_0\in\mathcal C^F$, the component
group $A_{\mathbf H}(u_0)=C_{\mathbf H}(u_0)/C_{\mathbf H}(u_0)^\circ$
and the group $\mathcal G_{\mathcal F}$ are elementary abelian of the same order,
and $F$ acts trivially on $A_{\mathbf H}(u_0)$.
Let $u_1,\ldots,u_d$ represent the $H$-conjugacy classes in
$\mathcal C^F$, and let $\Gamma_{u_j}$ be the corresponding generalised
Gelfand--Graev characters. Then there are
$\rho_1,\ldots,\rho_d\in\Irr_{s,\mathcal F}(H)$ such that
\begin{equation}\label{eq:type-b-isolated-pairings}
 \langle\rho_i^*,\Gamma_{u_j}\rangle_H=\delta_{ij}
 \qquad(1\leq i,j\leq d).
\end{equation}
\end{lemma}

\begin{lemma}\label{lem:symplectic-gggr-selection}
Let $m\geq3$, let $q_0$ be an odd prime power, and let
$\mathbf G_0=\Sp_{2m}\subseteq\mathbf H=\CSp_{2m}$ be defined over
$\mathbb F_{q_0}$ with the standard Frobenius endomorphism $F_0$.
Set $H=\mathbf H^{F_0}$.
Let $\mathcal C$ be a unipotent class of $\mathbf H$ stable under $F_0$.
Fix a maximal torus $\mathbf T^*$ of $\mathbf H^*$ split over
$\mathbb F_{q_0}$. For $s\in\mathbf T^*$, let $W_s$ denote the Weyl
group of $C_{\mathbf H^*}(s)$ relative to $\mathbf T^*$.
For the corresponding geometric unipotent class of $\mathbf G_0$,
let $s_0\in\mathbf T^*$ and $\mathcal F\subseteq\Irr(W_{s_0})$ be
the semisimple element and family selected in
\cite[10.1--10.3]{Taylor2014Maximal}.
Choose $z\in Z(\mathbf H^*)$ with $(zs_0)^2=1$ and set
$\widetilde s=zs_0$.
Denote by $\mathcal G_{\mathcal F}$ the finite group attached to
$\mathcal F$ in Lusztig's parametrisation, and by
$\Irr_{\widetilde s,\mathcal F}(H)\subseteq\mathcal E(H,\widetilde s)$
the corresponding family of ordinary irreducible characters.
Let $u_1,\ldots,u_d$ represent the $H$-conjugacy classes in $\mathcal C^{F_0}$,
and let $\Gamma_{u_j}$ be the corresponding generalised
Gelfand--Graev characters. Then there are
$\rho_1,\ldots,\rho_d\in\Irr_{\widetilde s,\mathcal F}(H)$ such that
\[
 \langle\rho_i^*,\Gamma_{u_j}\rangle_H=\delta_{ij}
 \qquad(1\leq i,j\leq d).
\]
\end{lemma}

\section{Symplectic groups}\label{sec:symplectic}

Let $q=p^f$ and $n\geq2$.  Throughout this section,
$S=\PSp_{2n}(q)$ denotes the simple group under consideration.  We write
$G=\Sp_{2n}(q)$ for the finite symplectic group used in the blockwise
arguments for $S$.
We assume that $(n,q)\neq(2,2)$, so that $S$ is nonabelian
simple. The purpose of this section is to prove the following result.

\begin{theorem}\label{thm:symplectic}
Let $S=\PSp_{2n}(q)$. Then, for every prime $\ell\mid |S|$, $S$ satisfies iBAW at $\ell$ on
the universal $\ell'$-covering group $X_\ell(S)$.
\end{theorem}

\subsection{Covering groups and some reductions}

The centre and Schur multiplier tables
\cite[Tables~24.2--24.3 and Remark~24.19]{MalleTesterman2011} determine the universal
$\ell'$-covering groups in Table~\ref{tab:typec-carriers}.

\begin{table}[ht]
\centering
\caption{Universal $\ell'$-covering groups in type $\mathsf C$.}
\label{tab:typec-carriers}
\begin{tabular}{@{}lll@{}}
\toprule
Parameters & Prime & $X_\ell(S)$ \\
\midrule
$q$ odd & $\ell=2$ & $\PSp_{2n}(q)$ \\
$q$ odd & $\ell$ odd & $\Sp_{2n}(q)$ \\
$q$ even, $(n,q)\neq(3,2)$ & every $\ell$ & $\Sp_{2n}(q)$ \\
$(n,q)=(3,2)$ & $\ell=2$ & $\Sp_6(2)$ \\
$(n,q)=(3,2)$ & $\ell\in\{3,5,7\}$ & $2.\Sp_6(2)$ \\
\bottomrule
\end{tabular}
\end{table}

The case $\ell=p$ of Theorem~\ref{thm:symplectic} follows from
\cite[Theorem~C]{Spath2013}.  For odd $q$ and $\ell=2$, we use the notation
$\varepsilon$, $\varepsilon_\Gamma$, $d_\Gamma$, $\mathcal F_1$,
$\mathcal F_2$, $m_\Gamma$, $w_\Gamma$, and $\mathcal P(w_\Gamma)$ from
\cite[Sections~2 and~5]{FengMalle2022}.

As noted by Feng, Yu, and Zhang \cite{FengYuZhang2024}, An's list of basic
subgroups used in the weight parametrisation \cite[p.~190]{An1993} omits
the subgroups denoted $R^0_{m,1,\gamma}$ in their notation when $m$ is even.
Their calculation of the centralisers distinguishes the two discriminants
of the multiplicity space \cite[proof of Proposition~3.48]{FengYuZhang2024}.
In the notation of \cite[p.~11]{FengMalle2022}, when
$\varepsilon_\Gamma=-\varepsilon$, $\gamma=0$, and $c_1=1$, the correction
in \cite[Remark~3.49]{FengYuZhang2024} replaces
$E_{d_\Gamma,1,1}\wr A_{c'}$ by
\[
 R^{0,-}_{d_\Gamma,1,1}\wr A_{c'},
 \qquad c'=(c_2,\ldots,c_t).
\]
The corrected subgroup has a different centraliser in the symplectic group
on the corresponding summand
\cite[proof of Proposition~3.48]{FengYuZhang2024}.  Although
\cite[Remark~3.49]{FengYuZhang2024} states that the remaining arguments
continue to apply, it does not supply
the corresponding calculation of the weight characters underlying the blockwise
weight labelling.  For principal blocks, the omitted family is excluded by
\cite[Remark~3.50]{FengYuZhang2024}. We check that no parameters indexed
by $\Gamma\in\mathcal F_1\cup\mathcal F_2$ occur and that the remaining
$R^0$ factors have multiplicity spaces of dimension one. These facts allow
us to use the principal block weight calculation in
\cite[Section~5]{FengMalle2022}. Together with the type~$\mathsf A$
cases, the Jordan reduction described in
\cite[Remark~5.8]{FengLiZhang2022Jordan} yields the result for all blocks.

Let $\mathbf Y$ be a connected reductive algebraic group with Frobenius
endomorphism $F_Y$, and set $Y=\mathbf Y^{F_Y}$. Let $\ell$ be a prime
different from the defining characteristic. Write $\mathcal E(Y,\ell')$
for the union of the rational Lusztig series $\mathcal E(Y,u)$, where
$u$ runs through the semisimple $\ell'$-elements of $Y^*$.  For a
semisimple $\ell'$-element $s\in Y^*$, let $\mathcal E_\ell(Y,s)$ denote
the set $\mathcal E_\pi(Y,(s))$ defined by Brou\'e and Michel
\cite[Section~2]{BroueMichel1989} with $\pi=\{\ell\}$, using the notation
of \cite[Section~2.3]{FengLiZhang2019}. Explicitly,
$\mathcal E_\ell(Y,s)=\bigcup_t\mathcal E(Y,st)$, where $t$ runs through
the semisimple $\ell$-elements of $Y^*$ commuting with $s$. By
\cite[Theorem~2.2]{BroueMichel1989}, this set is a union of $\ell$-blocks.
Let $e_s^Y$ be the sum of the block idempotents $e_b\in kY$ for
which $\Irr(b)\subseteq\mathcal E_\ell(Y,s)$.
We say that $b$ belongs to $e_s^Y$ if $e_be_s^Y=e_b$, and write
$\IBr(Y,e_s^Y)$ for the union of $\IBr(b)$ over these blocks.
As $s$ runs through representatives of the semisimple $\ell'$-classes
of $Y^*$, the idempotents $e_s^Y$ are pairwise orthogonal and sum to $1$.

We shall use the following form of the Jordan reduction of Feng, Li, and
Zhang.

\begin{lemma}\label{lem:jordan-reduction}
Let $\mathbf G$ be a simple, simply connected algebraic group in
characteristic $p$, let $F$ be a Steinberg endomorphism of $\mathbf G$,
and set $G=\mathbf G^F$.  Assume that $G/Z(G)$ is simple and that $G$
is its universal covering group.  Let $\ell\neq p$ be a prime.
Fix a regular embedding $\mathbf G\hookrightarrow\widetilde{\mathbf G}$
as in \cite[Section~5]{FengLiZhang2022Jordan}, and set
$\widetilde G=\widetilde{\mathbf G}^F$.
Suppose that the following conditions hold.
\begin{enumerate}[label=\textup{(\roman*)}]
\item For every semisimple $\ell'$-element $s\in(\mathbf G^*)^F$, there
is a choice of the automorphism group $A_s$ used for $s$ in
\cite[Section~5]{FengLiZhang2022Jordan} such that every
$\widetilde G$-orbit in $\IBr(G)$ contains a character $\psi$ satisfying
\[
 (\widetilde G A_s)_\psi=\widetilde G_\psi(A_s)_\psi
\]
and extending to $G(A_s)_\psi$.
\item Let $\mathbf K$ be a simple, simply connected algebraic group
in characteristic $p$ whose Dynkin diagram is isomorphic to a
subdiagram of that of $\mathbf G$.  For every Steinberg endomorphism
$F'$ of $\mathbf K$ such that $\mathbf K^{F'}/Z(\mathbf K^{F'})$ is
simple, every strictly quasi-isolated $\ell$-block of $\mathbf K^{F'}$,
in the sense of \cite[Section~4.3]{FengLiZhang2022Jordan}, admits
an iBAW bijection.
\end{enumerate}
Then every $\ell$-block of $G$ is BAW-good.
\end{lemma}

\begin{proof}
In \cite[Section~5]{FengLiZhang2022Jordan}, the group $A$ is chosen after
fixing the semisimple element $s$. Condition~\textup{(i)} supplies
\cite[Assumption~5.3]{FengLiZhang2022Jordan} for this choice $A_s$, for
each $s$, as required by \cite[Hypothesis~5.5(a)]{FengLiZhang2022Jordan}.
Condition~\textup{(ii)} is \cite[Hypothesis~5.5(b)]{FengLiZhang2022Jordan}.
The conclusion follows from \cite[Theorem~5.7]{FengLiZhang2022Jordan}.
\end{proof}

To apply condition~\textup{(ii)} of Lemma~\ref{lem:jordan-reduction},
we also need the following consequence of the type~$\mathsf A$ results
\cite{FengLiZhang2023LowRankA,FengLiZhang2023Morita}.
Let $q_0$ be a prime power, let $d\geq2$, and let
$H=\operatorname{SL}_d(q_0)$ or $\operatorname{SU}_d(q_0)$.
If $S_0=H/Z(H)$ is nonabelian simple and $\ell$ is a prime not
dividing $q_0$, then every $\ell$-block of $H$ admits an iBAW bijection.
If $H$ is the universal covering group of $S_0$, its blocks are
BAW-good by the proof of \cite[Theorem~1]{FengLiZhang2023Morita},
and the passage from BAW-goodness to iBAW bijections in
Section~\ref{sec:preliminaries} gives the assertion.

The remaining possibilities for $H$, up to isomorphism, are
$\operatorname{SL}_2(4)$, $\operatorname{SL}_2(9)$,
$\operatorname{SL}_3(2)$, $\operatorname{SL}_3(4)$,
$\operatorname{SL}_4(2)$, $\operatorname{SU}_4(2)$,
$\operatorname{SU}_4(3)$ and $\operatorname{SU}_6(2)$
\cite[Remark~24.19 and Table~24.3]{MalleTesterman2011}.
Each corresponding simple group $S_0$ satisfies iBAW at $\ell$ by
\cite[Proposition~4.6]{FengLiZhang2023LowRankA}.
When $Z(H)=1$, Corollary~\ref{cor:fixed-point-descent} gives the
assertion directly.

It remains to consider $H=\operatorname{SL}_2(9)$,
$\operatorname{SU}_4(3)$, $\operatorname{SL}_3(4)$ or
$\operatorname{SU}_6(2)$. Suppose first that $\ell\mid|Z(H)|$.
In these cases, $P=Z(H)=O_\ell(H)$ and $H/P=S_0$.
Since $H$ is perfect, the natural map
$\Aut(H)\longrightarrow\Aut(S_0)$ is injective.
For an $\ell$-block $b$ of $H$,
Corollary~\ref{cor:fixed-point-descent} gives an iBAW bijection for
the block of $S_0$ dominated by $b$.
By \cite[Lemma~3.6(i)]{MartinezRizoRossi2026}, restricting the
automorphism group makes this bijection compatible with
$S_0\rtimes\Aut(H)_b$.
Lemma~\ref{lem:type-b-normal-core}, applied with
$A=H\rtimes\Aut(H)_b$, therefore gives an iBAW bijection for $b$.
For every other prime $\ell\nmid q_0$, the Sylow $\ell$-subgroups
of these four groups are cyclic, so
\cite[Proposition~8.5]{MartinezRizoRossi2026} gives the assertion.

We give an alternative proof of the following result of Feng and Malle
\cite[Theorem~1]{FengMalle2022}.  Our proof uses their principal block
calculation and Jordan reduction, without relying on the weight
parametrisation for arbitrary $2$-blocks affected by the classification
errors discussed above.

\begin{proposition}\label{prop:odd-two}
Let $q$ be odd and $n\geq2$.  Then
$S=\PSp_{2n}(q)$ satisfies iBAW at $2$ on $X_2(S)=S$.
\end{proposition}

\begin{proof}
By \cite[Theorem~1.1]{LiLi2019}, $\PSp_4(q_0)$ satisfies iBAW
at $2$ for every odd prime power $q_0$. This proves the assertion
when $n=2$. For later use in condition~\textup{(ii)} of
Lemma~\ref{lem:jordan-reduction}, we also need an iBAW bijection
for the principal $2$-block of $\Sp_4(q_0)$.
Let $K=\Sp_4(q_0)$ and $\overline K=K/Z(K)$.
Corollary~\ref{cor:fixed-point-descent}, applied to $\overline K$,
gives an iBAW bijection for its principal $2$-block.
Since $O_2(K)=Z(K)$ and $\Aut(K)\cong\Aut(\overline K)$,
Lemma~\ref{lem:type-b-normal-core}, applied with $G=K$,
$P=Z(K)$ and $A=K\rtimes\Aut(K)$, lifts this bijection to the
principal $2$-block of $K$.

\enlargethispage{2pt}
Assume henceforth that $n\geq3$.
We now establish the principal block result in higher rank.
Let $G_0=\Sp_{2m}(q_0)$, where $q_0$ is an odd prime power and
$m\geq3$, and let $B_1$ be its principal $2$-block.
In the notation of \cite[Section~5]{FengMalle2022}, this is the
block labelled by $s=1$.
Let $(R,\varphi)$ be a weight belonging to $B_1$.
By \cite[Lemma~2.3]{FengYuZhang2024},
$Z(R)\in\Syl_2(C_{G_0}(R))$.
By \cite[Theorem~3.37]{FengYuZhang2024}, choose a decomposition of
$R$ as a direct product of basic subgroups in the classification used there, each acting on its
corresponding summand in an orthogonal decomposition of the natural
module.
The omitted family does not occur in a principal weight
\cite[Remark~3.50]{FengYuZhang2024}. Indeed, suppose that a factor from this
family occurs in the decomposition.
By \cite[Proposition~3.48(4)]{FengYuZhang2024}, it has the form
$R^{0,-}_{d,1,1,\mathbf c}$ with $d$ even.  Let
\[
 R_0=I_W\otimes E_-^3\leq\Sp(W\otimes U),
\]
where $W$ is a nondegenerate orthogonal space over $\F_{q_0}$ of even
dimension $d\geq2$ and nonsquare discriminant, and
$E_-^3\cong Q_8$ acts on its symplectic module $U$.
The group $E_-^3$ acts trivially on $W$.  Thus, as an $E_-^3$-module,
$W\otimes U$ is a direct sum of $d$ copies of $U$, with multiplicity space $W$.
The omitted factor is a wreath product whose base group is a direct
product of copies of $R_0$. Choose a nonscalar orthogonal reflection
$t$ of $W$, acting as $-1$ on a nonsingular line and as $1$ on its
orthogonal complement. On the summand corresponding to this factor, set
$\widehat t=\operatorname{diag}(t\otimes I_U,\ldots,t\otimes I_U)$.  This is a symplectic
involution which centralises each copy of $R_0$ and the coordinate
permutations.  Extending it by the identity on the other summands gives
  $\widehat t\in C_{G_0}(R)$.  It does not lie in $R$: its permutation part is
trivial, while its action on the multiplicity space $W$ is nonscalar.
  Thus $\langle Z(R),\widehat t\rangle$ is a $2$-subgroup of $C_{G_0}(R)$
properly containing $Z(R)$, a contradiction. Together with the other
exclusions in \cite[Remark~3.50]{FengYuZhang2024}, this leaves cases~(3) and~(6) of
\cite[Proposition~3.48]{FengYuZhang2024}.
Up to conjugacy, the basic subgroups in these cases are those attached to
the elementary divisor $\Gamma=x-1$ in \cite[Section~5.1]{FengMalle2022}.
That section also describes the irreducible characters of $2$-defect zero
of their normaliser quotients, with normalisers taken in the symplectic
groups on the corresponding summands.  To specialise this parametrisation to $B_1$, recall that its
semisimple label is $s=1$.  Consequently, $m_\Gamma(1)=w_\Gamma(1)=0$ for
every elementary divisor $\Gamma\in\mathcal F_1\cup\mathcal F_2$.  Thus the product
of the sets $\mathcal P(w_\Gamma)$ over these elementary divisors in
\cite[Proposition~5.4]{FengMalle2022} is a singleton, and only the
parameters associated with the elementary divisor $\Gamma=x-1$ occur.  In particular, the cases involving the corrected
$E_{d_\Gamma,1,1}$ factor do not enter the calculation for the principal
block.  By \cite[Proposition~3.48]{FengYuZhang2024},
every remaining factor of type $R^0$ has multiplicity space of dimension
one. Its isometry group is therefore $\{I,-I\}$, as in the centraliser
calculation in \cite[proof of Proposition~3.48]{FengYuZhang2024}.
Equal factors may still occur repeatedly in the direct product
decomposition of $R$. For these repeated factors, we use
the construction of characters of wreath products in \cite[(6D)]{An1993},
recalled after \cite[Proposition~5.4]{FengMalle2022}.  This construction assigns to
each irreducible character of $2$-defect zero of the corresponding normaliser
quotient a partition of the form $(r,r-1,\ldots,1)$ for some $r\geq 0$,
with the empty partition understood when $r=0$.  Its size is
$r(r+1)/2$, a triangular number.  The sizes of these partitions
sum to the number of copies of the corresponding basic subgroup.
This sum need not be triangular, as the example in
Remark~\ref{rem:triangular-multiplicity} shows.

Let $(R,\varphi)$ be a pair produced by this construction, with
$\varphi\in\dz(N_{G_0}(R)/R)$.
In each of the two remaining families, a basic subgroup, including its
wreath extension, has centraliser $\{I,-I\}$ on its corresponding summand
\cite[p.~198]{An1993}.
The central elements acting as $-I$ on one summand and as $I$ on all the
others force every element of $C_{G_0}(R)$ to preserve each summand.  Thus $C_{G_0}(R)=Z(R)$, so $Z(R)$ is a Sylow $2$-subgroup
of $C_{G_0}(R)$ and $RC_{G_0}(R)=R$.
The inflation of $\varphi$ to $N_{G_0}(R)$ has $R$ in its kernel
and has defect zero on $N_{G_0}(R)/(RC_{G_0}(R))$.
Hence \cite[Lemma~2.3]{FengYuZhang2024} shows that
$(R,\varphi)$ is a weight belonging to $B_1$.
The construction following \cite[Proposition~5.4]{FengMalle2022}
recovers from each $x-1$ parameter the conjugacy types and
multiplicities of the basic factors and the character $\varphi$.
The $G_0$-conjugacy class of $(R,\varphi)$ also determines
its parameter. Thus this construction identifies $\Alp(B_1)$
with the $x-1$ parameter set.

\enlargethispage{3pt}For the principal block Brauer character calculation, we use
\cite[Corollary~4.3]{FengMalle2022}.
The required unitriangularity is stated in
\cite[Theorem~2.9]{Chaneb2021} for every odd $q_0$.
Its proof uses the character selection in
\cite[Proposition~2.8]{Chaneb2021}, based on
\cite[Proposition~4.3]{GeckHezard2008} and
\cite[Proposition~5.4]{Taylor2013}.
We verify this character selection for the families used in the automorphism
argument in \cite[Section~4.1]{FengMalle2022}.
Let $\mathbf G_0=\Sp_{2m}\subseteq\mathbf H=\CSp_{2m}$,
with the split Frobenius endomorphism $F_0$ defining the given
group $G_0$, and set $H=\mathbf H^{F_0}$.
The partition parametrisation shows that every geometric unipotent
class of $\mathbf H$ is stable under $F_0$.
For each such class $\mathcal C$,
Lemma~\ref{lem:symplectic-gggr-selection} supplies the characters
of $H$ with the required Kronecker delta scalar products.
For unipotent elements $u\in\mathbf G_0$, the component groups
$A_{\mathbf G_0}(u)$ are elementary abelian $2$-groups \cite[proof of Proposition~2.4]{Taylor2013}.
For each geometric unipotent class, choose a representative as in
\cite[Definition~2.1]{Taylor2013}, which is possible by
Proposition~2.8 of the same paper.
For the split pair $(\widetilde s,\mathcal F)$ selected in
Lemma~\ref{lem:symplectic-gggr-selection}, the symbols in any type
$\mathsf D$ factor of the selected family are nondegenerate. Hence
the argument in
\cite[10.3]{Taylor2014Maximal} gives condition~\textup{(P5)} of
\cite[Theorem~3.2]{Taylor2013}.

We follow the restriction argument in the proof of
\cite[Proposition~5.4]{Taylor2013}, using
Lemma~\ref{lem:symplectic-gggr-selection} for the selection
supplied there by \cite[Proposition~4.3]{GeckHezard2008}.
The induction identity for generalised Gelfand--Graev characters
\cite[proof of Lemma~14.12]{Taylor2016}, compatibility of Alvis--Curtis duality
with restriction \cite[Proposition~3.4.3]{GeckMalle2020}, Frobenius reciprocity
and condition~\textup{(P5)} of \cite[Theorem~3.2]{Taylor2013}
give the required scalar products on restriction to $G_0$.
The class count and Clifford theory argument in the proof of
\cite[Proposition~5.4]{Taylor2013} match the irreducible constituents
with the $G_0$-classes contained in each $H$-class.

The selected characters have unipotent support $\mathcal C$ by the
proof of Lemma~\ref{lem:symplectic-gggr-selection} in
Appendix~\ref{app:gggr-proofs}, and so do their restriction constituents \cite[Example~2.7.18]{GeckMalle2020}.
These constituents lie in $\mathcal E(G_0,s)$, where $s$ is the
image of $\widetilde s$ in $(\mathbf G_0^*)^{F_0}$, by
\cite[Proposition~2.6.16]{GeckMalle2020}.
Since $s^2=1$ and Alvis--Curtis duality preserves rational Lusztig
series \cite[Proposition~9.8(iv)]{CabanesEnguehard2004}, these
constituents and their duals belong to $B_1$ by
\cite[Theorem~21.14]{CabanesEnguehard2004}.
Together with the vanishing results for wave front sets in
\cite[Lemma~14.15 and Proposition~15.2]{Taylor2016}, this gives
the identity diagonal blocks and the vanishing entries used in
\cite[Theorem~2.9]{Chaneb2021}.
Thus, for $m\geq3$ and every odd $q_0$, this establishes the
character selection and unitriangularity used in the principal
block argument of \cite[Section~4.1]{FengMalle2022}.

For the principal $x-1$ parameters obtained above,
\cite[Lemma~5.1]{FengMalle2022} gives the field and diagonal actions on
the characters of the normaliser quotients of the basic subgroups. Applying these actions
to the wreath product construction, as in the proof of
\cite[Proposition~5.5]{FengMalle2022}, gives the weight actions described
in \cite[Proposition~5.9]{FengMalle2022}.
Together with \cite[Corollary~4.3]{FengMalle2022}, the counting arguments in
\cite[Proposition~6.1 and the proof of Theorem~6.2]{FengMalle2022} show that
$\IBr(B_1)$ and $\Alp(B_1)$ have equal cardinalities and equal numbers of
elements fixed by the nontrivial diagonal outer automorphism, while field
automorphisms fix both sets pointwise.  Matching the diagonal orbits of
sizes one and two gives an $\Aut(G_0)$-equivariant bijection from
$\IBr(B_1)$ to $\Alp(B_1)$.
Set $Z_0=Z(G_0)$ and $\overline G_0=G_0/Z_0$.
The correspondences in the proof of Lemma~\ref{lem:type-b-normal-core},
with $P=Z_0$, give a corresponding
$\Aut(\overline G_0)$-equivariant bijection for the principal
$2$-block of $\overline G_0$.
For $B_1$, the equivariant bijection gives condition~(1) of
\cite[Proposition~3.4]{FengMalle2022}, and the trivial action of field
automorphisms gives conditions~(2) and~(3).  The blockwise argument in the
proof of that proposition therefore verifies the inductive BAW condition
for the principal block of $\overline G_0$.
Apply Lemma~\ref{lem:type-b-normal-core} with
$G=G_0$, $P=Z_0=O_2(G_0)$ and $A=G_0\rtimes\Aut(G_0)$.
The required iBAW bijection for the principal block of $\overline G_0$ follows
from Theorem~\ref{thm:spath-triples}, using the natural isomorphism
$\Aut(G_0)\cong\Aut(\overline G_0)$.
Thus $B_1$ admits the required iBAW bijection.

Let $\mathbf G$ be a simply connected algebraic group of type $\mathsf C_n$
with Frobenius endomorphism $F$ such that
$\mathbf G^F=G=\Sp_{2n}(q)$.  We verify the two assumptions of
Lemma~\ref{lem:jordan-reduction} at $\ell=2$.
For each semisimple $2'$-element $s\in(\mathbf G^*)^F$, conjugate the Levi
subgroup associated with $s$ and its series idempotent by a suitable element
of $G$. By \cite[Corollary~4.2 and Lemmas~4.3 and~4.5]{Ruhstorfer2022Derived},
the group $A_s$ in condition~\textup{(i)} of Lemma~\ref{lem:jordan-reduction}
can then be chosen inside the field automorphism group. By \cite[Corollary~4.6]{FengMalle2022}, if a diagonal
automorphism and a field automorphism have the same action on a Brauer
character, then both fix it. Thus, for $\widetilde G=\CSp_{2n}(q)$,
\[
 (\widetilde G A_s)_\psi=\widetilde G_\psi(A_s)_\psi
 \qquad(\psi\in\IBr(G)).
\]
Since $(A_s)_\psi$ is cyclic, Lemma~\ref{lem:cyclic-extension} gives
the required extension.
A connected Dynkin subdiagram of
$\mathsf C_n$ has type $\mathsf A_r$ or $\mathsf C_r$.
Let $\mathbf K$ and $F'$ be as in condition~\textup{(ii)} of
Lemma~\ref{lem:jordan-reduction}. If $\mathbf K$ has type $\mathsf A_r$,
then $\mathbf K^{F'}$ is a special linear or unitary group.
The result for special linear and unitary groups established after
Lemma~\ref{lem:jordan-reduction} gives the required iBAW bijections
for its $2$-blocks.
This also covers type~$\mathsf C_1$, since $\mathsf C_1=\mathsf A_1$.

Suppose now that $\mathbf K$ has type $\mathsf C_r$, where $r\geq2$,
and let $c$ be a strictly quasi-isolated $2$-block of $\mathbf K^{F'}$.  Choose a strictly
quasi-isolated semisimple $2'$-element $s\in\mathbf K^{*F'}$ such that
$\Irr(c)\subseteq\mathcal E_2(\mathbf K^{F'},s)$, where
$\mathbf K^*\cong\operatorname{SO}_{2r+1}$.
Since $s$ is quasi-isolated and the defining characteristic is odd,
\cite[Proposition~4.11(a)]{Bonnafe2005} gives $s^2=1$.  The element $s$
has odd order, so $s=1$.  Thus $c$ is the principal $2$-block, the unique
unipotent $2$-block of $\mathbf K^{F'}$
\cite[Theorem~21.14]{CabanesEnguehard2004}.  We have $\mathbf K^{F'}\cong\Sp_{2r}(q_0)$ for an odd prime
power $q_0$. If $r=2$, the rank two construction at the start
of this proof gives an iBAW bijection for $c$.
If $r\geq3$, the principal block result proved above, applied
with $m=r$, gives such a bijection.
Hence both hypotheses of Lemma~\ref{lem:jordan-reduction} hold.

Since $G$ is the full universal covering group of $S$,
Lemma~\ref{lem:jordan-reduction} shows that every $2$-block of $G$ is
BAW-good. Applying Remark~\ref{rem:block-aggregation}\textup{(b)} to the blockwise
equivariant bijections gives an $\Aut(G)$-equivariant
bijection $\IBr(G)\longrightarrow\Alp(G)$ that preserves blocks.
This verifies condition~\textup{(1)} of Feng and Malle's criterion
\cite[Proposition~3.4]{FengMalle2022}.  Condition~\textup{(2)} follows
from \cite[Corollary~4.6]{FengMalle2022}, and
\cite[Remark~3.5]{FengMalle2022} then gives condition~\textup{(3)}.
The criterion therefore shows that $S$ satisfies iBAW at $2$ on
$X_2(S)=S$.
\end{proof}

\begin{remark}\label{rem:triangular-multiplicity}
The final sentence of \cite[Remark~3.50]{FengYuZhang2024} extends the
triangularity assertion of \cite[Remark~3.47(2)]{FengYuZhang2024} to
symplectic groups. The following principal $2$-weight of $\Sp_4(3)$ shows that
this extension fails for the total multiplicity of a basic subgroup, although
the individual partitions in its parametrisation have triangular sizes.
Let $V=V_1\perp V_2$ be an orthogonal direct sum of two symplectic
planes over $\F_3$, and let $R=Q_8\times Q_8$ act with one factor
on each plane. Each factor is a basic subgroup in the construction
of \cite[Section~3]{FengYuZhang2024}.
Set $G=\Sp(V)\cong\Sp_4(3)$.
The summands are nonisomorphic absolutely irreducible $R$-modules,
so $C_G(R)=\{\pm I_{V_1}\}\times\{\pm I_{V_2}\}=Z(R)$,
and every element of $N_G(R)$ permutes the two summands.
Since $Q_8\trianglelefteq\operatorname{SL}_2(3)=\Sp(V_i)$, we have
$N_G(R)=(\operatorname{SL}_2(3)\times\operatorname{SL}_2(3))\rtimes C_2$,
where the involution interchanges the factors.
Hence $N_G(R)/R\cong C_3\times S_3$, which has no nontrivial
normal $2$-subgroup, so $R$ is radical.
This quotient has irreducible characters of degree $2$ and
$2$-defect zero.  By \cite[Lemma~2.3]{FengYuZhang2024}, these give
principal $2$-weights.
Here the same basic subgroup occurs twice, so its total multiplicity
is $2$, which is not triangular. In the labelling by assignments of
$2$-cores recalled from \cite[(6D)]{An1993} after
\cite[Proposition~5.4]{FengMalle2022}, the three linear characters of
$N_{\Sp(V_i)}(Q_8)/Q_8\cong C_3$ are assigned the partitions
$(1),(1),\varnothing$, up to permutation.  Their sizes are $1,1,0$, which are individually
triangular, while their sum is $2$. The proof of
Proposition~\ref{prop:odd-two} uses the individual partition sizes and their
sum, and does not require that sum to be triangular.
\end{remark}

\begin{proposition}\label{prop:rank-two}
Let $q>2$ be a prime power.  Then $S=\PSp_4(q)$ satisfies iBAW at every
prime dividing its order.
\end{proposition}

\begin{proof}
Write $q=p^f$, and let $\ell$ be a prime dividing $|S|$.
If $\ell=p$, the result follows from \cite[Theorem~C]{Spath2013}.
If $q$ is odd, the result for every $\ell\neq p$ is proved in
\cite[Theorem~1.1]{LiLi2019}, with the cases $\ell\mid q^2-1$
also proved by Brough and Schaeffer Fry
\cite[Theorem~1.1]{BroughSchaefferFry2020}.
If $\ell\neq p$ and $q$ is even, it follows from
\cite[Theorem~5.5]{SchaefferFry2014}.
\end{proof}

It remains to treat rank at least three at odd nondefining primes.
For even $q$ and $n\geq4$, we give an independent proof. For odd $q$, we replace the
uses of unitriangularity in Li's proof of \cite[Theorem~2]{Li2021}.
Neither argument assumes unitriangularity.

\subsection{Even fields}

Let $\mathbf H$ be a simply connected algebraic group of type $\mathsf C_r$,
where $r\geq4$, over $\overline{\F}_2$, and let $F_2$ be the Frobenius
endomorphism induced by the squaring map on $\overline{\F}_2$.  For a
positive integer $a$, set $F'=F_2^a$ and
$H=\mathbf H^{F'}=\Sp_{2r}(2^a)$.  The map $F_2$ restricts to a field
automorphism of $H$ of order $a$.  Let $\mathbf H^*$ be the dual group of
$\mathbf H$.  Then $\mathbf H^*$ is adjoint of type $\mathsf B_r$, and the
Frobenius endomorphism dual to $F'$ is also denoted by $F'$.

Let $\ell$ be an odd prime dividing $|H|$, let $e$ be the
multiplicative order of $2^a$ modulo $\ell$, and let
$E_H=\langle F_2|_H\rangle\cong C_a$.  For an $\ell$-block $C$ of $H$, let
$\mathcal X_C=\Irr(C)\cap\mathcal E(H,\ell')$.  We follow
\cite[Definition~3.5]{FengMalleZhang2026} in writing $e$-JGC for
$(e,\ell)$-\emph{Jordan-generalised-cuspidal} when $\ell$ is clear.
For an $F'$-stable torus $\mathbf T$, we use the notation
$\mathbf T_{\phi_e}$ of \cite[Section~3.2]{FengMalleZhang2026} for its
Sylow $e$-torus.  In the notation of
\cite[Definitions~3.17--3.18]{FengMalleZhang2026}, $\mathcal W(C)$ denotes
the set of $H$-conjugacy classes of generic $(e,\ell)$-weights belonging
to $C$.
For an $F'$-stable Levi subgroup $\mathbf L$ and
$\lambda\in\Irr(\mathbf L^{F'})$, write
$W_H(\mathbf L,\lambda)=N_H(\mathbf L,\lambda)/\mathbf L^{F'}$.
The prime $\ell$ is good for $\mathbf H$, and
$Z(\mathbf H)=Z(\mathbf H^*)=1$, so
\cite[Condition~3.23]{FengMalleZhang2026} holds.
In particular, if $(\mathbf T,\eta)$ represents an element of
$\mathcal W(C)$, then $\bl(\eta)^H=C$ by
\cite[Proposition~3.24(b)]{FengMalleZhang2026}.

\begin{lemma}\label{lem:even-field-fixed}
If $C$ is a unipotent $\ell$-block of $H=\Sp_{2r}(2^a)$, then $E_H$
fixes every element of $\mathcal X_C$ and $\mathcal W(C)$.
\end{lemma}

\begin{proof}
By \cite[Theorem~4.4(iii)]{CabanesEnguehard1994}, the characters in
$\Irr(C)$ lie in rational Lusztig series labelled by semisimple
$\ell$-elements.  The disjointness of these series therefore gives
$\mathcal X_C=\Irr(C)\cap\mathcal E(H,1)$.
By \cite[Theorem~4.4(i)]{CabanesEnguehard1994}, this is the generalised
$e$-Harish-Chandra series belonging to $C$.
Since field automorphisms fix the unipotent characters of $H$
\cite[p.~700]{CabanesSpath2013}, the group $E_H$
fixes $\mathcal X_C$ pointwise.

Let $(\mathbf T,\eta)$ represent an element of $\mathcal W(C)$
and set $\mathbf M=C_{\mathbf H}(\mathbf T)$.
By the definition of a generic weight,
$\mathbf T=Z^\circ(\mathbf M)_{\phi_e}$ and there is an
$e$-JGC character $\lambda\in\mathcal E(\mathbf M^{F'},\ell')$
such that $\eta\in\Irr\bigl(N_H(\mathbf T)\mid\lambda\bigr)$
and some $\chi\in\Irr(C)$ occurs in the Lusztig induction of
$\lambda$ to $H$.
Since $\mathbf M=C_{\mathbf H}(\mathbf T)$ and
$\mathbf T=Z^\circ(\mathbf M)_{\phi_e}$, we have
$N_H(\mathbf T)=N_H(\mathbf M)$.
By compatibility of Lusztig induction with rational Lusztig series
\cite[Proposition~15.7]{CabanesEnguehard2004}, every constituent of the
Lusztig induction of $\lambda$ lies in a series labelled by an
$\ell'$-element. Hence $\chi\in\mathcal X_C$.
Since $\mathcal X_C\subseteq\mathcal E(H,1)$, the same compatibility
forces the semisimple label of $\lambda$ to be trivial, so $\lambda$
is unipotent.

  Let $\sigma\in E_H$ be arbitrary. Choose $0\leq j<a$ such that
  $\sigma=F_2^j|_H$, and use the same symbol $\sigma$ for the
  endomorphism $F_2^j$ of $\mathbf H$. Let
  $\mathbf T_0\leq\mathbf B_0$ be an $F_2$-stable maximal torus and Borel
  subgroup, and let $\Pi$ be the corresponding set of simple roots.
  Let $W=N_{\mathbf H}(\mathbf T_0)/\mathbf T_0$ be the Weyl group.
  For~$I\subseteq\Pi$, let $\mathbf L_I$ be the corresponding standard
  Levi subgroup.  By the rational parametrisation of $F'$-stable Levi
  subgroups in \cite[Section~2.3]{CabanesSpath2013}, choose a pair
  $(I,w)$ representing the $H$-conjugacy class of $\mathbf M$, where
  $w\in W$ satisfies $w(I)=I$.  Choose a representative
  $\dot w\in N_{\mathbf H}(\mathbf T_0)$ of $w$ such that
  $F_2(\dot w)=\dot w$.  By Lang's theorem, there is $g\in\mathbf H$ such that
  $g^{-1}F'(g)=\dot w$.
  After replacing $(\mathbf T,\eta)$ by an $H$-conjugate and $\lambda$ by
  the corresponding conjugate, we may assume $\mathbf M=g\mathbf L_Ig^{-1}$.
  Since $\mathbf L_I$ and $\dot w$ are
  defined over $\F_2$,
  $\sigma(\mathbf L_I)=\mathbf L_I$ and
  $\sigma(\dot w)=\dot w$.  Define $h_\sigma=g\sigma(g)^{-1}$ and set
  $\tau=\operatorname{Int}(h_\sigma)\circ\sigma$, where
  $\operatorname{Int}(x)$ denotes conjugation by $x$.
  Since $F'(g)=g\dot w$ and $F'$ commutes with $\sigma$, we have
  \[
       F'(h_\sigma)
       =g\dot w\,\sigma(g\dot w)^{-1}
       =g\sigma(g)^{-1}=h_\sigma.
  \]
  Hence $h_\sigma\in H$, and $\tau$ commutes with $F'$.  Since $w$
  stabilises $I$, the endomorphism
  $F_w=\operatorname{Int}(\dot w)\circ F'$ restricts to a Frobenius
  endomorphism of $\mathbf L_I$.  Moreover, $\sigma$ commutes with $F_w$.
  Conjugation by $g$
  maps $\mathbf L_I^{F_w}$ onto $\mathbf M^{F'}$ and intertwines $\sigma$
  with $\tau$:
  \begin{equation}\label{eq:even-field-intertwining}
       \tau(gxg^{-1})=g\sigma(x)g^{-1}\qquad(x\in\mathbf H).
  \end{equation}
Equation~\eqref{eq:even-field-intertwining} shows that $\tau$ stabilises $\mathbf M$.
Since $\tau$ commutes with $F'$, it also stabilises
$\mathbf T=Z^\circ(\mathbf M)_{\phi_e}$.
Define
\[
  \lambda_I(x)=\lambda(gxg^{-1})
  \qquad(x\in\mathbf L_I^{F_w}).
\]
Then $\lambda_I$ is a unipotent character of $\mathbf L_I^{F_w}$.
By \eqref{eq:even-field-intertwining}, it suffices to prove
$\lambda_I^\sigma=\lambda_I$ to obtain $\lambda^\tau=\lambda$.
The map $\sigma$ preserves each simple component of
$[\mathbf L_I,\mathbf L_I]$ and commutes with $F_w$, so it preserves
each $F_w$-orbit of components.
Projection onto one component in each orbit identifies the corresponding
finite factor with the fixed points of a power of $F_w$.
Since $\sigma=F_2^j$ preserves each component, it commutes with this
projection. Thus its action on each finite factor is induced by
$F_2^j$ and is a field automorphism.
These factors have type $\mathsf A$ or $\mathsf C$, and their unipotent
characters are fixed by field automorphisms
\cite[p.~700]{CabanesSpath2013}.
The standard parametrisation of unipotent characters under central isogenies
and direct products identifies the unipotent characters of $\mathbf L_I^{F_w}$
with tuples of unipotent characters of these finite factors
\cite[Remark~4.2.1]{GeckMalle2020}.
This identification commutes with $\sigma$, and the torus
contributes only its trivial character.
Hence $\lambda_I^\sigma=\lambda_I$, and
\eqref{eq:even-field-intertwining} gives $\lambda^\tau=\lambda$.

  We next show that $\tau$ acts trivially on
  $N_H(\mathbf M)/\mathbf M^{F'}$.
  Let $W_I=N_{\mathbf L_I}(\mathbf T_0)/\mathbf T_0$.  The quotient
  $N_{\mathbf H}(\mathbf L_I)/\mathbf L_I$ identifies with
  $N_W(W_I)/W_I$.  Every element of $W$ has a representative fixed by
  $F_2$, so $\sigma=F_2^j$ acts trivially on $W$ and on this quotient.
  For $y\in N_{\mathbf H}(\mathbf L_I)^{F_w}$, the element
  $y^{-1}\sigma(y)$ belongs to $\mathbf L_I$ and is fixed by $F_w$,
  since $\sigma$ commutes with $F_w$.  Thus $\sigma$ acts trivially on
  $N_{\mathbf H}(\mathbf L_I)^{F_w}/\mathbf L_I^{F_w}$.
  Conjugation by $g$ identifies this quotient with
  $N_H(\mathbf M)/\mathbf M^{F'}$. By \eqref{eq:even-field-intertwining},
  $\tau$ acts trivially on $N_H(\mathbf M)/\mathbf M^{F'}$.

  Set $N=N_H(\mathbf T)=N_H(\mathbf M)$ and
  $N_\lambda=N_H(\mathbf M,\lambda)$. To obtain a $\tau$-invariant
  extension of $\lambda$ to $N_\lambda$, embed $\mathbf H$ in
  $\mathbf K=\operatorname{PGL}_{2r+1}$ by
  $A\mapsto\pi(\operatorname{diag}(A,1))$, where
  $\pi\colon\operatorname{GL}_{2r+1}\to\operatorname{PGL}_{2r+1}$ is the
  quotient map.  This embedding is defined over $\F_2$ and commutes with
  $F'$ and the standard field endomorphisms.  We identify $\mathbf H$, and hence $\mathbf M$, with
  their images in $\mathbf K$.  The group $\mathbf K$ is simple of adjoint type, and
  $\mathbf M$ is a closed connected $F'$-stable reductive subgroup of
  $\mathbf K$.  Let $E(\mathbf K^{F'})$ denote the subgroup of
  $\Aut(\mathbf K^{F'})$ generated by the standard field and graph
  automorphisms.  The element
  $h_\sigma\sigma\in\mathbf K^{F'}E(\mathbf K^{F'})$ induces $\tau$ on $H$
  by conjugation.  We use the same symbol $\tau$ for this element.  Define
  $\widehat N_\lambda=
  \bigl(\mathbf K^{F'}E(\mathbf K^{F'})\bigr)_{\mathbf M,\lambda}$,
  where the subscript denotes the stabiliser of the pair
  $(\mathbf M,\lambda)$.  Since $\lambda$ is unipotent, Sp\"ath's theorem
  \cite[Theorem~3.3]{Spath2025Extensions} shows that $\lambda$ extends to
  $\widehat N_\lambda$.  Both $N_\lambda$ and $\tau$ lie in
  $\widehat N_\lambda$.  Choose
  $\widetilde\lambda\in\Irr(\widehat N_\lambda)$ extending $\lambda$.
  Since $\tau$ normalises $N$ and $N_\lambda$, the restriction
  $\widehat\lambda=\operatorname{Res}_{N_\lambda}^{\widehat N_\lambda}
  \widetilde\lambda$ is a $\tau$-invariant extension of $\lambda$.

  Since $\eta\in\Irr(N\mid\lambda)$ and $N_\lambda$ is the inertia group
  of $\lambda$ in $N$, Clifford theory and Gallagher's theorem give
  \[
    \eta=\operatorname{Ind}_{N_\lambda}^N(\widehat\lambda\xi)
    \qquad\text{for some }\xi\in\Irr(N_\lambda/\mathbf M^{F'}),
  \]
  where $\xi$ is inflated to $N_\lambda$.
  The automorphism $\tau$ fixes $\widehat\lambda$ and acts trivially on
  $N_\lambda/\mathbf M^{F'}$, so it fixes $\xi$ and hence $\eta$.
  Thus $\tau$ fixes $(\mathbf T,\eta)$.
  Because $h_\sigma\in H$ and
  $\tau=\operatorname{Int}(h_\sigma)\circ\sigma$, the automorphisms
  $\tau$ and $\sigma$ induce the same action on $H$-conjugacy classes of
  generic weights.  Hence $\sigma$ fixes the $H$-conjugacy class of
  $(\mathbf T,\eta)$.  This proves that $E_H$ fixes $\mathcal W(C)$
  pointwise.
\end{proof}

The following lemma combines the unipotent character parametrisation with
Feng--Malle--Zhang's correspondence for relative Weyl groups
\cite{FengMalleZhang2026} and Lemma~\ref{lem:even-field-fixed}.

\begin{lemma}\label{lem:even-unipotent-correspondence}
If $C$ is a unipotent $\ell$-block of $H=\Sp_{2r}(2^a)$, then there is an
$E_H$-equivariant bijection $\mathcal X_C\longrightarrow\mathcal W(C)$.
\end{lemma}

\begin{proof}
Let $(\mathbf L_0,\lambda_0)$ be the unipotent $e$-cuspidal pair
labelling $C$.  As shown in the proof of
Lemma~\ref{lem:even-field-fixed}, the set
$\mathcal X_C=\Irr(C)\cap\mathcal E(H,1)$ is the generalised
$e$-Harish-Chandra series attached to $(\mathbf L_0,\lambda_0)$.
Taking the inverse of the parametrisation in
\cite[Theorem~2.8]{CabanesSpath2013} gives a bijection
$\mathcal X_C\longrightarrow\Irr\bigl(W_H(\mathbf L_0,\lambda_0)\bigr)$.
For unipotent characters, $e$-cuspidality and $e$-Jordan-cuspidality coincide
\cite[Definition~2.1 and Remark~2.2]{KessarMalle2015}.
The block labellings in \cite[Theorem~4.4(i)]{CabanesEnguehard1994} and
\cite[Theorem~A(e)]{KessarMalle2015} therefore assign $C$ to the same pair
$(\mathbf L_0,\lambda_0)$.  Thus
\cite[Theorem~7.5]{FengMalleZhang2026} gives a bijection from
$\Irr\bigl(W_H(\mathbf L_0,\lambda_0)\bigr)$ to
\[
 \bigsqcup_{(\mathbf M,\lambda)}
 \dz\bigl(W_H(\mathbf M,\lambda)\bigr),
\]
where the pairs run through the $H$-classes of $e$-JGC pairs
with $\lambda\in\mathcal E(\mathbf M^{F'},\ell')$
whose induced block is $C$.  The maximal extendibility condition holds
here by the type $\mathsf C$ case of
\cite[Proposition~3.20]{FengMalleZhang2026}.  By
\cite[Lemma~3.21]{FengMalleZhang2026} and the definition of
$\mathcal W(C)$, this disjoint union is in bijection with $\mathcal W(C)$.
Composing these maps gives a bijection
$\mathcal X_C\longrightarrow\mathcal W(C)$.  By
Lemma~\ref{lem:even-field-fixed}, $E_H$ fixes both sets pointwise, so this
bijection is $E_H$-equivariant.
\end{proof}

The BAW-goodness assertion below appears in
\cite[Remark~5.9]{FengLiZhang2022Jordan}. Here we derive it from
the preceding correspondence for generic weights and an integral basic set.

\begin{proposition}\label{prop:even-unipotent}
Every unipotent $\ell$-block of $H=\Sp_{2r}(2^a)$ is BAW-good and
admits an iBAW bijection.
\end{proposition}

\begin{proof}
Let $C$ be a unipotent $\ell$-block of $H$.
The hypotheses of \cite[Condition~6.1]{FengMalleZhang2026} hold:
$\mathbf H$ is simple and simply connected, $\ell$ is odd and good, and
$\ell\nmid 2^{a+1}|Z(\mathbf H)^{F'}|$.
The group $\mathcal B$ of field and graph automorphisms in
\cite[Condition~6.1]{FengMalleZhang2026} is $E_H$, since type
$\mathsf C_r$ has no nontrivial graph automorphisms.

Since $Z(\mathbf H)/Z^\circ(\mathbf H)=1$, the hypothesis of
\cite[Theorem~A]{Geck1993} holds.  The blockwise consequence stated
immediately after that theorem, together with the equality
$\mathcal X_C=\Irr(C)\cap\mathcal E(H,1)$ established in the proof of
Lemma~\ref{lem:even-field-fixed}, shows that $\mathcal X_C$ is an integral
basic set for $C$.  By Lemma~\ref{lem:even-field-fixed}, $E_H$ fixes
$\mathcal X_C$ pointwise, so $C$ is $E_H$-stable.
The decomposition map is $E_H$-equivariant and maps $\mathcal X_C$ to a
basis of $\mathbb Z\IBr(C)$.  Hence $E_H$ also fixes $\IBr(C)$ pointwise,
and any bijection $\IBr(C)\longrightarrow\mathcal X_C$ is $E_H$-equivariant.  Since
$\mathcal B=E_H$,
  \cite[Theorem~6.2]{FengMalleZhang2026} gives an
$E_H$-equivariant bijection $\mathcal W(C)\longrightarrow\Alp(C)$.
Composing this map with the preceding bijection
$\IBr(C)\longrightarrow\mathcal X_C$ and the bijection
$\mathcal X_C\longrightarrow\mathcal W(C)$ from
Lemma~\ref{lem:even-unipotent-correspondence} gives an $E_H$-equivariant
bijection
\[
  \Omega_C:\IBr(C)\longrightarrow\Alp(C).
\]

For $r\geq4$, the group $H=\Sp_{2r}(2^a)$ is nonabelian simple,
with $\Aut(H)=H\rtimes E_H$.
By Table~\ref{tab:typec-carriers}, $H$ is its own universal covering
group.  Hence Lemma~\ref{lem:cyclic-outer-baw} proves the result.
\end{proof}

We now combine Proposition~\ref{prop:even-unipotent} with Jordan reduction
to obtain the result for even $q$ and $n\geq4$. This gives an independent
proof of the inductive conclusion in \cite[Theorem~3]{Li2021} in these
ranks, without its unitriangularity assumption.
At linear primes, the unitriangularity assumed in \cite[Theorem~3]{Li2021}
follows from \cite[Theorem~8.2]{GruberHiss1997}.

\begin{proposition}\label{prop:even-bridge}
Let $G=\Sp_{2n}(q)$ and $S=\PSp_{2n}(q)$, where $q=2^f$, $n\geq4$,
and $\ell$ is an odd prime dividing $|G|$.
Then every $\ell$-block of $G$ is BAW-good.
Consequently, $S$ satisfies iBAW at $\ell$ on $G=X_\ell(S)$.
\end{proposition}

\begin{proof}
Let $\mathbf G$ be a simple, simply connected algebraic group of type
$\mathsf C_n$ in characteristic $2$, and let $F$ be a Frobenius endomorphism
such that $\mathbf G^F=G$.  We verify the two hypotheses of
Lemma~\ref{lem:jordan-reduction}.  Choose the standard regular embedding
$\mathbf G=\Sp_{2n}\hookrightarrow\widetilde{\mathbf G}=\CSp_{2n}$,
and set $\widetilde G=\widetilde{\mathbf G}^F=\CSp_{2n}(q)$.
Let $\mu:\widetilde G\to\F_q^\times$ be the multiplier homomorphism,
characterised by $\langle gv,gw\rangle=\mu(g)\langle v,w\rangle$ for
$g\in\widetilde G$ and $v,w\in\F_q^{2n}$, where
$\langle\cdot,\cdot\rangle$ is the defining symplectic form.
Since $q$ is even, for each $g\in\widetilde G$ there is
$c\in\F_q^\times$ with $c^2=\mu(g)$.  Then $c^{-1}g\in G$, so
$\widetilde G=GZ(\widetilde G)$ and $\widetilde G$ acts on $G$
by inner automorphisms. Let $E\cong C_f$ act on $G$ and
$\widetilde G$ by the field automorphisms induced by entrywise squaring.
Then
$\Aut(G)=G\rtimes E$.  For every $\psi\in\IBr(G)$, the group
$\widetilde G$ fixes $\psi$, and hence
\[
 (\widetilde G\rtimes E)_\psi=\widetilde G\rtimes E_\psi.
\]
The group $E_\psi$ is cyclic, so by Lemma~\ref{lem:cyclic-extension},
$\psi$ extends to $G\rtimes E_\psi$.
For every semisimple $\ell'$-element $s\in(\mathbf G^*)^F$, the
construction in \cite[Corollary~4.2 and Lemmas~4.3 and~4.5]{Ruhstorfer2022Derived}
allows us, after conjugating the Levi subgroup associated with $s$ and
its series idempotent by an element of $G$, to choose the group $A_s$ in
Lemma~\ref{lem:jordan-reduction} inside $E$, since type~$\mathsf C_n$
with $n\geq4$ has no graph automorphism.
Restricting the preceding factorisation and extensions to $A_s$ verifies
condition~\textup{(i)} of that lemma.

To verify condition~\textup{(ii)} of
Lemma~\ref{lem:jordan-reduction}, let $(\mathbf K,F')$ be a pair
as in that condition.  If $\ell\nmid|\mathbf K^{F'}|$, then every $\ell$-block $c$ has defect
zero.  Writing $\Irr(c)=\{\chi\}$, the map $\chi^0\mapsto(1,\chi)$ is
an iBAW bijection for $c$, since the global and local modular character
triples are identical.  We may therefore assume that
$\ell\mid|\mathbf K^{F'}|$.

A connected Dynkin subdiagram of $\mathsf C_n$ has type $\mathsf A_m$ or
$\mathsf C_m$. If $\mathbf K$ has type~$\mathsf A_m$
(including $\mathsf C_1=\mathsf A_1$), then $\mathbf K^{F'}$
is a special linear or unitary group. The result for these groups
established after Lemma~\ref{lem:jordan-reduction} gives the required
iBAW bijections for its $\ell$-blocks.

Suppose that $\mathbf K$ has type $\mathsf C_m$, where $m\geq2$, and let
$c$ be a strictly quasi-isolated $\ell$-block of $\mathbf K^{F'}$.  Choose a
strictly quasi-isolated semisimple $\ell'$-element
$s\in\mathbf K^{*F'}$ such that
$\Irr(c)\subseteq\mathcal E_\ell(\mathbf K^{F'},s)$.
Since $\mathbf K^*$ has type $\mathsf B_m$ in characteristic $2$
and $s$ is quasi-isolated, \cite[Example~4.8]{Bonnafe2005} gives
$s=1$.  Hence $c$ is unipotent.

First suppose that $F'$ is split, and write
$\mathbf K^{F'}=\Sp_{2m}(2^a)$. For $m\geq4$,
Proposition~\ref{prop:even-unipotent} gives the required iBAW bijection
for $c$. The groups $\Sp_4(2^a)$ with $a\geq2$ and
$\Sp_6(2^a)$ with $a\geq1$ satisfy iBAW at $\ell$ by
\cite[Theorem~5.5]{SchaefferFry2014}.  Corollary~\ref{cor:fixed-point-descent}
therefore gives the required iBAW bijections for their $\ell$-blocks.
The simplicity assumption in Lemma~\ref{lem:jordan-reduction}\textup{(ii)}
excludes $\Sp_4(2)$.
The program \path{sp6.g}, described in
Section~\ref{app:sp6}, supplies the numerical details of the
calculation on $2.\Sp_6(2)$ at $\ell=3$ in the proof of
\cite[Theorem~5.5]{SchaefferFry2014}.
If $F'$ is not split, then $m=2$ and $\mathbf K^{F'}$ is a
Suzuki group.  Each simple Suzuki group satisfies iBAW at $\ell$
by \cite[Corollary~6.3]{Spath2013}, so
Corollary~\ref{cor:fixed-point-descent} gives the required iBAW
bijections for its $\ell$-blocks.

This proves condition~\textup{(ii)} of
Lemma~\ref{lem:jordan-reduction}. Condition~\textup{(i)} was verified above.
Since $G=S$ is its own universal covering group by
Table~\ref{tab:typec-carriers}, Lemma~\ref{lem:jordan-reduction}
shows that every
$\ell$-block of $G$ is BAW-good.  Since $G=X_\ell(S)$,
Remark~\ref{rem:block-aggregation}\textup{(a)} gives the final assertion.
\end{proof}

\subsection{Odd fields}

Throughout this subsection, assume that $q=p^f$ is odd, $n\geq3$, and that
$\ell\mid |S|$ is an odd prime distinct from $p$.
For linear primes, the unitriangularity assumed in
\cite[Theorem~2]{Li2021} follows from
\cite[Theorem~8.2 and Corollary~8.7]{GruberHiss1997}.
For $\Sp_{2n}(q)$ and $\operatorname{Spin}_{2n+1}(q)$, Feng and Sp\"ath
proved unitriangularity for the sums of $\ell$-blocks associated with
semisimple $\ell'$-elements whose connected centralisers are Levi
subgroups of the dual algebraic group
\cite[Proposition~7.3]{FengSpath2023}.

Set
\[
   G=\Sp_{2n}(q),\qquad \widetilde G=\CSp_{2n}(q).
\]
Let $F_p$ be the automorphism of $\widetilde G$ obtained by raising every
matrix entry to the $p$th power, and set $E=\langle F_p\rangle\cong C_f$.
The group $E$ acts on $G$ by restriction.  Here $G=X_\ell(S)$.  We use the conformal group to describe the diagonal
and field actions.  Let $N$ be the group of linear Brauer
characters of $\widetilde G/G$, and set $A=N\rtimes E$.
The group $N$ is cyclic of order equal to the $\ell'$-part of $q-1$.
The group $A$ acts on the $\ell$-blocks and irreducible Brauer
characters of $\widetilde G$, with $N$ acting by tensoring and
$E$ by field automorphisms.  For a block
$\widetilde B$ of $\widetilde G$, let $J=A_{\widetilde B}$ be its
stabiliser in $A$. Write $D=J\cap N$.

\begin{lemma}\label{lem:exact-stabilizer}
For every $\ell$-block $\widetilde B$ of $\widetilde G=\CSp_{2n}(q)$,
$D$ is a normal subgroup of $J$ of order at most $2$, and $J/D$ is cyclic.
Consequently, every subgroup of $J$ is $2$-hypoelementary.
\end{lemma}

\begin{proof}
The multiplier homomorphism $\mu:\widetilde G\to\F_q^\times$ is
surjective with kernel $G$, and $\mu(cI)=c^2$ for $c\in\F_q^\times$.
Thus it identifies $\widetilde G/G$ with $\F_q^\times$ and sends
$Z(\widetilde G)$ onto $(\F_q^\times)^2$.
Hence $|\widetilde G:GZ(\widetilde G)|=2$.
Applying Lemma~\ref{lem:linear-block-stabilizer} to
$G\lhd\widetilde G$ and $\widetilde B$, using reduction to identify
the linear characters, gives $|D|\leq2$.
Since $N\triangleleft A$, we have $D\triangleleft J$.
The projection $A\to E$ embeds $J/D$ in $E$, so $J/D$ is cyclic.
As $|D|\leq2$, we have $D\leq O_2(J)$, and $J/O_2(J)$ is cyclic.
Thus $J$ and every subgroup of $J$ are $2$-hypoelementary.
\end{proof}

Write $Z_{\widetilde B}=\Irr(\widetilde B)\cap
\mathcal E(\widetilde G,\ell')$.
The set $Z_{\widetilde B}$ is stable under $J$.  Indeed,
$\mathcal E(\widetilde G,\ell')$ is preserved by tensoring with linear
$\ell'$-characters \cite[Remark~2.5]{Li2021} and by field automorphisms
\cite[Theorem~3.1]{CabanesSpath2013}.

\begin{proposition}
\label{prop:odd-conformal-bridge}
For every $\ell$-block $\widetilde B$ of $\widetilde G=\CSp_{2n}(q)$,
there is a $J$-equivariant bijection
$\IBr(\widetilde B)\longrightarrow Z_{\widetilde B}$.
\end{proposition}

\begin{proof}
The prime $\ell$ is odd and hence good for type $\mathsf C$.
The conformal algebraic group has connected centre.  By
\cite[Theorem~5.1]{GeckHiss1991}, $\mathcal E(\widetilde G,\ell')$
is an integral basic set for $\widetilde G$.
Since the decomposition map respects the block decomposition, $Z_{\widetilde B}$
is an integral basic set for $\widetilde B$.

Identify $N$ with the group of ordinary linear $\ell'$-characters of
$\widetilde G/G$ via reduction.  This identification is $E$-equivariant,
and the decomposition map commutes with tensoring and field
automorphisms.  Hence $d_{\widetilde B}$ is $J$-equivariant.
The set $Z_{\widetilde B}$ is $J$-stable, and $J$ is
$2$-hypoelementary by Lemma~\ref{lem:exact-stabilizer}.
Lemma~\ref{lem:basicset-bridge} gives the asserted bijection.
\end{proof}

\begin{lemma}
\label{lem:odd-g-factorization}
Let $G=\Sp_{2n}(q)$ and $\widetilde G=\CSp_{2n}(q)$.
Then there is an $\Aut(G)$-equivariant bijection
\[
  \beta:\IBr(G)\longrightarrow\mathcal E(G,\ell')
\]
such that
\[
  \beta(\IBr(B))=\Irr(B)\cap\mathcal E(G,\ell')
\]
for every $\ell$-block $B$ of $G$.
Moreover, for every $\varphi\in\IBr(G)$,
\[
  (\widetilde G\rtimes E)_\varphi
    =\widetilde G_\varphi\rtimes E_\varphi,
\]
and $\varphi$ extends to $G\rtimes E_\varphi$.
\end{lemma}

\begin{proof}
Let $\mathbf G$ be the simply connected symplectic algebraic group and let
$F$ be its defining Frobenius endomorphism, so that $\mathbf G^F=G$.
Since $\ell$ is odd and $Z(\mathbf G)=\{\pm I\}$,
\cite[Theorem~A]{Geck1993} shows that $\mathcal E(G,\ell')$ is an
integral basic set for $G$.
Fix an $\ell$-block $B$ of $G$ and set
$\mathcal X_B=\Irr(B)\cap\mathcal E(G,\ell')$.
Since the decomposition map respects the block decomposition, $\mathcal X_B$ is an
integral basic set for $B$.

Diagonal automorphisms preserve each rational Lusztig series
\cite[Proposition~15.6(i)]{CabanesEnguehard2004}. Field automorphisms
send a series $\mathcal E(G,s)$ to a series $\mathcal E(G,s')$ with
$|s'|=|s|$ \cite[Proposition~7.2]{Taylor2018Automorphisms}. Thus
$\mathcal E(G,\ell')$ is $\Aut(G)$-stable, and $\mathcal X_B$ is
$\Aut(G)_B$-stable.
The actions on $\mathcal X_B$ and $\IBr(B)$ factor through
$L=\Aut(G)_B/\Inn(G)$, and the decomposition map is $L$-equivariant.
The group $\Out(G)\cong C_2\times C_f$ is $2$-hypoelementary, so
the same holds for $L$.
Lemma~\ref{lem:basicset-bridge} gives an $\Aut(G)_B$-equivariant
bijection $\beta_B:\IBr(B)\longrightarrow\mathcal X_B$.
Applying Remark~\ref{rem:block-aggregation}\textup{(b)} to these blockwise
bijections gives the required map $\beta$.

Let $\varphi\in\IBr(G)$ and write $\chi=\beta(\varphi)$.
Applying \cite[Theorem~3.1]{CabanesSpath2017TypeC} to the ordinary
character $\chi$ gives
$(\widetilde G\rtimes E)_\chi=\widetilde G_\chi\rtimes E_\chi$.
Equivariance and bijectivity of $\beta$ therefore give
\[
  (\widetilde G\rtimes E)_\varphi
    =(\widetilde G\rtimes E)_\chi
    =\widetilde G_\chi\rtimes E_\chi
    =\widetilde G_\varphi\rtimes E_\varphi.
\]
Finally, $E_\varphi$ is cyclic, so Lemma~\ref{lem:cyclic-extension}
gives the required extension.
\end{proof}

\begin{remark}
The proofs of Proposition~\ref{prop:odd-conformal-bridge} and
Lemma~\ref{lem:odd-g-factorization} apply Conlon's theorem over the
$2$-adic integers through Lemma~\ref{lem:basicset-bridge}.  Here $\ell$ is odd.  These arguments give neither a canonical Brauer
labelling nor a unitriangular decomposition matrix.
\end{remark}

For odd $q$, the following proposition removes the unitriangularity
assumptions from Li's result \cite[Theorem~2]{Li2021}. We use the ordinary character
and weight calculations in that paper, together with its stabiliser
and extension results for weights. Proposition~\ref{prop:odd-conformal-bridge}
replaces the use of unitriangularity for $\widetilde G$.
Lemma~\ref{lem:odd-g-factorization} gives the conclusions of
\cite[Lemma~5.4]{Li2021} without the unitriangularity assumption for $G$.

\begin{proposition}
\label{prop:odd-conlon}
Let $G=\Sp_{2n}(q)$ and $S=\PSp_{2n}(q)$.
Then every $\ell$-block of $G$ satisfies the inductive BAW condition.
Consequently, $S$ satisfies iBAW at $\ell$ on $G=X_\ell(S)$.
\end{proposition}

\begin{proof}
We apply \cite[Theorem~4.5]{BroughSpath2022} with $\mathcal B$ equal to
the set of all $\ell$-blocks of $G$ and with $\widetilde G$ and $E$ as
above.  The group $E$ is cyclic, $G=[\widetilde G,\widetilde G]$, and
\[
 C_{\widetilde G\rtimes E}(G)=Z(\widetilde G),
 \qquad
 (\widetilde G\rtimes E)/Z(\widetilde G)\cong\Aut(G).
\]
Since $\widetilde G/G$ is cyclic, every
$\varphi\in\IBr(G)$ extends to $\widetilde G_\varphi$ by
Lemma~\ref{lem:cyclic-extension}.
For a weight $(Q,\theta)$ of $G$, the quotient
$(N_{\widetilde G}(Q)/Q)_\theta/(N_G(Q)/Q)$ is also cyclic,
so the same lemma extends $\theta$ to
$(N_{\widetilde G}(Q)/Q)_\theta$.
This verifies condition~\textup{(i)} of
\cite[Theorem~4.5]{BroughSpath2022}.

The parametrisations of ordinary characters in \cite[Lemma~3.13]{Li2021} and
of weights in \cite[Theorem~4.13 and Remarks~4.7--4.8]{Li2021} give a
bijection
\[
 \rho:\mathcal E(\widetilde G,\ell')
       \longrightarrow\Alp(\widetilde G)
\]
that preserves blocks.
The proofs of \cite[Lemmas~5.1 and~5.3]{Li2021} show that field
automorphisms and tensoring by $N$ induce the same changes of labels on
the ordinary characters and on the corresponding weights.
In the proof of \cite[Lemma~5.1]{Li2021},
unitriangularity identifies the field action on Brauer character labels
with that on ordinary character labels. In the proof of
\cite[Lemma~5.3]{Li2021}, it identifies the corresponding actions of
tensoring with linear characters. The comparison between ordinary
characters and weights uses neither identification. Thus $\rho$ is
$A$-equivariant.

Choose one block $\widetilde B$ in each $A$-orbit and let
$J=A_{\widetilde B}$.
Proposition~\ref{prop:odd-conformal-bridge} gives a $J$-equivariant
bijection
\[
 \alpha_{\widetilde B}:\IBr(\widetilde B)
       \longrightarrow Z_{\widetilde B}.
\]
Composing with $\rho$ and applying Remark~\ref{rem:block-aggregation}\textup{(b)}
to the $A$-orbits gives an
$A$-equivariant bijection
\[
 \widetilde\Omega:\IBr(\widetilde G)\longrightarrow\Alp(\widetilde G)
\]
that preserves blocks.
In the notation of \cite[Theorem~4.5(ii)]{BroughSpath2022},
\[
 J_G(\widetilde\varphi)
   =J_G(\widetilde\Omega(\widetilde\varphi))
   =\widetilde G
 \qquad(\widetilde\varphi\in\IBr(\widetilde G))
\]
by \cite[Remark~2.2(1)]{Li2021}.
This verifies condition~\textup{(ii)} of
\cite[Theorem~4.5]{BroughSpath2022}.
For every $\varphi\in\IBr(G)$,
Lemma~\ref{lem:odd-g-factorization} gives the stabiliser factorisation
and extension required by condition~\textup{(iii)} of
\cite[Theorem~4.5]{BroughSpath2022}.

Let $\omega\in\Alp(G)$.
By \cite[Section~2.C]{Li2021}, choose a representative $(R,\psi)$
with $R$ in the standard family used in \cite[Lemma~5.5]{Li2021}.
By \cite[Lemma~5.5(1)--(2)]{Li2021},
\[
 (\widetilde GE)_{R,\psi}
   =\widetilde G_{R,\psi}(GE)_{R,\psi},
\]
and $\psi$ extends to $(GE)_{R,\psi}/R$.
For $H\in\{\widetilde GE,\widetilde G,GE\}$, the stabiliser of the
conjugacy class $\omega$ satisfies $H_\omega=GH_{R,\psi}$.  Hence
\[
 (\widetilde GE)_\omega
   =\widetilde G_\omega(GE)_\omega
   =\widetilde G_\omega E_\omega.
\]
This verifies condition~\textup{(iv)} of
\cite[Theorem~4.5]{BroughSpath2022}.

Since $\Out(G)$ is abelian, \cite[Theorem~4.5]{BroughSpath2022} proves
the inductive BAW condition for every $\ell$-block of $G$.
As $G=X_\ell(S)$, Remark~\ref{rem:block-aggregation}\textup{(a)} shows that $S$
satisfies iBAW at $\ell$.
\end{proof}

\Needspace{6\baselineskip}
\subsection{Proof of Theorem~\ref{thm:symplectic}}

\begin{proof}
If $n=2$, Proposition~\ref{prop:rank-two} gives the result.
We may therefore assume that $n\geq3$.
The case $\ell=p$ follows from \cite[Theorem~C]{Spath2013}, so
assume that $\ell\neq p$.

If $q$ is odd, Proposition~\ref{prop:odd-two} applies when $\ell=2$,
and Proposition~\ref{prop:odd-conlon} applies when $\ell$ is odd.
If $q$ is even, then $\ell$ is odd.  For $n=3$, the result follows
from \cite[Theorem~5.5]{SchaefferFry2014}, and for $n\geq4$ it follows
from Proposition~\ref{prop:even-bridge}.
\end{proof}

\section{Odd-dimensional orthogonal groups}\label{sec:type-b}

In even characteristic, the finite simple groups of types $\mathsf B_n$
and $\mathsf C_n$ are isomorphic.  In every characteristic, the
identifications $\mathsf B_1=\mathsf A_1$ and $\mathsf B_2=\mathsf C_2$
reduce the cases of ranks one and two to the corresponding results for
types $\mathsf A$ and $\mathsf C$.  Unless stated otherwise, throughout
this section $q=p^f$ is an odd prime power, $n\geq3$, and
$S=\Omega_{2n+1}(q)$ denotes the nonabelian simple group under consideration.
Our arguments use special
Clifford groups, reduction to proper Levi subgroups, and generalised
Gelfand--Graev characters, and include a separate treatment of the
exceptional universal $2'$-covering group $3.\Omega_7(3)$.

\begin{theorem}\label{thm:type-b}
Let $S=\Omega_{2n+1}(q)$, and let $\ell$ be a prime dividing $|S|$.
Then every $\ell$-block of $X_\ell(S)$ satisfies the inductive BAW
condition.  Consequently, $S$ satisfies iBAW at $\ell$ on the universal
$\ell'$-covering group $X_\ell(S)$.
\end{theorem}

\subsection{Covering groups}\label{subsec:type-b-covers}

By
\cite[Tables~24.2--24.3 and Remark~24.19]{MalleTesterman2011}, the
universal covering group of $\Omega_{2n+1}(q)$ has kernel $C_2$, except that the
multiplier of $\Omega_7(3)$ is $C_6$.  The required covering groups are
listed in Table~\ref{tab:typeb-covers}.

\begin{table}[ht]
\centering
\caption{Universal $\ell'$-covering groups of type $\mathsf B$ in odd characteristic.}
\label{tab:typeb-covers}
\begin{tabular}{@{}lll@{}}
\toprule
Parameters & Prime & $X_\ell(\Omega_{2n+1}(q))$ \\
\midrule
$(n,q)\neq(3,3)$ & $\ell=2$
  & $\Omega_{2n+1}(q)$ \\
$(n,q)\neq(3,3)$ & $\ell$ odd
  & $\operatorname{Spin}_{2n+1}(q)$ \\
$(n,q)=(3,3)$ & $\ell=2$
  & $3.\Omega_7(3)$ \\
$(n,q)=(3,3)$ & $\ell=3$
  & $2.\Omega_7(3)=\operatorname{Spin}_7(3)$ \\
$(n,q)=(3,3)$ & $\ell\notin\{2,3\}$
  & $6.\Omega_7(3)$ \\
\bottomrule
\end{tabular}
\end{table}

The last line is used only for $\ell\in\{5,7,13\}$.  In particular, the
universal $2'$-covering group of $\Omega_7(3)$ is $3.\Omega_7(3)$.

We have $\Out(S)\cong C_2\times C_f$, where the two factors are induced by
diagonal and field automorphisms, respectively.  See, for example,
\cite[Proposition~24.21 and Theorem~24.24]{MalleTesterman2011}.

\subsection{Odd nondefining primes}\label{subsec:type-b-odd-primes}

For odd nondefining primes, Feng, Li, and Zhang proved the result for
$\Omega_7(3)$ without a unitriangularity assumption and for the remaining
groups under their unitriangularity assumption
\cite[Theorem~1 and Assumption~3.11]{FengLiZhang2022TypeB}.
We remove this assumption from the bijection and stabiliser arguments
using permutation lattices and Lemma~\ref{lem:linear-block-stabilizer}.

Throughout this subsection, $\ell$ is an odd prime not dividing $q$.
Let $G=\operatorname{Spin}_{2n+1}(q)$, and let $G\lhd\widetilde G$ be its
embedding in the corresponding finite special Clifford group described in
\cite[Section~3.1]{FengLiZhang2022TypeB}. Let $F_p$ be the field
automorphism of $\widetilde G$ induced by $x\mapsto x^p$ on $\F_q$,
and set $E=\langle F_p\rangle$. The group $E$ acts on $G$ by restriction.
Let $C$ be the group of ordinary linear characters of $\widetilde G/G$ of
$\ell'$-order, and set $A=C\rtimes E$.
The group $A$ acts on the $\ell$-blocks and ordinary and Brauer characters
of $\widetilde G$.  Here $C$ acts by tensoring with the inflated linear
characters, using their reductions on Brauer characters, and $E$ acts by
field automorphisms.  The natural actions of $\Aut(G)$ on blocks and characters factor through
$\Out(G)$.
For $M\in\{G,\widetilde G\}$ and an $\ell$-block $B$ of $M$, set
\[
 \mathcal X_B=\Irr(B)\cap\mathcal E(M,\ell').
\]

\begin{lemma}\label{lem:type-b-conlon-block}
Let $G=\operatorname{Spin}_{2n+1}(q)$, and let $\widetilde G$ be the
corresponding finite special Clifford group.
Let $M\in\{G,\widetilde G\}$, let $B$ be an $\ell$-block of $M$, and set
\[
 J_B=
 \begin{cases}
  \Aut(G)_B/\Inn(G),&M=G,\\
  A_B,&M=\widetilde G.
 \end{cases}
\]
Then $J_B$ is $2$-hypoelementary, $\mathcal X_B$ is $J_B$-stable, and the
decomposition map restricts to an isomorphism
\[
 d_B:\mathbb Z[\mathcal X_B]\xrightarrow{\sim}\mathbb Z[\IBr(B)]
\]
of $\mathbb Z J_B$-permutation lattices.  Consequently, there exists a
$J_B$-equivariant bijection
\[
 \IBr(B)\longrightarrow\mathcal X_B.
\]
\end{lemma}

\begin{proof}
The connected reductive algebraic groups defining $G$ and $\widetilde G$
have type $\mathsf B_n$.  Their centres have component groups
$Z(\mathbf H)/Z^\circ(\mathbf H)$ of orders $2$ and $1$, respectively,
where $\mathbf H$ denotes the corresponding algebraic group.
Since the odd prime $\ell$ is good for type $\mathsf B_n$ and
divides the order of neither component group, the hypotheses
of \cite[Theorem~2.3]{FengLiZhang2022TypeB} hold for both groups.  Applying that theorem to all rational $\ell'$-series
shows that $\mathcal E(G,\ell')$ and $\mathcal E(\widetilde G,\ell')$
are integral basic sets.  Since the decomposition map is block diagonal,
$\mathcal X_B$ is an integral basic set for $B$.

Diagonal automorphisms preserve each rational Lusztig series
\cite[Proposition~15.6(i)]{CabanesEnguehard2004}. For
$M\in\{G,\widetilde G\}$, field automorphisms send a series
$\mathcal E(M,s)$ to a series $\mathcal E(M,s')$ with $|s'|=|s|$
\cite[Proposition~7.2]{Taylor2018Automorphisms}. Thus field and diagonal
automorphisms preserve $\mathcal E(G,\ell')$, and field automorphisms
preserve $\mathcal E(\widetilde G,\ell')$.
Tensoring with the inflation of an ordinary linear character of
$\widetilde G/G$ of $\ell'$-order also preserves
$\mathcal E(\widetilde G,\ell')$, since it multiplies the semisimple
element labelling the series by a central $\ell'$-element
\cite[Theorem~4.7.1(3)]{GeckMalle2020}.
Hence $\mathcal X_B$ is $J_B$-stable.  The decomposition map commutes
with automorphisms and with tensoring by linear characters, so its
restriction to $\mathbb Z[\mathcal X_B]$ is $J_B$-equivariant.
For $M=G$, we have
$\Out(G)\cong C_2\times E$.  This group is $2$-hypoelementary, and
therefore so is its subgroup $J_B$.

Now let $M=\widetilde G$, write $J=A_B$, and set $D=J\cap C$.
We have $|\widetilde G:GZ(\widetilde G)|=2$
\cite[proof of Lemma~5.16]{FengLiZhang2022TypeB}.
Applying Lemma~\ref{lem:linear-block-stabilizer} to
$G\lhd\widetilde G$ and $B$ gives $|D|\leq2$.
Since $C\lhd A$, we have $D\lhd J$, and hence $D\leq O_2(J)$.
Projection to $E$ embeds $J/D$ in the cyclic group $E$.
Therefore $J/O_2(J)$ is cyclic, so $J=J_B$ is $2$-hypoelementary.
Lemma~\ref{lem:basicset-bridge} now gives the asserted lattice isomorphism
and equivariant bijection in both cases.
\end{proof}

\begin{proposition}\label{prop:type-b-odd-primes}
Let $S=\Omega_{2n+1}(q)$, and let $\ell$ be an odd prime dividing
$|S|$ with $\ell\ne p$.
Then every $\ell$-block of $X_\ell(S)$ satisfies the inductive BAW
condition.
\end{proposition}

\begin{proof}
First suppose $(n,q)\neq(3,3)$.  Then $G=X_\ell(S)$ by
Table~\ref{tab:typeb-covers}.  We apply Brough--Sp\"ath's criterion in
the form of \cite[Theorem~2.1]{FengLiZhang2022TypeB} and verify its
conditions~\textup{(1)--(4)}.
The structural hypotheses in condition~\textup{(1)} hold for $G$,
$\widetilde G$, and $E$, as noted in the proof of
\cite[Theorem~1]{FengLiZhang2022TypeB}.

For each $A$-orbit of blocks of $\widetilde G$, choose a
representative $\widetilde B$. Compose the
$A_{\widetilde B}$-equivariant bijection
$\IBr(\widetilde B)\longrightarrow\mathcal X_{\widetilde B}$
of Lemma~\ref{lem:type-b-conlon-block} with the restriction
$\mathcal X_{\widetilde B}\longrightarrow\Alp(\widetilde B)$
of the blockwise $A$-equivariant bijection in
\cite[Proposition~7.2]{FengLiZhang2022TypeB}.
Extending these bijections to the remaining blocks by
Remark~\ref{rem:block-aggregation}\textup{(b)} gives the blockwise $A$-equivariant
bijection required by condition~\textup{(2)(i)} of the
criterion. Condition~\textup{(2)(ii)} holds by
\cite[Remark~2.2]{FengLiZhang2022TypeB}, since $\ell$ is odd.

Let $B$ be an $\ell$-block of $G$, let $\varphi\in\IBr(B)$, and let
$\chi\in\mathcal X_B$ be its image under the bijection of
Lemma~\ref{lem:type-b-conlon-block}. Equivariance and the fact that each
character determines its block give $\Aut(G)_\varphi=\Aut(G)_\chi$.
The ordinary character stabiliser theorem for type $\mathsf B$ due to
Cabanes and Sp\"ath, stated in
\cite[Theorem~3.10]{FengLiZhang2022TypeB}, gives
$(\widetilde G\rtimes E)_\chi=\widetilde G_\chi\rtimes E_\chi$.
The equality of the automorphism stabilisers identifies the stabilisers of
$\varphi$ and $\chi$ in $\widetilde G\rtimes E$, $\widetilde G$ and $E$.
Hence
\[
 (\widetilde G\rtimes E)_\varphi
   =\widetilde G_\varphi\rtimes E_\varphi.
\]
This verifies condition~\textup{(3)} of the criterion.
The weight stabiliser factorisation required by condition~\textup{(4)}
follows from \cite[Proposition~7.5]{FengLiZhang2022TypeB}.
The criterion therefore proves the inductive BAW condition for every
$\ell$-block of $G$.

If $(n,q)=(3,3)$, then $\ell\in\{5,7,13\}$ and
$X_\ell(S)=6.\Omega_7(3)$ by Table~\ref{tab:typeb-covers}.
Its Sylow $\ell$-subgroups have order $\ell$, so every $\ell$-block has
cyclic defect. As in the proof of \cite[Theorem~1]{FengLiZhang2022TypeB},
the result follows from \cite[Theorem~1.1]{KoshitaniSpath2016}.
\end{proof}

\subsection{Character representatives for the Levi reduction}
\label{subsec:type-b-two-jordan}

The reductions for nonprincipal $2$-blocks require compatible character
representatives on a Levi subgroup and its derived subgroup.
Lemma~\ref{lem:type-b-component-return} combines suitable representatives
on direct factors. Lemma~\ref{lem:type-b-regular-levi-orbits} describes
the product decomposition of the derived subgroup and the action induced
by the regular embedding.
Lemma~\ref{lem:type-b-char2-clifford} deduces the stabiliser
factorisation for a Brauer character of the Levi subgroup.
The required representatives and stabiliser factorisations on the
factors are obtained in the proof of
Proposition~\ref{prop:type-b-brauer-hypothesis}.

\begin{lemma}\label{lem:type-b-component-return}
Let $I=\{1,\ldots,r\}$, and for each $i\in I$ let $H_i$ be a finite group with
$\Delta_i\leq\Aut(H_i)$.  Set
\[
 N=H_1\times\cdots\times H_r,\qquad
 \Delta=\prod_{i\in I}\Delta_i,
\]
with $\Delta$ acting componentwise on $N$.  Let a cyclic group $E$ act
on $N$ by automorphisms, normalise $\Delta$, and permute the pairs
$(H_i,\Delta_i)$.
Choose $\chi_i\in\IBr(H_i)$ for each $i\in I$. Write $\mathcal O$ for the
$\Delta$-orbit of $\chi_1\times\cdots\times\chi_r$ and $\mathcal O_i$ for the
$\Delta_i$-orbit of $\chi_i$.
Let $E_{\mathcal O}$ be the setwise stabiliser of $\mathcal O$ in $E$,
and choose a generator $\tau$ of $E_{\mathcal O}$.
Suppose that for each cycle $C$ of the permutation of $I$ induced by
$\tau$, there are $i\in C$ and $\theta_i\in\mathcal O_i$ such that
\[
 (\Delta_i\rtimes\langle\tau^{|C|}\rangle)_{\theta_i}
   =(\Delta_i)_{\theta_i}\langle\tau^{|C|}\rangle_{\theta_i}.
\]
Then $\mathcal O$ contains a character $\theta$ fixed by
$E_{\mathcal O}$.  In particular,
\[
 (\Delta\rtimes E_{\mathcal O})_\theta
   =\Delta_\theta\rtimes E_{\mathcal O}.
\]
\end{lemma}

\begin{proof}
Since $\Delta$ acts componentwise, $\mathcal O$ consists of the products
of characters in the orbits $\mathcal O_i$, and $\tau$ sends
$\mathcal O_i$ to $\mathcal O_{i^\tau}$. For a cycle $C$ of length $t$,
choose $i$ and $\theta_i$ as in the hypothesis. Some element of
$\Delta_i\tau^t$ fixes $\theta_i$, and projection of the assumed
stabiliser equality onto $\langle\tau^t\rangle$ shows that $\tau^t$
fixes $\theta_i$. Set $\theta_{i^{\tau^j}}=\theta_i^{\tau^j}$ for
$0\leq j<t$ on each cycle. Their product belongs to $\mathcal O$ and is
fixed by $\tau$, hence by $E_{\mathcal O}$, which gives the stabiliser
equality.
\end{proof}

\begin{lemma}\label{lem:type-b-regular-levi-orbits}
Let $\mathbf G\hookrightarrow\widetilde{\mathbf G}$ be a regular
embedding, where $\mathbf G$ is simply connected of type $\mathsf B_n$.
Let $F$ be a Steinberg endomorphism of $\widetilde{\mathbf G}$
stabilising $\mathbf G$, and let $\mathbf L$ be an $F$-stable Levi
subgroup of $\mathbf G$. Set
\[
 \mathbf H=[\mathbf L,\mathbf L],\qquad
 L_0=\mathbf H^F,\qquad L=\mathbf L^F,\qquad
 \widetilde L=(\mathbf LZ(\widetilde{\mathbf G}))^F.
\]
Write $\mathbf H=\prod_{j\in I}\mathbf H_j$, where $I$ is a finite index set
for its simple components. Denote the $F$-orbits on $I$ by $C_1,\ldots,C_r$.
Choose $x_i\in C_i$, and set $n_i=|C_i|$ and $H_i=\mathbf H_{x_i}^{F^{n_i}}$.
Then the following hold.
\begin{enumerate}[label=\textup{(\arabic*)}]
\item There is an isomorphism
\[
 L_0\cong\prod_{i=1}^rH_i,\tag{\ref{lem:type-b-regular-levi-orbits}.1}
 \label{eq:type-b-regular-levi-product}
\]
where conjugation by $\widetilde L$ preserves each factor $H_i$.
Let $\rho_i:\widetilde L\longrightarrow\Aut(H_i)$ be the conjugation homomorphism on the $i$th factor of
\eqref{eq:type-b-regular-levi-product}, and set
$\Delta_i=\rho_i(\widetilde L)$.  Then $\Delta_i$ is the full group
of inner and diagonal automorphisms of $H_i$, and the homomorphism
\[
 \widetilde L\longrightarrow\prod_{i=1}^r\Delta_i,\qquad
 g\longmapsto\bigl(\rho_1(g),\ldots,\rho_r(g)\bigr)
\]
is surjective.
\item For $\chi_i\in\IBr(H_i)$, let $\mathcal O_i$ be the
$\Delta_i$-orbit of $\chi_i$.  Using
\eqref{eq:type-b-regular-levi-product}, the $\widetilde L$-orbit of
$\chi_1\times\cdots\times\chi_r$ is
\[
 \left\{\psi_1\times\cdots\times\psi_r\;\middle|\;
 \psi_i\in\mathcal O_i\text{ for }1\leq i\leq r\right\}.
\]
\item We have $L_0\lhd L\lhd\widetilde L$, and the quotient
$\widetilde L/L_0$ is abelian.
If $\sigma$ is a field endomorphism of $\widetilde{\mathbf G}$ commuting
with $F$ and stabilising $\mathbf G$ and $\mathbf L$, then the
automorphism induced by $\sigma$ on $\widetilde{\mathbf G}^F$
normalises this chain.
\end{enumerate}
\end{lemma}

\begin{proof}
\textup{(1)} Since $\mathbf G$ is simply connected,
\cite[Proposition~12.14]{MalleTesterman2011} shows that
$\mathbf H=[\mathbf L,\mathbf L]$ is simply connected.
It is therefore a direct product of its simple components.
For $1\leq i\leq r$, set
\[
 \mathbf K_i=\prod_{j\in C_i}\mathbf H_j.
\]
Each $\mathbf K_i$ is $F$-stable, and projection onto the component
$\mathbf H_{x_i}$ gives an isomorphism
\[
 \mathbf K_i^F\longrightarrow\mathbf H_{x_i}^{F^{n_i}}=H_i.
\]
Taking the product over the $F$-orbits gives
\eqref{eq:type-b-regular-levi-product}.

Set $\mathbf M=\mathbf LZ(\widetilde{\mathbf G})$.
This is an $F$-stable Levi subgroup of $\widetilde{\mathbf G}$.
Its centre is connected because $Z(\widetilde{\mathbf G})$ is connected.
We have $[\mathbf M,\mathbf M]=\mathbf H$ and
$\mathbf M=\mathbf HZ(\mathbf M)$. Thus conjugation by $\mathbf M$
preserves each direct factor of $\mathbf H$, and
\[
 \mathbf M/Z(\mathbf M)\cong\mathbf H/Z(\mathbf H)
 \cong\prod_i\mathbf K_i/Z(\mathbf K_i).
\]
Since $Z(\mathbf M)$ is connected, Lang's theorem makes the natural
homomorphism
\begin{equation}\tag{\ref{lem:type-b-regular-levi-orbits}.2}\label{eq:type-b-regular-levi-action}
 \widetilde L=\mathbf M^F\longrightarrow
 (\mathbf M/Z(\mathbf M))^F
 \cong\prod_i(\mathbf K_i/Z(\mathbf K_i))^F
\end{equation}
surjective. Conjugation by $(\mathbf K_i/Z(\mathbf K_i))^F$ induces
the full group of inner and diagonal automorphisms of
$\mathbf K_i^F\cong H_i$
(see, for example, \cite[Remark~1.5.12]{GeckMalle2020}). Hence $\Delta_i$
is this group, and the homomorphism $\widetilde L\longrightarrow\prod_i\Delta_i$
is surjective.

\textup{(2)} Products of characters identify $\IBr(L_0)$ with
$\prod_i\IBr(H_i)$.  The orbit description follows from
\textup{(1)}.

\textup{(3)} The normal inclusions follow from
$\mathbf H\lhd\mathbf L\lhd\mathbf M$, and
$[\widetilde L,\widetilde L]\leq[\mathbf M,\mathbf M]^F=L_0$.
Hence $\widetilde L/L_0$ is abelian.
Since $\sigma$ preserves $\mathbf L$, $[\mathbf L,\mathbf L]$ and
$Z(\widetilde{\mathbf G})$, and commutes with $F$, it also preserves
$L_0$, $L$ and $\widetilde L$.
\end{proof}

\Needspace{14\baselineskip}
\begin{lemma}\label{lem:type-b-char2-clifford}
Let $\ell$ be a prime and $N\leq H\leq G$ be finite groups. Assume that $N$ is
normal in $G$ and that the quotient group $G/N$ is abelian.
Let a finite group $E$ act on $G$ by automorphisms
stabilising $N$ and $H$.  For $\psi\in\IBr(H)$, let $\theta$ be an irreducible constituent of $\psi_N$.  Suppose that $\theta$ extends
to $G_\theta$ and that
\[
 (G\rtimes E)_\theta
 =G_\theta\rtimes E_\theta.
\]
Then
\[
 (G\rtimes E)_\psi
 =G_\psi\rtimes E_\psi.
\]
\end{lemma}

\begin{proof}
Set $I=H_\theta$, choose an extension
$\widehat\theta\in\IBr(G_\theta)$ of $\theta$, and let
$\eta\in\IBr(I\mid\theta)$ be the Clifford correspondent of $\psi$.
By Gallagher's theorem for Brauer characters, we have
$\eta=(\widehat\theta)_I\lambda$ for some
$\lambda\in\IBr(I/N)$, inflated to $I$.
The character $(\widehat\theta)_I$ is fixed by
$G_\theta$, as it is the restriction of a character of that
group.  Since $G/N$ is abelian, conjugation by
$G_\theta$ acts trivially on $I/N$ and therefore fixes
$\lambda$.  Thus $G_\theta$ fixes $\eta$ and, by induction
from $I$ to $H$, also fixes $\psi$.

Let $g\in(G\rtimes E)_\psi$.  The constituents of
$\psi_N$ form a single $H$-orbit, so there is $h\in H$ such that
$hg$ fixes $\theta$.  By hypothesis, write $hg=ae$ with
$a\in G_\theta$ and $e\in E_\theta$.
The elements $h$, $g$ and $a$ fix $\psi$, so $e$ fixes $\psi$.
Consequently $g=h^{-1}ae\in G_\psi E_\psi$.
The reverse inclusion is immediate.
\end{proof}

\subsection{The principal \texorpdfstring{$2$}{2}-block in rank at least three}
\label{subsec:type-b-two-principal}

Throughout this subsection, let $\mathbf G$ be a simply connected algebraic
group of type $\mathsf B_n$,
and let $F_0$ be a standard Frobenius endomorphism of $\mathbf G$
such that $\mathbf G^{F_0}=\operatorname{Spin}_{2n+1}(p)$.
Set $F=F_0^f$ and
$G=\mathbf G^F=\operatorname{Spin}_{2n+1}(q)$.  Choose a regular embedding
$\mathbf G\hookrightarrow\widetilde{\mathbf G}$ to which $F_0$ extends,
and denote this extension also by $F_0$.  Set
$\widetilde G=\widetilde{\mathbf G}^F$.
For the remainder of this section, write $B_0(Y)$ for the principal
$2$-block of a finite group $Y$. For
$\mathbf H\in\{\mathbf G,\widetilde{\mathbf G}\}$ and $u\in\mathbf H$,
write $A_{\mathbf H}(u)=C_{\mathbf H}(u)/C_{\mathbf H}(u)^\circ$ for the
component group of the centraliser of $u$.
For each unipotent element $u\in G$, let $\Gamma_u$ denote the generalised
Gelfand--Graev character of $G$ associated with $u$, and let $P_u$ be its
principal block component.

The following lemma uses Taylor's analysis of the Frobenius action
on the component groups of spin groups \cite[Section~2]{Taylor2013},
together with Taylor's equivariance formula for generalised Gelfand--Graev
characters \cite[Proposition~11.10]{Taylor2018Automorphisms}.

\begin{lemma}\label{lem:type-b-rational-field}
Let $G=\operatorname{Spin}_{2n+1}(q)$.
Then every field automorphism of $G$ induced by a power of $F_0$ fixes every
unipotent conjugacy class of $G$ and every character $\Gamma_u$.
\end{lemma}

\begin{proof}
Let $\mathcal C$ be a geometric unipotent class of $\mathbf G$.
Unipotent classes in type $\mathsf B$ are parametrised by the Jordan block
sizes in the natural orthogonal representation.  Since $F_0$ preserves
these sizes, $\mathcal C$ is $F_0$-stable.  By the discussion of spin groups in
\cite[Section~2]{Taylor2013}, choose
$u_0\in\mathcal C^{F_0}$ such that $F_0$ acts trivially on
$A=A_{\mathbf G}(u_0)$.  Then $F=F_0^f$ also acts trivially on $A$, so
\cite[Theorem~21.11]{MalleTesterman2011} identifies the $G$-classes in
$\mathcal C^F$ with the ordinary conjugacy classes of $A$.
This identification commutes with $F_0$, which acts trivially on $A$.
Hence every power of $F_0$ fixes every $G$-class in $\mathcal C^F$.

The equivariance of generalised Gelfand--Graev characters
\cite[Proposition~11.10]{Taylor2018Automorphisms} now shows that every
character $\Gamma_u$ is fixed by the field automorphisms.
\end{proof}

\begin{lemma}\label{lem:type-b-principal-series}
Let $H=\operatorname{Spin}_{2m+1}(q_0)$, where $m\geq2$ and $q_0$ is
an odd prime power. Then
\[
 \Irr(B_0(H))=\mathcal E_2(H,1).
\]
Moreover, every irreducible character of $H$ whose semisimple label is
quasi-isolated belongs to $B_0(H)$, and $B_0(H)$ is the only strictly
quasi-isolated $2$-block of $H$.
\end{lemma}

\begin{proof}
The equality is \cite[Theorem~21.14]{CabanesEnguehard2004}.
The dual algebraic group is adjoint of type $\mathsf C_m$, and its
quasi-isolated semisimple elements have order dividing $4$ by
\cite[Proposition~5.3(a)]{Bonnafe2005}.
Hence every rational Lusztig series with such a label is contained
in $\mathcal E_2(H,1)$, proving the assertion about ordinary characters.
If a $2$-block $b$ is strictly quasi-isolated, its strictly
quasi-isolated semisimple label $s$ has odd order. Since $s$ is also
quasi-isolated, its order divides $4$, so $s=1$.
The equality then gives $b=B_0(H)$.
Conversely, the identity is strictly quasi-isolated, so $B_0(H)$ is
strictly quasi-isolated.
\end{proof}

For a Weyl group $W$ and $E\in\Irr(W)$, let $b(E)$ be the
least nonnegative integer $i$ for which $E$ occurs in the $i$th symmetric
power of the reflection representation \cite[4.1.1]{GeckMalle2020}.
Truncated induction of Weyl group characters, also called
$j$-induction, is denoted by $j_{W'}^W$
\cite[p.~476]{Taylor2014Maximal}.

\begin{proposition}\label{prop:type-b-gggr-rank}
Let $G=\operatorname{Spin}_{2n+1}(q)$, and let $\mathcal U$ be an ordered set
of representatives of the unipotent conjugacy classes of $G$. Assume that
representatives in the same geometric class are consecutive and that the
geometric classes are ordered by nondecreasing dimension. Then there are
characters $\eta_u\in\Irr(B_0(G))$, indexed by $u\in\mathcal U$, such that the
matrix
\[
 (\langle\eta_u,P_v\rangle_G)_{u,v\in\mathcal U}
\]
is block lower triangular, with one diagonal block for each geometric class,
and every diagonal block is a permutation matrix. Moreover, the $P_u$ form a
basis of the $\mathbb Q$-span of the projective indecomposable characters of
$B_0(G)$.
\end{proposition}

\begin{proof}
Fix a geometric unipotent class $\mathcal C$. As in the proof of
Lemma~\ref{lem:type-b-rational-field}, choose $u_0\in\mathcal C^F$ such
that $F$ acts trivially on $A=A_{\mathbf G}(u_0)$.
Let $Z_{\mathcal C}$ be the image of $Z(\mathbf G)$ in $A$. The regular
embedding gives
\[
 V=A/Z_{\mathcal C}\cong A_{\widetilde{\mathbf G}}(u_0).
\]
By \cite[Section~2]{Taylor2013}, $V$ is an elementary abelian $2$-group.
Since $F$ acts trivially on $A$, it also acts trivially on $V$.
The parametrisation of rational classes identifies the map from
$G$-classes to $\widetilde G$-classes in $\mathcal C^F$ with the map
from conjugacy classes of $A$ to elements of $V$.
For each representative $\widetilde u$ of a $\widetilde G$-class in
$\mathcal C^F$, let $\widetilde\Gamma_{\widetilde u}$ denote the
corresponding generalised Gelfand--Graev character.

We first choose an isolated family for
$\widetilde{\mathbf G}$. By
\cite[Proposition~2.4]{GeckHezard2008}, there is a geometric isolated
pair associated with $\mathcal C$ whose family group has order $|V|$.
Let $\mathbf D=[\widetilde{\mathbf G}^*,\widetilde{\mathbf G}^*]
\cong\operatorname{Sp}_{2n}$. Write the semisimple element in the
chosen pair as $zt$, where $z$ is central and $t\in\mathbf D$.
Multiplication by $z$ leaves the centraliser unchanged, so $t$
is isolated. Its eigenvalues belong to $\{1,-1\}$, since any
other reciprocal pair of eigenvalues would give a general linear
factor and place the centraliser in a proper Levi subgroup.
Let $2a$ and $2b$ be the dimensions of its $1$ and $-1$
eigenspaces, respectively, where $a+b=n$.

Choose an $F$-stable maximal torus $\widetilde{\mathbf T}^*$
contained in an $F$-stable Borel subgroup, and set
$\mathbf T_{\mathbf D}=\widetilde{\mathbf T}^*\cap\mathbf D$.
Relative to a symplectic basis diagonalising
$\mathbf T_{\mathbf D}$, let
$t'=\operatorname{diag}(I_a,-I_b,I_a,-I_b)$.
Its centraliser in $\mathbf D$ is
$\operatorname{Sp}_{2a}\times\operatorname{Sp}_{2b}$, with root
system $\mathsf C_a\sqcup\mathsf C_b$, omitting factors of rank zero.
The eigenspace dimensions show that $t'$ is conjugate to $t$
in $\mathbf D$ and is therefore isolated.
Since $\mathsf C_n$ has no nontrivial graph automorphism, $F$
acts by the $q$-power map on $\mathbf T_{\mathbf D}$.
Thus $t'$ is fixed by $F$ for odd $q$.

Let $W$ be the Weyl group of $\widetilde{\mathbf G}^*$, and let $W_t\leq W$
be the Weyl group of $C_{\widetilde{\mathbf G}^*}(t)$.
Removing the central factor and replacing $t$ by $t'$ replace $W_t$
by a conjugate in $W$. Apply this conjugacy to the family and its
unique special irreducible character. Conjugating $W_t$ and its
character leaves induction from $W_t$ to $W$ unchanged and preserves
the $b$-invariant. Hence the constituent selected by $j$-induction
is unchanged. The associated unipotent class is the class corresponding
to this constituent under the Springer correspondence
\cite[p.~821]{GeckHezard2008}, so it remains $\mathcal C$.
The order of the group attached to the resulting family is also unchanged.

Set $\widetilde s'=t'$, and denote the resulting family by $\mathcal F'$.
Thus $\widetilde s'\in(\widetilde{\mathbf T}^*)^F$ is isolated and
$|\mathcal G_{\mathcal F'}|=|V|$. The group $\mathcal G_{\mathcal F'}$
is elementary abelian \cite[Remark~4.2.17]{GeckMalle2020}.
Lemma~\ref{lem:type-b-isolated-coverage} gives characters
$\widetilde\rho_{\widetilde u}$ belonging to the family in
$\mathcal E(\widetilde G,\widetilde s')$ corresponding to $\mathcal F'$
such that
\[
 \langle\widetilde\rho_{\widetilde u}^{*},
       \widetilde\Gamma_{\widetilde v}\rangle_{\widetilde G}
 =\delta_{\widetilde u,\widetilde v}
\]
for all chosen representatives $\widetilde u,\widetilde v$.

Independently of this family, choose a pair $(\widetilde s,\psi)$
associated with $\mathcal C$ as in \cite[Theorem~3.2]{Taylor2013},
with the image $s$ of $\widetilde s$ in $\mathbf G^*$ quasi-isolated.
The character $\psi$ is unipotent on $C^\circ_{\mathbf G^*}(s)^F$.
Let $\widetilde\psi\in\mathcal E(\widetilde G,\widetilde s)$ be the
character corresponding to $\psi$ under the bijections in
\cite[(3.2)]{Taylor2013}.
Property~(P3) of \cite[Theorem~3.2]{Taylor2013} shows that
$\widetilde\psi$ has unipotent support $\mathcal C$.
Property~(P1) of the same theorem and the degree calculation in the
proof of \cite[Theorem~3.1]{Taylor2013} give
$n_{\widetilde\psi}=n_\psi=|V|$.
Also $n_{\widetilde\psi^*}=n_{\widetilde\psi}$ by
\cite[Proposition~3.4.21]{GeckMalle2020}.
For each chosen representative $\widetilde u$ of a $\widetilde G$-class
in $\mathcal C^F$, the Frobenius action on its component group differs
from that at $u_0$ by an inner automorphism of $V$. Since $V$ is abelian,
$[A_{\widetilde{\mathbf G}}(\widetilde u):
A_{\widetilde{\mathbf G}}(\widetilde u)^F]=1$.
Choose the Frobenius structure on the constant local system on
$\mathcal C$ for which the trace function takes the value $1$ at every
point of $\mathcal C^F$.
The wave front set
of $\widetilde\psi^*$ is $\mathcal C$ by \cite[Lemma~14.15]{Taylor2016}.
Thus \cite[Proposition~15.4]{Taylor2016}, with the expansion in
\cite[(11.15)]{Taylor2016}, gives
\[
 \sum_{\widetilde u}
 \langle\widetilde\psi^*,\widetilde\Gamma_{\widetilde u}
 \rangle_{\widetilde G}
 =\frac{|A_{\widetilde{\mathbf G}}(u_0)|}{n_{\widetilde\psi^*}}=1,
\]
where the sum ranges over the chosen representatives.
Each summand is a nonnegative integer, so exactly one scalar product
is $1$ and all the others vanish.

Every irreducible character of $\widetilde G$ has multiplicity free
restriction to $G$ \cite[Theorem~1.7.15]{GeckMalle2020}.
For $\widetilde\rho=\widetilde\psi$ or one of the characters
$\widetilde\rho_{\widetilde u}$ just chosen, write
\[
 \operatorname{Res}^{\widetilde G}_G(\widetilde\rho)
 =\rho_1+\cdots+\rho_r.
\]
Alvis--Curtis duality commutes with restriction from $\widetilde G$ to
$G$ and is an involution
\cite[Proposition~3.4.3 and Corollary~3.4.5]{GeckMalle2020}.
Consequently, $\operatorname{Res}^{\widetilde G}_G(\widetilde\rho^*)$
is a sum of the distinct irreducible characters $\rho_i^*$ with
coefficients in $\{1,-1\}$. Since this restriction is a character,
every coefficient is $1$. Together with the identity
$\widetilde\Gamma_v=\operatorname{Ind}_G^{\widetilde G}\Gamma_v$
in the proof of \cite[Lemma~14.12]{Taylor2016} and Frobenius
reciprocity, this gives
\[
 \sum_{i=1}^r\langle\rho_i^*,\Gamma_v\rangle_G
 =\langle\widetilde\rho^*,\widetilde\Gamma_v\rangle_{\widetilde G}
 \qquad(v\in\mathcal C^F).
 \tag{\ref{prop:type-b-gggr-rank}.1}
 \label{eq:type-b-fibre-clifford}
\]
For these selected characters, the right hand side belongs to
$\{0,1\}$ by their construction above. We now choose characters $\rho_u$,
indexed by $u\in\mathcal U\cap\mathcal C^F$, such that
$\bigl(\langle\rho_u^*,\Gamma_v\rangle_G\bigr)_{u,v\in\mathcal U\cap\mathcal C^F}$
is a permutation matrix.

Whenever a $\widetilde G$-class in $\mathcal C^F$ is a single
$G$-class, restrict the character $\widetilde\rho_{\widetilde u}$
chosen for that class. Conjugation by $\widetilde G$ acts transitively
on the restriction constituents and fixes the $G$-class.
The equivariance of generalised Gelfand--Graev characters
\cite[Proposition~2.2]{Geck1993} makes the summands in
\eqref{eq:type-b-fibre-clifford} equal for $v=\widetilde u$.
They are nonnegative integers with sum one, so $r=1$.
Equation~\eqref{eq:type-b-fibre-clifford} also shows that zero scalar
products in $\widetilde G$ remain zero for every restriction constituent.

If $Z_{\mathcal C}=1$, then $A\cong V$ and every
$\widetilde G$-class in $\mathcal C^F$ is a single $G$-class.
The irreducible restrictions just obtained therefore give the required
characters $\rho_u$, indexed by $u\in\mathcal U\cap\mathcal C^F$,
for this geometric class.

Suppose now that $Z_{\mathcal C}\ne1$, and denote its nonidentity
element by $\vartheta$. If $A$ is abelian, Lusztig's presentation of
$A$, as recalled in \cite[pp.~45--46]{Taylor2013}, gives
$A=Z_{\mathcal C}\cong C_2$. Indeed, nontrivial central image requires
each odd part of the orthogonal partition to occur once. The number
of odd parts is odd. If there were at least three, the generators
$y_2,y_3$ in that presentation would have nontrivial commutator
$\vartheta$, contradicting commutativity. Thus there is only one odd
part and $A$ is generated by $\vartheta$. In this case $V=1$, and the
unique $\widetilde G$-class contains two $G$-classes.

If $A$ is nonabelian, Lusztig's presentation recalled in
\cite[pp.~45--46]{Taylor2013} shows that $A$ is an extraspecial
$2$-group, with
$Z(A)=Z_{\mathcal C}=\langle\vartheta\rangle$.
The $\widetilde G$-class corresponding to $0\in V$ contains the two
$G$-classes corresponding to $1$ and $\vartheta$.
Every other $\widetilde G$-class is a single $G$-class, since the
inverse image of a nonzero element of $V$ is a conjugacy class
$\{x,x\vartheta\}$ of $A$.

In both cases with $Z_{\mathcal C}\ne1$, there is exactly one $\widetilde
G$-class containing two $G$-classes, and the required characters have already
been chosen for every other $\widetilde G$-class. Property~(P2) of
\cite[Theorem~3.2]{Taylor2013} and the restriction
formula in the proof of \cite[Theorem~3.1]{Taylor2013} give
exactly $|Z_{\mathcal C}|=2$ irreducible constituents of
$\operatorname{Res}^{\widetilde G}_G\widetilde\psi$.
On a $\widetilde G$-class consisting of one $G$-class, conjugation
makes the two summands in \eqref{eq:type-b-fibre-clifford} equal.
They are nonnegative integers, so their sum cannot be one.
Hence the unique nonzero scalar product of $\widetilde\psi^*$ occurs
on the class containing two $G$-classes.
Let $u_1,u_2$ represent those classes, and write
$\operatorname{Res}^{\widetilde G}_G\widetilde\psi=\rho_1+\rho_2$.
Equation~\eqref{eq:type-b-fibre-clifford} gives
\[
 \langle\rho_1^*,\Gamma_{u_j}\rangle_G+
 \langle\rho_2^*,\Gamma_{u_j}\rangle_G=1
 \qquad(j=1,2).
\]
An element interchanging the two constituents must interchange the
two $G$-classes. Otherwise equivariance would make the two entries
in either column equal, although they are nonnegative integers with
sum one. The scalar product matrix on these classes is consequently
\[
 \begin{pmatrix}a&b\\ b&a\end{pmatrix},
 \qquad a,b\in\mathbb Z_{\geq0},\qquad a+b=1,
\]
and is a permutation matrix. Together with the characters chosen for
the $\widetilde G$-classes containing a single $G$-class, this gives characters $\rho_u$ for which
$\bigl(\langle\rho_u^*,\Gamma_v\rangle_G\bigr)_{u,v\in\mathcal U\cap\mathcal C^F}$
is a permutation matrix.

All the selected characters of $\widetilde G$ have unipotent support
$\mathcal C$. Compatibility with duality and
\cite[Lemmas~14.12 and~14.15]{Taylor2016} show that each restriction
constituent $\rho_u$ has the same support.
Its semisimple label is the image of $\widetilde s$ or
$\widetilde s'$ in $\mathbf G^*$, according to the character selected.
The image of $\widetilde s$ is quasi-isolated by its choice. The image of $\widetilde s'$ is isolated by
\cite[Proposition~2.3(b)]{Bonnafe2005}.
Both images are therefore quasi-isolated $2$-elements
\cite[Proposition~5.3(a)]{Bonnafe2005}.

For every $u\in\mathcal U$, set $\eta_u=\rho_u^*$.
The semisimple label of $\rho_u$ is quasi-isolated, so
Lemma~\ref{lem:type-b-principal-series} gives
$\rho_u\in\Irr(B_0(G))$.
Alvis--Curtis duality preserves rational Lusztig series
\cite[Proposition~9.8(iv)]{CabanesEnguehard2004}, and therefore
$\eta_u$ lies in the same series and also belongs to $B_0(G)$. Hence
$\langle\eta_u,P_v\rangle_G=\langle\eta_u,\Gamma_v\rangle_G$ for all
$u,v\in\mathcal U$.

Let $\mathcal C_u$ be the geometric class containing $u$.
By \cite[Lemma~14.15]{Taylor2016}, $\mathcal C_u$ is the wave front
set of $\eta_u$. The result of Achar and Aubert in the form of
\cite[Proposition~15.2]{Taylor2016} shows that
$\langle\eta_u,\Gamma_v\rangle_G\neq0$ only if
$\mathcal C_v\subseteq\overline{\mathcal C_u}$, where $\overline{\mathcal
C_u}$ denotes the Zariski closure of $\mathcal C_u$ in $\mathbf G$. For
distinct classes, this inclusion implies $\dim\mathcal C_v<\dim\mathcal C_u$.
The chosen ordering therefore makes the full scalar product matrix block lower
triangular, with the permutation matrices obtained above on its diagonal.

By the equality $\mathcal E_2(G,1)=\Irr(B_0(G))$ in
Lemma~\ref{lem:type-b-principal-series},
\cite[Proposition~2.7]{Chaneb2021} gives
$|\mathcal U|=|\IBr(B_0(G))|$.
Each $\Gamma_u$ is induced from a linear character of a $p$-subgroup.
Since $p$ is odd, $P_u$ is projective in characteristic $2$.
The scalar product matrix is nonsingular, so the $P_u$ are linearly
independent. Their number equals the dimension of the $\mathbb Q$-span
of the projective indecomposable characters of $B_0(G)$, and they
therefore form a basis of that space.
\end{proof}

\begin{corollary}\label{cor:type-b-principal-selector}
Let $G=\operatorname{Spin}_{2n+1}(q)$, and let $E$ be the group of field
automorphisms of $G$. Then every $\psi\in\IBr(B_0(G))$ is fixed by $E$, satisfies
\[
 (\widetilde G\rtimes E)_\psi=\widetilde G_\psi\rtimes E,
\]
and extends to $G\rtimes E$. If $(n,q)\neq(3,3)$, then the principal
$2$-block of $\Omega_{2n+1}(q)$ is BAW-good.
\end{corollary}

\begin{proof}
By Lemma~\ref{lem:type-b-rational-field}, $E$ fixes every $\Gamma_u$
and therefore every principal block component $P_u$.
Proposition~\ref{prop:type-b-gggr-rank} shows that the $P_u$ form a
basis over $\mathbb Q$ of the space spanned by the projective
indecomposable characters of $B_0(G)$.
Thus $E$ acts trivially on this space. It fixes every projective
indecomposable character and, by duality, every irreducible Brauer
character of $B_0(G)$. In particular, $E$ fixes every
$\widetilde G$-conjugate of $\psi$, which gives the stabiliser equality.
Since $E$ is cyclic, Lemma~\ref{lem:cyclic-extension} extends $\psi$
to $G\rtimes E$.

Suppose that $(n,q)\neq(3,3)$. By Table~\ref{tab:typeb-covers},
$S=\Omega_{2n+1}(q)$ is its own universal $2'$-covering group.
Set $H=\operatorname{SO}_{2n+1}(q)$.
Since $Z(G)$ is a central $2$-group, inflation identifies
$\IBr(B_0(S))$ with $\IBr(B_0(G))$
\cite[Theorem~9.10]{Navarro1998}. Thus field automorphisms
fix every character in $\IBr(B_0(S))$.
The principal block of $H$ covers $B_0(S)$.
As $|H:S|=2$, each irreducible Brauer character of $S$ lies
under a unique irreducible Brauer character of $H$
\cite[Theorem~8.11]{Navarro1998}.
Consequently field automorphisms also fix every character in
$\IBr(B_0(H))$.
This gives the instance of
\cite[Corollary~5.15]{FengYuZhang2024} used in the proof of
Theorem~5.16 of that paper.
The preceding argument verifies the field invariance hypothesis of
\cite[Theorem~5.16]{FengYuZhang2024} and Assumption~5.1 of the same paper
for $B_0(G)$. That theorem therefore gives the inductive BAW condition
for the principal $2$-block of $S$, which is BAW-good by
Theorem~\ref{thm:spath-triples}.
\end{proof}

By the equivalence stated after
\cite[Question~5.27]{FengYuZhang2024},
Corollary~\ref{cor:type-b-principal-selector} gives a positive answer
for $\operatorname{Spin}_{2n+1}(q)$ with $q$ odd and $n\geq3$.
Indeed, field automorphisms fix both the unipotent conjugacy classes and
the principal block Brauer characters, so only the diagonal involution
acts on these sets. Applying \cite[Proposition~2.7]{Chaneb2021} to
$\operatorname{Spin}_{2n+1}(q)$ and $\operatorname{SO}_{2n+1}(q)$,
together with \cite[Theorem~21.14]{CabanesEnguehard2004} and
\cite[Theorem~8.11]{Navarro1998}, gives equal numbers of diagonal
orbits of length two. For $n=3$, the field invariance is already obtained
in the proof of \cite[Corollary~5.17]{FengYuZhang2024}.

\subsection{The exceptional triple cover}
\label{subsec:type-b-rank-three}

For odd $q\ne3$, Feng, Yu, and Zhang established the inductive BAW
condition for the principal $2$-block of $\Omega_7(q)$ in the proof of
\cite[Corollary~5.17]{FengYuZhang2024}.  This also follows from
Corollary~\ref{cor:type-b-principal-selector}.  We now treat the exceptional
universal $2'$-covering group $3.\Omega_7(3)$.

\begin{proposition}\label{prop:type-b-q3}
Let $X=3.\Omega_7(3)$. Then every $2$-block of $X$ satisfies the inductive BAW
condition.
\end{proposition}

\begin{proof}
Let $S=\Omega_7(3)$ and $H=\operatorname{SO}_7(3)$.  Since $X=X_2(S)$
and $\Out(S)\cong C_2$, \cite[Corollary~2.13]{FengLiZhang2019} shows
that a $2$-block $B$ of $X$ satisfies the inductive BAW condition whenever
there is an $\Aut(X)_B$-equivariant bijection
$\IBr(B)\to\Alp(B)$.  The calculations in
Section~\ref{app:type-b-computation} give the nine blocks
$B_1,\ldots,B_9$, with $B_1=B_0(X)$, and the data in
Table~\ref{tab:typeb-q3-blocks} below.

Conjugation by $H$ induces $\Aut(S)$.
By \cite[Lemma~5.7 and the proof of Proposition~5.11]{FengYuZhang2024},
the numbers of principal weight classes of $H$ and $S$ equal
$|\IBr(B_0(H))|$ and $|\IBr(B_0(S))|$, respectively.
The proof of \cite[Proposition~5.11]{FengYuZhang2024} applies to
$\operatorname{Spin}_7(3)$, whose principal
Brauer characters and principal weight classes correspond bijectively
to those of $S$ under the central quotient with kernel of order $2$.
The calculation in Section~\ref{app:type-b-computation} gives
$|\IBr(B_0(H))|=10$ and $|\IBr(B_0(S))|=12$.
The proof of \cite[Proposition~5.11]{FengYuZhang2024} also shows that
every principal weight class of $S$ is covered by a unique principal
weight class of $H$.
By \cite[Corollary~2.22(2)]{FengYuZhang2024}, the classes of $S$
covered by each class of $H$ form one orbit under $H/S$.
Consequently, $H/S\cong C_2$ has eight singleton orbits and two
orbits of size two on the principal weight classes of $S$.
The calculation in \path{o7brauer.g}, described in
Section~\ref{app:type-b-computation}, shows that eight of the ten
irreducible Brauer characters of $B_0(H)$ restrict irreducibly to $S$,
while the other two each restrict to a sum of two distinct irreducible
Brauer characters. The irreducible constituents of these ten
restrictions are precisely the characters in $\IBr(B_0(S))$.
By Clifford theory, the action of $H/S$ on this set has eight singleton
orbits and two orbits of size two.  Comparing this with the action on
$\Alp(B_0(S))$, we obtain an $\Aut(S)$-equivariant bijection
\[
 \IBr(B_0(S))\longrightarrow\Alp(B_0(S)).
\]
Since $X\to S$ has central
$2'$-kernel, \cite[Lemma~2.5]{FengLiZhang2019} lifts this to an
$\Aut(X)$-equivariant bijection $\IBr(B_1)\to\Alp(B_1)$.

For $B_2$, the character table calculation in
Section~\ref{app:type-b-computation} gives defect order $8$, five ordinary
characters and two Brauer characters. One ordinary character has
height $1$, so Kessar--Malle's theorem
\cite[Theorem~1.1]{KessarMalle2013} excludes abelian defect groups.
The defect group of $B_2$ is therefore dihedral or quaternion.
A block with quaternion defect group of order $8$ has one or three
irreducible Brauer characters
\cite[Theorem~8.1]{Sambale2014}.
Hence the defect group of $B_2$ is $D_8$.  By the proof of
\cite[Theorem~4.1]{Sambale2012}, $B_2$ has exactly two $X$-conjugacy
classes of weights, represented by weights whose radical subgroups
have orders $8$ and $4$, respectively.  Since automorphisms preserve
subgroup order, the outer involution fixes both classes.  The program
\path{o7block2.g} verifies that the outer involution also fixes
both irreducible Brauer characters of $B_2$.  Hence any bijection
$\IBr(B_2)\to\Alp(B_2)$ is $\Aut(X)_{B_2}$-equivariant.
The block $B_3$ has defect group $C_2$ and is therefore nilpotent. It satisfies the inductive BAW condition by
\cite[Theorem~1.3]{KoshitaniSpath2016}. The blocks $B_4$ and $B_5$
have defect zero and satisfy the condition trivially.

Table~\ref{tab:typeb-q3-blocks} below shows that $B_8$ and $B_9$
are the only blocks of defect $3$ with nontrivial central
characters. They are therefore the two faithful blocks with
dihedral defect groups of order $8$ described in
\cite[Section~5.5]{Macgregor2022}.  Each has two irreducible
Brauer characters, so \cite[Theorem~4.1]{Sambale2012} gives two
weight classes in each.
The exhaustive radical computation in
Section~\ref{app:type-b-computation} gives eight weight classes in
each faithful sector.  Since $B_8$ and $B_9$ account for two
in their respective sectors, $B_6$ and $B_7$ each have six, equal to
their numbers of Brauer characters.  The outer involution swaps
$B_6$ with $B_7$ and $B_8$ with $B_9$, so
$\Aut(X)_{B_i}=\Inn(X)$ for $i\in\{6,7,8,9\}$, so any bijection
$\IBr(B_i)\to\Alp(B_i)$ is $\Aut(X)_{B_i}$-equivariant for each such $i$.
The criterion of \cite[Corollary~2.13]{FengLiZhang2019} now proves the
inductive BAW condition for
$B_1,B_2,B_6,B_7,B_8$ and $B_9$.
\end{proof}

In Table~\ref{tab:typeb-q3-blocks} below, let $\omega$ be the nontrivial
linear character of $Z(X)\cong C_3$ over which $B_6$ lies.
Its square $\omega^2$ is the other nontrivial linear character.
For each block $B$, let $D_B$ be a defect group of $B$.

\begin{table}[ht]
\centering
\scriptsize
\caption{The $2$-blocks of $3.\Omega_7(3)$ in the ordering of
CTblLib~1.3.11 \cite{CTblLib2025}.}
\label{tab:typeb-q3-blocks}
\begin{tabular}{@{}cccll@{}}
\toprule
$B$ & $|D_B|$ & $|\IBr(B)|$ & central character & outer action \\
\midrule
$B_1$ & $2^9$ & 12 & $1$ & fixed \\
$B_2$ & $2^3$ & 2 & $1$ & fixed \\
$B_3$ & $2$ & 1 & $1$ & fixed \\
$B_4$ & $1$ & 1 & $1$ & fixed \\
$B_5$ & $1$ & 1 & $1$ & fixed \\
$B_6$ & $2^9$ & 6 & $\omega$ & $B_7$ \\
$B_7$ & $2^9$ & 6 & $\omega^2$ & $B_6$ \\
$B_8$ & $2^3$ & 2 & $\omega$ & $B_9$ \\
$B_9$ & $2^3$ & 2 & $\omega^2$ & $B_8$ \\
\bottomrule
\end{tabular}
\end{table}

\Needspace{12\baselineskip}
\subsection{All blocks at the prime \texorpdfstring{$2$}{2}}
\label{subsec:type-b-two-high-rank}

The following lemma uses the reduction in the proof of
\cite[Theorem~1]{FengYuZhang2024}, specialised to type~$\mathsf B$.
The proof includes the lower rank cases and uses
Lemma~\ref{lem:type-b-normal-core} to pass from the orthogonal
quotients to the spin groups.

\begin{lemma}
\label{lem:type-b-localized-return}
Let $\mathbf G$ be a simple, simply connected algebraic group of type
$\mathsf B_n$, where $n\geq3$, and let $F$ be a split Frobenius
endomorphism such that $\mathbf G^F=\operatorname{Spin}_{2n+1}(q)$,
with $(n,q)\neq(3,3)$.
Suppose that \cite[Assumption~5.3]{FengLiZhang2022Jordan} holds for
$(\mathbf G,F)$ at $\ell=2$. Then every $2$-block of
$\Omega_{2n+1}(q)$ is BAW-good.
\end{lemma}

\begin{proof}
Set $G=\mathbf G^F$ and $S=\Omega_{2n+1}(q)$. By
Section~\ref{subsec:type-b-covers}, $G$ is the universal covering group of $S$.
We apply Lemma~\ref{lem:jordan-reduction} with $\ell=2$.
The hypothesis gives the character representatives, stabiliser
factorisation and extensions required in part~\textup{(i)} of
Lemma~\ref{lem:jordan-reduction}.

Let $\mathbf K$ and $F'$ be as in
Lemma~\ref{lem:jordan-reduction}\textup{(ii)}. A connected Dynkin
subdiagram of $\mathsf B_n$ has type $\mathsf A_m$ or $\mathsf B_m$.
If $\mathbf K$ has type~$\mathsf A_m$, then $\mathbf K^{F'}$ is a
special linear or unitary group. The consequence of the type~$\mathsf A$ results stated after
Lemma~\ref{lem:jordan-reduction} gives the required iBAW bijection
for every $2$-block of $\mathbf K^{F'}$. This also covers $\mathsf B_1=\mathsf A_1$.

Suppose that $\mathbf K$ has type $\mathsf B_m$, where $m\geq2$.
Then $H=\mathbf K^{F'}\cong\operatorname{Spin}_{2m+1}(q_0)$ for some
power $q_0$ of $p$. By Lemma~\ref{lem:type-b-principal-series},
the only strictly quasi-isolated $2$-block of $H$ is its principal block.
For $m=2$, we have $\operatorname{Spin}_5(q_0)\cong\Sp_4(q_0)$.
The rank two construction at the start of the proof of
Proposition~\ref{prop:odd-two} gives the required iBAW bijection
for its principal block.

For $m\geq3$, set $T=\Omega_{2m+1}(q_0)$.
If $(m,q_0)\neq(3,3)$, Corollary~\ref{cor:type-b-principal-selector}
shows that the principal $2$-block of $T$ is BAW-good and therefore
admits an iBAW bijection. In the exceptional case, Proposition~\ref{prop:type-b-q3}
and Remark~\ref{rem:block-aggregation}\textup{(a)} show that $T$ satisfies iBAW at
$2$, so Corollary~\ref{cor:fixed-point-descent} gives the required
principal block bijection on $T$.
We have $Z(H)=O_2(H)$ and $H/Z(H)\cong T$.
By Section~\ref{subsec:type-b-covers}, $H$ is the universal covering group of
$T$ unless $(m,q_0)=(3,3)$, in which case
$H=\operatorname{Spin}_7(3)=X_3(T)$. Thus the natural isomorphism
$\Aut(H)\cong\Aut(T)$ identifies $(H\rtimes\Aut(H))/Z(H)$ with
$T\rtimes\Aut(T)$.
Lemma~\ref{lem:type-b-normal-core} therefore lifts this bijection to
an iBAW bijection for the principal $2$-block of $H$.

This verifies the second hypothesis of Lemma~\ref{lem:jordan-reduction},
which shows that every $2$-block of $G$ is BAW-good.
Since $G/O_2(Z(G))=S$, every $2$-block of $S$ is dominated by a
$2$-block of $G$ and is consequently BAW-good.
\end{proof}

\begin{proposition}\label{prop:type-b-brauer-hypothesis}
Let $\mathbf G$ be a simple, simply connected algebraic group of type
$\mathsf B_n$, where $n\geq3$, in odd characteristic $p$.
Let $\phi$ be a standard Frobenius endomorphism of $\mathbf G$
such that $\mathbf G^\phi=\operatorname{Spin}_{2n+1}(p)$.
Let $f\geq1$, and set $q=p^f$, $F=\phi^f$ and $G=\mathbf G^F$.
Choose a regular embedding
$\mathbf G\hookrightarrow\widetilde{\mathbf G}$ to which $\phi$
extends, set $\widetilde G=\widetilde{\mathbf G}^F$, and let $E$ be
the cyclic group generated by the restriction of $\phi$ to
$\widetilde G$.
Then, at the prime $2$, every $\widetilde G$-orbit in $\IBr(G)$ contains
a character $\psi$ such that
\[
 (\widetilde G\rtimes E)_\psi
   =\widetilde G_\psi\rtimes E_\psi
\]
and $\psi$ extends to $G\rtimes E_\psi$.
\end{proposition}

\begin{proof}
If $f=1$, then $E=1$ and both assertions are immediate.
Assume that $f>1$.
Corollary~\ref{cor:type-b-principal-selector} supplies the required
representatives for the principal block of $G$.

Let $b$ be a nonprincipal $2$-block of $G$.
We use the idempotents $e_s^G$ defined in Section~\ref{sec:symplectic}.
By \cite[Theorem~9.12]{CabanesEnguehard2004}, choose a semisimple
element $s\in(\mathbf G^*)^F$ of odd order such that $b$ belongs to
$e_s^G$.
Since $b$ is nonprincipal, Lemma~\ref{lem:type-b-principal-series}
shows that $s$ is not strictly quasi-isolated.

As in the proof of \cite[Theorem~5.7]{FengLiZhang2022Jordan}, choose
a minimal Levi subgroup $\mathbf L^*$ of $\mathbf G^*$ containing
$C_{\mathbf G^*}^{\circ}(s)C_{\mathbf G^*}(s)^F$.
This product contains a maximal torus, and intersections of Levi
subgroups containing a common maximal torus are Levi subgroups.
Hence $\mathbf L^*$ is unique and therefore $F$-stable. It is proper
because $s$ is not strictly quasi-isolated.

The action of $\widetilde G$ fixes $e_s^G$ by
\cite[Lemma~7.4]{BonnafeDatRouquier2017}. Set $D=E_{e_s^G}$.
Choose a positive divisor $c$ of $f$ such that $F_0=\phi^c$ induces
a generator of $D$, and set $h=f/c$, so that $F=F_0^h$.
By \cite[Lemma~4.5(a)--(b)]{Ruhstorfer2022Derived}, there is a Frobenius
endomorphism conjugate to $F_0$ by an element of $G$ which stabilises
a Levi subgroup $\mathbf L$ in duality with $\mathbf L^*$ and its
idempotent $e_s^L$. Conjugating this Levi subgroup and making the
corresponding rational conjugation of $(s,\mathbf L^*)$, we may therefore
assume that $F_0$ itself stabilises $\mathbf L$ and $e_s^L$,
where $L=\mathbf L^F$. Set
\[
 L_0=[\mathbf L,\mathbf L]^F,\qquad
 \widetilde L=(\mathbf LZ(\widetilde{\mathbf G}))^F.
\]

Write $L_0\cong\prod_i H_i$ as in
Lemma~\ref{lem:type-b-regular-levi-orbits}.
The factors arising from components of type $\mathsf A$ are special
linear or unitary groups over extension fields.
There is at most one component containing a short root.
If such a component occurs, then $F_0$ preserves it because $F_0$
preserves root lengths. Its Dynkin diagram has no nontrivial
automorphism preserving root lengths. The group of points fixed by
$F$ on this component is therefore $\operatorname{Spin}_{2m+1}(q)$
for some $m<n$.
If $d$ is a block of $L$ belonging to $e_s^L$ and $d_0$ is a block
of $L_0$ covered by $d$, then $d_0$ is strictly quasi-isolated by
the proof of \cite[Theorem~5.7]{FengLiZhang2022Jordan}.
Write $d_0=d_1\otimes\cdots\otimes d_r$, where $d_i$ is a block of
$H_i$. The proof of \cite[Proposition~5.6]{FengLiZhang2022Jordan}
shows that each $d_i$ is strictly quasi-isolated.
If $H_i$ has type $\mathsf B_m$ with $m\geq2$, then $d_i$ is
principal by Lemma~\ref{lem:type-b-principal-series}.

For a type $\mathsf A$ factor, the representatives and stabiliser
factorisation for the diagonal, field and graph automorphisms are
provided by \cite[Theorem~8.1]{FengLiZhang2021Equivariant}.
This includes $\mathsf B_1=\mathsf A_1$ and the group
$\operatorname{SL}_2(3)$.
For a factor $H_i$ of type $\mathsf B_m$, let $\Delta_i$ be the
group induced by $\widetilde L$ on $H_i$, as in
Lemma~\ref{lem:type-b-regular-levi-orbits}, and let $E_i$ denote
the group of field automorphisms of $H_i$.
For $m=2$, identify $H_i$ with $\Sp_4(q)$.
Here $f>1$, so $q\geq9$.
The verification following Lemma~4.5 of Brough and Schaeffer Fry
\cite[p.~1203]{BroughSchaefferFry2020} gives a representative in each
diagonal orbit with the required stabiliser factorisation.
Here $\Delta_i/\Inn(H_i)$ has order two, so field automorphisms
centralise this quotient and the factorisation is preserved by
diagonal conjugation. Thus
\begin{equation}\label{eq:type-b-factor-stabiliser}
 (\Delta_i E_i)_\chi=(\Delta_i)_\chi(E_i)_\chi
 \qquad(\chi\in\IBr(B_0(H_i))).
\end{equation}
For $m\geq3$, Corollary~\ref{cor:type-b-principal-selector} shows
that $E_i$ fixes every character in $\IBr(B_0(H_i))$, so
\eqref{eq:type-b-factor-stabiliser} also holds. In particular, it holds
with $E_i$ replaced by any subgroup of $E_i$.

Recall that $\IBr(L,e_s^L)$ is the union of $\IBr(d)$ over the
$2$-blocks $d$ of $L$ belonging to $e_s^L$.
Fix an $\widetilde L$-orbit $\mathcal P$ in $\IBr(L,e_s^L)$ and choose
$\psi_0\in\mathcal P$. Let $\theta_0$ be an irreducible constituent
of $(\psi_0)_{L_0}$, and let $\mathcal O$ be its $\widetilde L$-orbit.
Let $d=\bl(\psi_0)$, and let $d_0$ be the block containing $\theta_0$.
Then $d$ covers $d_0$. Since $d$ belongs to $e_s^L$, the argument above
shows that every factor block of $d_0$ is strictly quasi-isolated.
The type $\mathsf A$ factors therefore satisfy
\cite[Theorem~8.1]{FengLiZhang2021Equivariant}, while the type
$\mathsf B$ factors of rank at least two are principal and satisfy
\eqref{eq:type-b-factor-stabiliser}.
By Lemma~\ref{lem:type-b-regular-levi-orbits}, the image of the
conjugation action of $\widetilde L$ on $L_0$ is $\prod_i\Delta_i$. If
$\theta_0=\theta_{0,1}\times\cdots\times\theta_{0,r}$, then
$\mathcal O$ consists of the products of characters in the
$\Delta_i$-orbits of $\theta_{0,i}$.

Let $D_{\mathcal O}$ be the stabiliser of $\mathcal O$ in $D$, and
choose a generator $\tau$ induced by $F_0^a$, with $a>0$.
Consider a cycle $C$ of the permutation of the factors $H_i$
induced by $\tau$, and set
$t=|C|$. Choose one index $i$ from the cycle $C$, and let
$\mathbf H_j$ be the algebraic component used to define $H_i$ in
Lemma~\ref{lem:type-b-regular-levi-orbits}.
Thus $H_i=\mathbf H_j^{F^{n_i}}$, where $n_i$ is the length of the
$F$-orbit of $\mathbf H_j$.
Since $\tau^t$ stabilises this factor, there is an integer $k$ such
that $F_0^{at}(\mathbf H_j)=F^k(\mathbf H_j)$.
Replacing $k$ by a congruent integer modulo $n_i$, choose $k\leq0$,
which gives $at-hk\geq0$.
Under the projection onto $\mathbf H_j$ in the proof of
Lemma~\ref{lem:type-b-regular-levi-orbits}, the action of $\tau^t$ on
$\mathbf H_j^{F^{n_i}}$ is induced by
\[
 F^{-k}F_0^{at}=F_0^{at-hk}.
\]
If $u$ is the $F_0$-orbit length of $\mathbf H_j$, both $at-hk$
and $hn_i$ are multiples of $u$.
Thus $F_0^{at-hk}$ and $F^{n_i}$ are powers of $F_0^u$ on
$\mathbf H_j$.
By the classification of Steinberg endomorphisms
\cite[Theorem~22.5]{MalleTesterman2011} and Lang's theorem,
we may conjugate $F_0^u$ to a standard Frobenius endomorphism, possibly
composed with a graph automorphism.
The same conjugation identifies both powers with powers of this
standard endomorphism.
Consequently, $\tau^t$ acts on $\mathbf H_j^{F^{n_i}}$ as a field
automorphism in type $\mathsf B$, and as a field automorphism possibly
composed with a graph automorphism in type $\mathsf A$.
Let $\Phi_i\leq\Aut(H_i)$ be the group generated by the field
automorphisms and, when $H_i$ has type $\mathsf A$, the graph
automorphisms. Thus $\Phi_i=E_i$ in type $\mathsf B$.
The automorphism induced by $\tau^t$ belongs to $\Phi_i$.
In type $\mathsf B$,
$\Delta_i\cap E_i=1$, so \eqref{eq:type-b-factor-stabiliser} holds in
$\Delta_i\rtimes\Phi_i$. In type $\mathsf A$,
\cite[Theorem~8.1]{FengLiZhang2021Equivariant} gives the corresponding
factorisation in $\Delta_i\rtimes\Phi_i$ for a suitable representative
of each $\Delta_i$-orbit. The action just described defines a homomorphism
$\alpha_i:\langle\tau^t\rangle\longrightarrow\Phi_i$.
Taking inverse images under the homomorphism
$(\delta,\sigma)\longmapsto(\delta,\alpha_i(\sigma))$ from
$\Delta_i\rtimes\langle\tau^t\rangle$ to $\Delta_i\rtimes\Phi_i$
verifies the hypothesis of Lemma~\ref{lem:type-b-component-return},
which gives $\theta\in\mathcal O$ fixed by $D_{\mathcal O}$. Hence
\[
 (\widetilde L D_{\mathcal O})_\theta
 =\widetilde L_\theta D_{\mathcal O}.
\]
Choose $y\in\widetilde L$ such that $\theta_0^y=\theta$, and set
$\psi=\psi_0^y$. Then $\psi\in\mathcal P\subseteq\IBr(L,e_s^L)$,
and $\theta$ is a constituent of $\psi_{L_0}$.

The proof of Lemma~\ref{lem:type-b-regular-levi-orbits} shows that
$[\mathbf L,\mathbf L]\hookrightarrow\mathbf LZ(\widetilde{\mathbf G})$
is a regular embedding. Apply \cite[Theorem~1.7.15]{GeckMalle2020}
to this embedding.
The restriction to $L_0$ of every irreducible $2$-modular
representation of $\widetilde L$ is multiplicity free.
Choosing a representation lying over $\theta$, modular Clifford
theory therefore shows that $\theta$ extends to $\widetilde L_\theta$.
Lemma~\ref{lem:type-b-regular-levi-orbits} also shows that
$L_0\lhd L\lhd\widetilde L$, that $\widetilde L/L_0$ is abelian,
and that $D_{\mathcal O}$ normalises this chain.
Lemma~\ref{lem:type-b-char2-clifford} now gives
\[
 (\widetilde L D_{\mathcal O})_\psi
 =\widetilde L_\psi(D_{\mathcal O})_\psi.
\]
If $x=\widetilde g\delta\in(\widetilde L D)_\psi$, with
$\widetilde g\in\widetilde L$ and $\delta\in D$, then Clifford
theory shows that $\theta^x$ is an $L$-conjugate of $\theta$.
Thus $\mathcal O^\delta\cap\mathcal O\neq\varnothing$, so
$\delta\in D_{\mathcal O}$.
Consequently,
\[
 (\widetilde L D)_\psi=\widetilde L_\psi D_\psi,
 \qquad (D_{\mathcal O})_\psi=D_\psi.
\]

Let $\Psi$ be the image of $\psi$ under the bijection
$\IBr(L,e_s^L)\longrightarrow\IBr(G,e_s^G)$ induced by Jordan
decomposition. This bijection preserves the block correspondence and is
$\widetilde L D$-equivariant by
\cite[Lemma~5.1 and the proof of Proposition~5.2]{FengLiZhang2022Jordan}.
Since $Z(\mathbf G)\leq\mathbf L$, the intersection
$\mathbf G\cap\mathbf LZ(\widetilde{\mathbf G})=\mathbf L$ is connected.
Lang's theorem applied to
$\widetilde{\mathbf G}=\mathbf G(\mathbf LZ(\widetilde{\mathbf G}))$
therefore gives $\widetilde G=G\widetilde L$, with
$G\cap\widetilde L=L$. Inclusion induces an isomorphism
$\widetilde L/L\cong\widetilde G/G$.
Under the equivariant Jordan bijection, this isomorphism identifies
the diagonal actions on $\psi$ and $\Psi$. Hence
\[
 (\widetilde G D)_\Psi=\widetilde G_\Psi D_\Psi,
 \qquad D_\Psi=D_\psi.
\]
The equivariance of rational Lusztig series
\cite[Proposition~7.2]{Taylor2018Automorphisms} also shows that field
automorphisms permute the idempotents $e_s^G$.
Since each block belongs to exactly one of these idempotents,
every element of $(\widetilde G E)_\Psi$ fixes $e_s^G$.
The group $\widetilde G$ fixes this idempotent, so the image of any
such element under the projection $\widetilde G\rtimes E\longrightarrow E$
lies in $D$. Consequently,
\[
 (\widetilde G E)_\Psi=(\widetilde G D)_\Psi
 =\widetilde G_\Psi E_\Psi,\qquad E_\Psi=D_\Psi.
\]
The group $E_\Psi$ is cyclic, so Lemma~\ref{lem:cyclic-extension}
extends $\Psi$ to $G\rtimes E_\Psi$.

Since $\mathcal P$ was arbitrary, the required representatives
exist on every $\widetilde G$-orbit in $\IBr(G,e_s^G)$.
Varying $s$ covers all nonprincipal orbits, while
Corollary~\ref{cor:type-b-principal-selector} covers the principal
ones. This proves the assertion.
\end{proof}

\begin{proposition}\label{prop:type-b-two-high-rank}
Let $S=\Omega_{2n+1}(q)$.
Then every $2$-block of $X_2(S)$ satisfies the inductive BAW condition.
\end{proposition}

\begin{proof}
If $(n,q)=(3,3)$, the result is Proposition~\ref{prop:type-b-q3}.
Otherwise, $X_2(S)=S$ by Table~\ref{tab:typeb-covers}.
Let $\mathbf G$ be simply connected of type $\mathsf B_n$, and
choose $F$ with $G=\mathbf G^F=\operatorname{Spin}_{2n+1}(q)$.
For every semisimple $2'$-element $s\in(\mathbf G^*)^F$,
\cite[Corollary~4.2]{Ruhstorfer2022Derived} describes the stabiliser
of $e_s^G$ in $\Out(G)$ using diagonal automorphisms and a power of
the standard field automorphism, since type~$\mathsf B$ has no
graph automorphisms.
By \cite[Lemma~4.5(a)--(b)]{Ruhstorfer2022Derived}, after conjugation
by $G$, the associated Levi subgroup and its idempotent may be chosen
stable under this field automorphism. Thus the group $A_s$ in
Lemma~\ref{lem:jordan-reduction} may be chosen inside the field
automorphism group $E$ of Proposition~\ref{prop:type-b-brauer-hypothesis}.
Restricting the stabiliser factorisation and extensions supplied by
that proposition to $A_s$ verifies
\cite[Assumption~5.3]{FengLiZhang2022Jordan} at the prime $2$.
Lemma~\ref{lem:type-b-localized-return} therefore shows that every
$2$-block of $S$ is BAW-good, and
Theorem~\ref{thm:spath-triples} gives the inductive BAW condition.
\end{proof}

\subsection{Proof of Theorem~\ref{thm:type-b}}
\label{subsec:type-b-completion}

\begin{proof}
The defining prime case follows from \cite[Theorem~C]{Spath2013}.
Proposition~\ref{prop:type-b-odd-primes} treats odd nondefining primes,
and Proposition~\ref{prop:type-b-two-high-rank} treats the prime $2$.
Thus every $\ell$-block of $X_\ell(S)$ satisfies the inductive BAW
condition. The final assertion follows from
Remark~\ref{rem:block-aggregation}\textup{(a)}.
\end{proof}

\begin{corollary}\label{cor:type-b-census}
Every finite simple group of type $\mathsf B$ satisfies iBAW at every prime
dividing its order.
\end{corollary}

\begin{proof}
Theorem~\ref{thm:type-b} treats odd characteristic in ranks at least
three. In even characteristic, the isomorphisms between the finite simple
groups of types $\mathsf B_n$ and $\mathsf C_n$ reduce the assertion to
Theorem~\ref{thm:symplectic}, which also covers rank two through
$\mathsf B_2=\mathsf C_2$. The case of rank one follows from
\cite[Proposition~4.6]{FengLiZhang2023LowRankA}, since
$\mathsf B_1=\mathsf A_1$. The same result covers the derived subgroup
$\Sp_4(2)'\cong A_6\cong\operatorname{PSL}_2(9)$.
\end{proof}

\section{Sporadic groups}\label{sec:sporadic}

Let $S$ be one of the $26$ sporadic simple groups and let
$\ell\mid |S|$.  As before, $X=X_\ell(S)$ denotes the universal
$\ell'$-covering group.

\begin{theorem}\label{thm:sporadic}
Every sporadic simple group $S$ satisfies iBAW at every prime
$\ell\mid |S|$ on $X_\ell(S)$.
\end{theorem}

The overview distributed with CTBlocks~0.9.5
\cite{BreuerSporadicOverview} states that the inductive blockwise
Alperin weight condition in \cite[Definition~4.1]{Spath2013} has been
checked for every sporadic simple group other than $J_4$, $Fi'_{24}$,
the Baby Monster and the Monster. We use the CTBlocks verification for
those $22$ groups and prove the remaining four cases below.

\subsection{Lemmas for the sporadic proof}

When the outer automorphism group is trivial, the criterion in
\cite[Corollary~2.13]{FengLiZhang2019} gives the following.

\begin{lemma}\label{lem:complete-collapse}
Let $S$ be a nonabelian finite simple group and let
$X=X_\ell(S)$ be its universal $\ell'$-covering group.  Suppose that
$\Out(X)=1$ and $|\IBr(b)|=|\Alp(b)|$ for every $\ell$-block $b$ of
$X$.  Then $S$ satisfies iBAW at $\ell$ on $X$.
\end{lemma}

\begin{proof}
Since $X$ is the universal $\ell'$-covering group of $S$, we have
$\Out(S)\cong\Out(X)=1$. For each $\ell$-block $b$ of $X$, the
cardinality assumption gives a bijection
$\Omega_b:\IBr(b)\longrightarrow\Alp(b)$.
Every automorphism of $X$ is inner and therefore fixes every
irreducible Brauer character and every $X$-conjugacy class of weights.
Thus $\Omega_b$ is $\Aut(X)_b$-equivariant.
\cite[Corollary~2.13]{FengLiZhang2019} therefore shows that every
$\ell$-block of $X$ satisfies the inductive BAW condition. Remark~\ref{rem:block-aggregation}\textup{(a)} gives the conclusion.
\end{proof}

The following lemma is an elementary consequence of the orbit decomposition
of finite $C_2$-sets.

\begin{lemma}\label{lem:block-cancellation}
Let $\mathcal B$ be a finite $C_2$-set with orbits
$\mathcal B_0,\ldots,\mathcal B_r$. Let $Y$ and $Y'$ be finite
$C_2$-sets with decompositions
\[
 Y=\bigsqcup_{b\in\mathcal B}Y_b,\qquad
 Y'=\bigsqcup_{b\in\mathcal B}Y'_b
\]
such that $Y_b^a=Y_{b^a}$ and $(Y'_b)^a=Y'_{b^a}$ for
$a\in C_2$ and $b\in\mathcal B$.
Set $U_i=\bigsqcup_{b\in\mathcal B_i}Y_b$ and
$U'_i=\bigsqcup_{b\in\mathcal B_i}Y'_b$.
If $Y\cong Y'$ and $U_i\cong U'_i$ as $C_2$-sets for
$1\leq i\leq r$, then $U_0\cong U'_0$.
Moreover, this isomorphism can be chosen to map $Y_b$ onto $Y'_b$
for every $b\in\mathcal B_0$.
\end{lemma}

\begin{proof}
A finite $C_2$-set is determined up to isomorphism by its numbers
of singleton orbits and orbits of size two. Cancelling the known
parts therefore gives $U_0\cong U'_0$.
If $\mathcal B_0$ is a singleton, this isomorphism has the required
property. Otherwise, write $\mathcal B_0=\{b,b^t\}$, where $t$
generates $C_2$. The action of $t$ exchanges the two parts on
each side, so $|Y_b|=\tfrac12|U_0|=\tfrac12|U'_0|=|Y'_b|$.
Choose a bijection $f_b:Y_b\to Y'_b$ and extend it by
$f(y^t)=f_b(y)^t$ for $y\in Y_b$.
Together with $f|_{Y_b}=f_b$, this defines the required
$C_2$-equivariant bijection.
\end{proof}

\Needspace{12\baselineskip}
\subsection{The four remaining groups}

The Schur multipliers and outer automorphism groups used below are taken from
\cite[Table~1]{Atlas1985}.

\begin{proposition}\label{prop:four-boundary}
For each of the $45$ pairs $(S,\ell)$ listed below, $S$ satisfies iBAW at
$\ell$ on its universal $\ell'$-covering group.
\[
\begin{array}{c|c}
S&\text{primes }\ell\\
\hline
J_4&2,3,5,7,11,23,29,31,37,43\\
Fi'_{24}&2,3,5,7,11,13,17,23,29\\
\Baby&2,3,5,7,11,13,17,19,23,31,47\\
\Monster&2,3,5,7,11,13,17,19,23,29,31,41,47,59,71
\end{array}
\]
\end{proposition}

\begin{proof}
For $J_4$ the Schur multiplier and outer automorphism group are trivial.
The blockwise Alperin weight conjecture for every block with noncyclic
defect group is \cite[Theorem~5.2]{AnOBrienWilsonJ4}.  Blocks with cyclic
defect group, including blocks of defect zero, are covered by
\cite[Proposition~6.2]{Spath2013}.  Lemma~\ref{lem:complete-collapse}
shows that $J_4$ satisfies iBAW at every prime dividing its order.

We next consider $Fi'_{24}$. The blockwise weight counts for the blocks of
$Fi'_{24}$ with noncyclic defect groups are established in
\cite[Theorem~4.3]{AnCannonOBrienUngerFi24}. We also determine the actions
of the outer involution needed for the equivariant bijections.
Set $X=X_\ell(Fi'_{24})$, so that
$X=Fi'_{24}$ at $3$ and $X=3.Fi'_{24}$ at every other prime.
For $\ell\ne3$, let $Z=Z(X)$ and let
$\pi:X\longrightarrow\overline X:=X/Z=Fi'_{24}$ be the quotient map.
For every $\ell$-subgroup $\overline Q$ of $\overline X$,
Schur--Zassenhaus gives a complement $Q$ to $Z$ in
$\pi^{-1}(\overline Q)$. Since $Z$ is central, this complement is
unique. Moreover, $Q$ is radical if and only if $\overline Q$ is radical
\cite[Lemma~2.3(c)]{NavarroTiep2011}, and
\begin{equation}\label{eq:fi24-central-quotient}
 \begin{aligned}
 \overline{N_X(Q)}&=N_{\overline X}(\overline Q),\\
 N_X(Q)/(QZ)&\cong N_{\overline X}(\overline Q)/\overline Q.
 \end{aligned}
\end{equation}
By \cite[Theorem~9.9(c)]{Navarro1998}, the block correspondence for
$X\longrightarrow\overline X$ gives a bijection between the blocks of
$\overline X$ and the blocks of $X$ lying over $1_Z$. The latter are
precisely the blocks in the trivial sector of
Definition~\ref{def:central-sector}.
For corresponding blocks $\overline b$ and $b$, inflation gives a bijection
\[
 \IBr(\overline b)\longrightarrow\IBr(b).
\]
For a weight $(\overline Q,\overline\theta)$ of $\overline b$, let $Q$
be the unique lift described above, and let $\theta$ be the inflation
of $\overline\theta$ along the quotient homomorphism
$N_X(Q)/Q\longrightarrow N_{\overline X}(\overline Q)/\overline Q$,
$nQ\longmapsto\overline{nQ}$.
Passing to conjugacy classes gives a bijection
\[
 \Alp(\overline b)\longrightarrow\Alp(b)
\]
by \cite[Lemma~2.3(iii)]{FengLiZhang2019}.
These bijections commute with the induced automorphism actions.
Thus an equivariant bijection $\IBr(\overline b)\to\Alp(\overline b)$
gives an equivariant bijection $\IBr(b)\to\Alp(b)$
\cite[Lemma~2.5]{FengLiZhang2019}.

For $\ell\neq3$, the outer involution exchanges the two faithful sectors
of $X=3.Fi'_{24}$ \cite[Lemma~4.7]{AnDietrich2012}. The stabiliser of a
block in either sector is therefore $\Inn(X)$, so any bijection between
its irreducible Brauer characters and weight classes is equivariant.
Remark~\ref{rem:block-aggregation}\textup{(b)} extends this bijection to the other
block in its orbit.

For a finite $C_2$-set $Y$, we use the pair $(|Y|,|Y^{C_2}|)$, with
$C_2$ acting through the nontrivial outer automorphism. We compare
these pairs for $\IBr(b)$ and $\Alp(b)$ whenever the outer involution
fixes the block $b$.
Throughout the following argument, we use Table~8 of
\cite{AnDietrich2012}. In the column headed $|\Irr^0|$, an entry $a/b$
denotes $a+b$ weights, of which $a$ are fixed by the outer involution.  For
the blocks specified below, the program
\path{fi24blocks.g} computes the number of Brauer characters, the number
fixed by the outer involution, and the action of the outer involution on blocks.  It
lists the block defects separately.  At $2$,
the pairs of total and fixed Brauer character counts for the blocks
in the trivial sector are
\begin{equation}\label{eq:fi24-p2-signatures}
 (33,25),\quad(3,1),\quad(3,3),\quad(1,1),\quad(1,1).
\end{equation}
The last two blocks have defect zero. For each of them we use the
bijection $\chi^0\mapsto(1,\chi)$, where $\chi$ is its unique ordinary
irreducible character. The second and third blocks have defect
groups of order $4$ and $8$ by
\cite[Lemma~4.2(d)]{AnCannonOBrienUngerFi24}.
The Alperin weight conjecture holds for both blocks by
\cite[Theorem~13.7]{Sambale2014}. Together with
\eqref{eq:fi24-p2-signatures}, this gives three weights for each block.

The radical subgroup of a weight belonging to either block is conjugate
to a subgroup of its defect group \cite[Theorem~4.14]{Navarro1998}.
The row labelled $1$ in \cite[Table~8]{AnDietrich2012} accounts for
the two blocks of defect zero, and the row labelled $2$ has entry $0/0$
in the trivial column.
The classes $(2^2)_a$ and $(2^2)_b$ in
\cite[Table~8]{AnDietrich2012} are denoted by $2^2$ and $(2^2)^*$,
respectively, in \cite[Table~2]{AnCannonOBrienUngerFi24}.
The defect groups of the block of defect~$2$ lie in the class $(2^2)_b$
by \cite[Lemma~4.2(d)]{AnCannonOBrienUngerFi24}.
Its three weights therefore lie among the four weights in the
$(2^2)_b$ row of \cite[Table~8]{AnDietrich2012}.
The entry $2/2$ means that these four weights consist of two fixed
weights and one orbit of length two. Since the block is fixed by the outer
involution, its three weights consist of the orbit of length two and one
fixed weight, giving the pair $(3,1)$.

For the block with dihedral defect group of order $8$, the possible radical
classes are $(2^2)_a$, $(2^2)_b$ and $D_8$, with entries $1/0$, $2/2$ and $1/0$
in the trivial column of \cite[Table~8]{AnDietrich2012}.
Three of the four $(2^2)_b$-weights belong to the block of defect~$2$.
Hence the only weights that can belong to the block with defect group
$D_8$ are the $(2^2)_a$-weight, the remaining fixed $(2^2)_b$-weight and
the $D_8$-weight. Since this block has three weights, it contains all three.
Each is fixed, giving the pair $(3,3)$.
These pairs agree with the corresponding Brauer character counts
in \eqref{eq:fi24-p2-signatures}, so each of the two blocks $b$ admits an
$\Aut(X)_b$-equivariant bijection $\IBr(b)\to\Alp(b)$.
The program \path{fi24blocks.g}, described in Section~\ref{app:fischer},
independently confirms this distribution of the $(2^2)_a$- and
$(2^2)_b$-weights for every candidate fusion.

Summing the entries in the trivial column of
\cite[Table~8]{AnDietrich2012} gives the pair $(41,31)$ for the union
of the weight sets of all blocks in the trivial sector.
Lemma~\ref{lem:block-cancellation}, after removal of the two blocks just
treated and the two blocks of defect zero, gives an equivariant bijection
for the principal block.  Its pair of total and fixed weight counts is
$(33,25)$, in agreement with \eqref{eq:fi24-p2-signatures}.

At $2$, each faithful sector contains exactly two blocks.
They have defects $21$ and $3$ and contain $23$ and $2$ Brauer
characters, respectively. Summing the column labelled $\lambda_3$ in
\cite[Table~8]{AnDietrich2012} gives $25$ weights in each faithful sector.  The
block of defect $3$ has two weights by
\cite[Theorem~13.7]{Sambale2014}.  The other block therefore has $23$
weights.

At $3$, set $G=Fi'_{24}$. The classification of radical subgroups and
their normalisers in \cite[Proposition~4.1 and Table~1]{AnCannonOBrienUngerFi24}
gives a radical subgroup $Q\cong C_3^2$ with
$N_G(Q)\cong(3^2\!:\!2\times G_2(3)).2$.
By \cite[Lemma~4.2(c) and Theorem~4.3]{AnCannonOBrienUngerFi24}, $G$
has a unique nonprincipal block $B_1$ of positive defect, with defect
group $Q$ and four weights. In CTblLib~1.3.11 \cite{CTblLib2025},
the table \path{F3+} represents $G$. The table
\path{(3^2:2xG2(3)).2} is the seventeenth entry in CTblLib's list of
character tables of maximal subgroups of $G$, and represents $N_G(Q)$.

Let $b$ be the $3$-block of $N_G(Q)$ numbered $2$ in the CTblLib
calculation described in Section~\ref{app:fischer}. By \cite[Theorem~4.14]{Navarro1998}, the induced block $b^G$ is defined.
It can be identified from the reduced central characters of induced
ordinary characters using \cite[Lemma~6.1]{Navarro1998}.  The programs \path{fi24p3.g} and \path{fi24blocks.g} select the four
characters of $N_G(Q)/Q$ that have defect zero and whose inflations
lie in $b$. The program \path{fi24blocks.g} applies the central character test to
every class fusion returned by \texttt{PossibleClassFusions} for the
indicated tables.
For each of the six ordinary characters in $b$, every candidate
identifies the block of $G$ numbered $2$. The three blocks of $G$ have defects
$16,2,0$, respectively, and the first is principal. Thus the second is
$B_1$, and $b^G=B_1$. The four selected characters therefore give all
four $B_1$-weights.

The $3^2$ entry of \cite[Table~8]{AnDietrich2012} shows that two of these
weights are fixed by the outer involution, giving the pair $(4,2)$.
The program \path{fi24blocks.g} gives the same pair of total and fixed
Brauer character counts.  Thus the two $C_2$-sets are isomorphic, giving an equivariant
bijection for $B_1$. For the block of defect zero, use $\chi^0\mapsto(1,\chi)$.  Lemma~\ref{lem:block-cancellation}, applied to the equivariant bijection
constructed in \cite[Section~4.3.4]{AnDietrich2012} between the irreducible
Brauer characters and the conjugacy classes of weights in the entire sector,
gives an equivariant bijection for the principal block.  Its pair of total and fixed weight counts is $(25,25)$, equal to the
corresponding pair for Brauer characters computed by the program.

At $5$, the nontrivial radical subgroups of $Fi'_{24}$ have orders
$5$ and $25$ \cite[Proposition~4.1 and Table~1]{AnCannonOBrienUngerFi24}.
The central quotient correspondence in
\eqref{eq:fi24-central-quotient} gives their unique radical lifts to
$3.Fi'_{24}$ and the corresponding normalisers. Weight characters in the
trivial sector are obtained by the inflation described there. The program \path{fi24weights.g}, described in
Section~\ref{app:fischer}, determines the block distributions and outer
actions of these weights. It verifies that the resulting pairs of total and fixed weight counts are
independent of the candidate class fusions and that every weight associated
with a radical subgroup of order $5$ belongs to a block with cyclic defect. The classification in
\cite[Proposition~4.1 and Table~1]{AnCannonOBrienUngerFi24}, together
with the corresponding CTblLib character tables, accounts for the weights
associated with all nontrivial radical subgroups.  In the
trivial sector the program gives the pairs $(16,16)$, $(14,6)$ and
$(16,16)$ for the three blocks with defect group $C_5^2$, in the block
order of the character table. The program \path{fi24blocks.g} gives the
same pairs for Brauer characters, so these three blocks admit
equivariant bijections. It also gives the
pairs of blocks $45,46$ and $47,48$ interchanged by the outer involution.
CTblLib's identifications of the character tables and the outer action show
that these are the two pairs of blocks exchanged between the faithful sectors. The blocks in these pairs have $16$ and
$14$ weights and Brauer characters, respectively. The blocks in the faithful sectors are
treated by applying Remark~\ref{rem:block-aggregation}\textup{(b)} to their outer
automorphism orbits.  All remaining $5$-blocks have cyclic defect or
defect zero and are covered by \cite[Proposition~6.2]{Spath2013}.

At $7$, each of the three sectors has one block of noncyclic defect.
Use the equivariant bijection for each entire sector constructed in
\cite[Section~4.3.4]{AnDietrich2012}, with the counts in \cite[Table~8]{AnDietrich2012}.  The blocks with cyclic defect
and the blocks of defect zero admit equivariant bijections by
\cite[Proposition~6.2]{Spath2013}.  In the trivial sector, apply
Lemma~\ref{lem:block-cancellation} to the bijection on the entire sector, using these
blockwise bijections for the cyclic and defect zero blocks.  This gives an
equivariant bijection for the unique block of noncyclic defect in the
trivial sector. Its pair of total and fixed weight counts is $(22,12)$, equal to the
corresponding pair for Brauer characters computed by \path{fi24blocks.g}.
Each faithful sector has $63$ weights and, by the sector bijection, $63$
Brauer characters. The program \path{fi24blocks.g} gives $22$ Brauer
characters in its unique block of noncyclic defect. The remaining blocks
therefore contain $41$ Brauer characters and, by
\cite[Proposition~6.2]{Spath2013}, $41$ weights. Thus the block of
noncyclic defect has $22$ weights. Choose a bijection for one of these
two blocks and define the bijection for the other by conjugating with
the outer involution.

At $11$, $13$, $17$, $23$, and $29$, every block of positive defect has cyclic
defect, as verified by \path{fi24blocks.g}. These blocks and the blocks of
defect zero are covered by \cite[Proposition~6.2]{Spath2013}. We have therefore constructed an
$\Aut(X)_b$-equivariant bijection $\IBr(b)\to\Alp(b)$ for every
$\ell$-block $b$ of $X=X_\ell(Fi'_{24})$.
Since $\Out(Fi'_{24})\cong C_2$, \cite[Corollary~2.13]{FengLiZhang2019}
shows that every such block satisfies the inductive BAW condition.
Remark~\ref{rem:block-aggregation}\textup{(a)} completes the proof for $Fi'_{24}$.

For $\Baby$, the universal $\ell'$-covering group is $\Baby$ at $2$ and
$2.\Baby$ at odd primes.  At $2$ the ordinary character table has two
blocks.  They contain $179$ and $5$ ordinary irreducible characters and
have defects $41$ and $3$, respectively.  The correctness of
the ordinary character table is proved in
\cite[Section~7]{BreuerMagaardWilsonBaby2020}.  The exact restriction
matrices on the $2$-regular classes have ranks $25$ and $2$.  The
calculation is given in Section~\ref{app:baby-monster}.  By
\eqref{eq:restriction-rank}, it gives $|\IBr(\Baby)|=27$.  The weight calculation in
\cite[proof of Theorem~4.2 and Table~2]{AnDietrich2012}, with the
corrections in \cite{AnDietrichErratum}, gives $27$ conjugacy classes
of $2$-weights.  The nonprincipal block $b$ of
defect $3$ has $|\IBr(b)|=2$ by \eqref{eq:restriction-rank}, and
$|\IBr(b)|=|\Alp(b)|$ by \cite[Theorem~13.7]{Sambale2014}.
Subtraction gives $25$ weights for the principal block, which has $25$
Brauer characters by \eqref{eq:restriction-rank}.  Thus
$|\IBr(b)|=|\Alp(b)|$ for every $2$-block $b$ of $\Baby$.  Since
$\Out(\Baby)=1$, Lemma~\ref{lem:complete-collapse} shows that $\Baby$
satisfies iBAW at $2$.
At $\ell=7$, the total $220$ printed for the number of
$7$-regular conjugacy classes of $2.\Baby$ in the proof of
\cite[Lemma~5.2]{AnWilsonBaby} should be $222$.  The supplementary
calculation described in Section~\ref{app:baby-monster} verifies the
corrected total and the restriction ranks $24,24,21,24$ for the four
blocks with noncyclic defect.  The blockwise values used in
\cite[Theorem~5.3]{AnWilsonBaby} are therefore unaffected.
At odd primes, every block $b$ of $2.\Baby$ with noncyclic defect
satisfies $|\IBr(b)|=|\Alp(b)|$ by \cite[Theorem~5.3]{AnWilsonBaby}.
Since $\Out(\Baby)=1$, \cite[Proposition~6.2]{Spath2013}, applied to
the simple group $\Baby$, gives the same equality for every block of
$2.\Baby$ with cyclic defect, including defect zero.  Since
$\Out(2.\Baby)\cong\Out(\Baby)=1$,
Lemma~\ref{lem:complete-collapse} applies.

For $\Monster$ the universal $\ell'$-covering group is $\Monster$ at
every prime.  At $2$ there are five blocks, of defects $46,0,0,4,0$ and
numbers of ordinary irreducible characters $183,1,1,8,1$.  The correctness
of the ordinary character table is proved in
\cite[Section~5]{BreuerMagaardWilsonMonster2026}.
The program in Section~\ref{app:baby-monster} verifies that $\Monster$ has
$61$ $2$-regular conjugacy classes, so $|\IBr(\Monster)|=61$ by
\cite[Corollary~2.10]{Navarro1998}.  The weight calculation in
\cite[proof of Theorem~4.2 and Table~17]{AnDietrich2012}, with the
corrections in \cite{AnDietrichErratum}, gives $61$ conjugacy classes
of $2$-weights.  Each of the three blocks $b$ of defect zero satisfies
$|\IBr(b)|=|\Alp(b)|=1$.  The unique nonprincipal block $b$ of defect $4$ satisfies
$|\IBr(b)|=|\Alp(b)|$ by \cite[Theorem~13.7]{Sambale2014}.  These four blocks are all the nonprincipal $2$-blocks. Subtracting their
Brauer character and weight counts from the respective totals gives
$|\IBr(B_0(\Monster))|=|\Alp(B_0(\Monster))|$ for the principal block.  Thus
$|\IBr(b)|=|\Alp(b)|$ for every $2$-block $b$ of $\Monster$, and
Lemma~\ref{lem:complete-collapse} shows that $\Monster$ satisfies iBAW
at $2$.
For every odd $\ell\mid|\Monster|$, every $\ell$-block $b$ of
$\Monster$ with noncyclic defect satisfies
$|\IBr(b)|=|\Alp(b)|$ by \cite[Theorem~4.4]{AnWilsonMonster}.  Since
$\Out(\Monster)=1$, \cite[Proposition~6.2]{Spath2013}, applied to the
simple group $\Monster$, gives the same equality for every block with
cyclic defect, including defect zero.  Thus
$|\IBr(b)|=|\Alp(b)|$ for every $\ell$-block $b$ of $\Monster$, and
Lemma~\ref{lem:complete-collapse} shows that $\Monster$ satisfies iBAW
at every odd prime dividing its order.
\end{proof}

\Needspace{6\baselineskip}
\subsection{Proof of the sporadic theorem}

\begin{proof}[Proof of Theorem~\ref{thm:sporadic}]
The result follows from Proposition~\ref{prop:four-boundary} and the
CTBlocks verification \cite{BreuerSporadicOverview} for the other
$22$ sporadic groups.
\end{proof}

\section{Proof of Theorem~\ref{thm:main} and
Corollary~\ref{cor:reduction}}

\begin{proof}
Theorem~\ref{thm:main} follows from Theorem~\ref{thm:symplectic},
Corollary~\ref{cor:type-b-census}, and Theorem~\ref{thm:sporadic}.
For Corollary~\ref{cor:reduction}, let $H$ be the finite group in its
statement. Sp\"ath's reduction theorem
\cite[Theorem~A]{Spath2013} reduces the assertion to the nonabelian
simple groups involved in $H$.
Theorem~\ref{thm:main} applies when $\ell$ divides the order of such a group $S$.
Otherwise, every $\ell$-block of its universal $\ell'$-covering group
has defect zero and satisfies the inductive BAW condition trivially.
Remark~\ref{rem:block-aggregation}\textup{(a)} gives the conclusion.
\end{proof}

\appendix\suppressfloats[t]
\section{Finite computations}\label{app:computation}

The computational supplement contains eleven programs in the directory
\path{code} and their complete transcripts in \path{outputs}.
The programs were run with \GAP~4.16.0 \cite{GAP2026}, CTblLib~1.3.11
\cite{CTblLib2025}, and AtlasRep~2.1.11 \cite{AtlasRep2026}.
They are listed in Table~\ref{tab:computation-programs}.
The identifications of named character tables and stored fusions with the
groups and inclusions below are taken from CTblLib and the classifications
cited in each calculation. The computations with these tables do not
establish those identifications.

\begin{table}[tbp]
\centering
\caption{Programs in the computational supplement.}
\label{tab:computation-programs}
\small
\begin{tabular}{@{}p{0.22\textwidth}p{0.68\textwidth}@{}}
\toprule
Program & Calculation \\
\midrule
\path{sp6.g} &
The $3$-radical subgroups and blockwise weight counts for $2.\Sp_6(2)$. \\
\path{o7weights.g} &
Two principal weight classes of $\operatorname{SO}_7(3)$ that each
cover two principal weight classes of $\Omega_7(3)$. \\
\path{o7brauer.g} &
Restrictions of principal block Brauer characters from
$\operatorname{SO}_7(3)$ to $\Omega_7(3)$. \\
\path{o7blocks.g} &
The nine $2$-blocks of $3.\Omega_7(3)$ and their permutation under the
outer involution. \\
\path{o7block2.g} &
The decomposition matrix and outer action for the block $B_2$. \\
\path{o7radical.g} &
The $2$-radical subgroups of $3.\Omega_7(3)$ and weight counts in its
three central character sectors. \\
\path{fi24blocks.g} &
Block data, Brauer character counts, outer actions, and block induction
for $Fi'_{24}$ and $3.Fi'_{24}$. \\
\path{fi24p3.g} &
The quotient character table and inflated characters for the normaliser
of a radical subgroup $C_3^2$ of $Fi'_{24}$. \\
\path{fi24weights.g} &
Block distributions and outer actions of weights associated with radical
subgroups of orders $5$ and $25$. \\
\path{baby.g} &
Numbers of irreducible Brauer characters of the Baby Monster at $2$ and
its double cover at $7$. \\
\path{monster.g} &
The ordinary $2$-blocks and number of $2$-regular conjugacy classes of
the Monster. \\
\bottomrule
\end{tabular}
\end{table}

\subsection{Numbers of irreducible Brauer characters}
\label{app:restriction-counts}

For a block $b$, let $V_b$ be the complex span of the restrictions to the
$\ell$-regular classes of the ordinary characters in $b$.  The irreducible Brauer characters in $b$ are linearly independent on the
$\ell$-regular classes. Since the decomposition matrix of $b$ has full column
rank, these characters span $V_b$.
Consequently
\begin{equation}\label{eq:restriction-rank}
                         \dim V_b=|\IBr(b)|.
\end{equation}

\subsection{\texorpdfstring{$2.\Sp_6(2)$ at the prime $3$}{2.Sp6(2) at the prime 3}}
\label{app:sp6}

Let $X=2.\Sp_6(2)$.  The program \path{sp6.g} supplies the numerical
details of the \GAP{}
calculation for $X$ at $3$ in
\cite[proof of Theorem~5.5]{SchaefferFry2014}.  It obtains the faithful
permutation representation of degree $240$ from AtlasRep.  For a Sylow
$3$-subgroup $P$, it chooses representatives $Q$ of the $20$
$P$-conjugacy classes of subgroups, tests the condition
$Q=O_3(N_X(Q))$ for each representative, and then fuses the surviving
classes in $X$.  There are five conjugacy classes in $X$ of $3$-radical
subgroups, including the trivial subgroup, of orders
$1,3,27,27,81$, whose respective normalisers have orders
$2903040,8640,2592,2592,648$.
CTblLib supplies the named character tables used for these normalisers,
together with stored fusions into the character table of $X$.
The program verifies the orders of these tables, their
$3$-cores, and the corresponding quotient orders.  For each named table, it
selects the characters which are trivial on $Q$ and have defect zero after
passage to $N_X(Q)/Q$.  It induces these characters to $X$, uses the
resulting central characters to determine the blocks of $X$ to which they belong, and checks
that each induced character determines a unique block.  In the
block order of the character table, the resulting
weight counts are $10,2,2,1,6$.
In the same order, the matrices of restrictions of the ordinary characters in
these blocks to the $3$-regular classes have ranks $10,2,2,1,6$.  By \eqref{eq:restriction-rank},
these are the numbers of irreducible Brauer characters.  The Brauer table
supplied by CTblLib confirms these numbers.
Thus the program verifies the numerical blockwise equalities in the
exceptional calculation.  The inductive BAW conclusion is supplied by
\cite[Theorem~5.5]{SchaefferFry2014}.

\subsection{\texorpdfstring{$3.\Omega_7(3)$ at the prime $2$}{3.Omega7(3) at the prime 2}}
\label{app:type-b-computation}

Set $X=3.\Omega_7(3)$. The programs \path{o7blocks.g} and
\path{o7block2.g} use the CTblLib tables \path{3.O7(3)} and
\path{3.O7(3).2} of $X$ and its extension by the outer involution,
together with the stored fusion for this inclusion, to determine the
compatible table automorphism and its action on blocks and Brauer
characters.

\subsubsection{The principal block}

Set $H=\operatorname{SO}_7(3)$ and $S=\Omega_7(3)$.
The program \path{o7weights.g} enumerates the $P$-conjugacy
classes of subgroups of orders $2^7$ and $2^8$ in a Sylow
$2$-subgroup $P$ of $S$. Among these classes, it retains one at each
order whose representatives $Q$ satisfy $C_P(Q)=Z(Q)$ and are
$2$-radical in $H$. For each retained representative $Q$, it verifies
$C_H(Q)=Z(Q)$ and finds one ordinary character of $2$-defect zero in
$N_H(Q)/Q$.
These give principal weights by \cite[Lemma~2.3]{FengYuZhang2024}.
Since $Q\leq S$, \cite[Corollary~2.22(2)]{FengYuZhang2024} shows that
each resulting principal $H$-weight class covers two $S$-weight classes.

The program \path{o7brauer.g} computes the
restrictions of the ten principal block Brauer characters of $H$ to $S$.
Eight restrictions are irreducible and two have two distinct constituents.
Modular Clifford theory identifies the two pairs of constituents as the two
orbits of size two.  Thus the principal Brauer characters and principal
weights of $S$ have isomorphic $C_2$-actions.

\subsubsection{The nine blocks of the triple cover}

The program \path{o7blocks.g} determines the nine $2$-blocks of $X$,
their defects, central character sectors, Brauer characters, and their
permutation under the outer involution.  These data give
Table~\ref{tab:typeb-q3-blocks}.

For $B_2$, \path{o7block2.g} gives five ordinary characters with
respective heights $(0,0,1,0,0)$ and two Brauer characters.  Its
decomposition matrix is
\[
 \begin{pmatrix}1&0\\1&0\\0&1\\1&1\\1&1\end{pmatrix}.
\]
It also shows that the outer involution fixes both Brauer characters.
Proposition~\ref{prop:type-b-q3} uses these data to identify the defect
group as $D_8$ and to show that the outer involution fixes both weight
classes.

In the AtlasRep permutation representation of degree $1134$, the
program \path{o7radical.g} enumerates the
$3021$ conjugacy classes of subgroups of a Sylow $2$-subgroup.  It retains
the $26$ representatives satisfying
$Q=O_2(N_X(Q))$ and fuses them into $12$ $X$-classes.
Every $2$-radical subgroup of $X$ is conjugate into the chosen Sylow
subgroup, so the enumeration is exhaustive.  For each fused
class it forms $N_X(Q)/Q$, finds its ordinary characters of defect zero, and
lists the restrictions of their inflations to $Z(X)\cong C_3$.  The resulting totals in
the three central character sectors are $(17,8,8)$.
The proof of Proposition~\ref{prop:type-b-q3} uses the two faithful
totals to determine the weight counts of $B_6$ and $B_7$.

\subsection{The Fischer group}
\label{app:fischer}

The program \path{fi24blocks.g} loads the named ordinary character
tables \path{F3+}, \path{3.F3+}, and their outer extensions, with the
corresponding maximal subgroup data from CTblLib.
For a fixed character table and prime, block numbers refer to the
numbering returned by \texttt{PrimeBlocks}, separately for each table.
The program checks the covering group order and the set of prime
divisors of the group order. At $2$, $3$, $5$ and $7$, it treats the blocks considered in
Section~\ref{sec:sporadic}. At $5$ and $7$, these are precisely the blocks of
noncyclic defect.  For an invariant selected block, it applies
\eqref{eq:restriction-rank} to the restricted
ordinary characters and to their sums with their images under the outer
involution. The second rank is the dimension of the subspace fixed by the
involution, so twice that rank minus the first is the number of fixed
Brauer characters.  For these blocks, the program checks the
stated numbers of Brauer characters and fixed points, the central
characters of the blocks, and the permutations of their labels under the
outer involution.  It separately checks the complete distributions of block defects at $5$, $7$, $11$, $13$, $17$, $23$, and $29$.  The weight
counts and sector totals are supplied separately by \path{fi24weights.g}
or the cited sources.
At $2$, the program \path{fi24blocks.g} also gives an independent check
of the distribution among blocks of the $(2^2)_a$- and $(2^2)_b$-weights
obtained in the proof of Proposition~\ref{prop:four-boundary}.
It uses the identifications, supplied by \cite{AnCannonOBrienUngerFi24}
and CTblLib, of the tables \path{2^2.U6(2).3.2} and
\path{(A4xO8+(2).3).2} with the normalisers in $Fi'_{24}$ of
representatives of these two radical classes, respectively.
The selected characters are the inflations of the corresponding defect
zero characters of the normaliser quotients. The program uses induced
central characters to determine their blocks in $Fi'_{24}$. The
fusions induced by the corresponding inclusions are among the candidates
tested. For every fusion returned by \texttt{PossibleClassFusions}, the
program verifies that each central character determines a unique block
and that the resulting distribution is independent of the candidate.

For $G=Fi'_{24}$, the classification of radical subgroups and their
normalisers in \cite[Proposition~4.1 and Table~1]{AnCannonOBrienUngerFi24}
gives a radical subgroup $Q\cong C_3^2$ with
$N_G(Q)\cong(3^2\!:\!2\times G_2(3)).2$.
The table \path{(3^2:2xG2(3)).2} is the seventeenth entry in CTblLib's
list of character tables of maximal subgroups of $G$. CTblLib also
supplies its fusion into \path{F3+}.

The program \path{fi24p3.g} performs a separate calculation using the
character table of $N_G(Q)$. It forms the character table of the quotient by
the computed $3$-core and selects the four irreducible ordinary characters
belonging to $3$-blocks of defect zero. It inflates these characters to
$N_G(Q)$, determines their block numbers, and verifies that each corresponding
row of the decomposition matrix has exactly one nonzero entry. Character
numbers refer to positions in the ordered lists of irreducible ordinary
characters. Each list of six entries in the output records, in order, the
character number in the quotient table, its degree and block number, the
character number and block number of its inflation, and the column containing
the nonzero decomposition number.

The program \path{fi24blocks.g} then considers each entry in the list of
$32$ class fusions returned by \texttt{PossibleClassFusions} between the
character tables of $N_G(Q)$ and $G$. No assertion that the entries are
distinct or form an orbit is used.  The stored CTblLib fusion is
one of these.  For each of the six irreducible ordinary characters at positions
$23,24,51,52,65,76$ in the character table of $N_G(Q)$, all belonging to block
$2$, the induced central character determines block $2$ of the character table
of $G$ for every returned candidate. The three
blocks of $G$ have
defects $16,2,0$.  By
\cite[Theorem~4.14 and Lemma~6.1]{Navarro1998}, these calculations with central characters give the block induction used in Section~\ref{sec:sporadic}.
This calculation treats $Q\cong C_3^2$ and does not enumerate all
nontrivial radical subgroups of $G$.

The program \path{fi24weights.g} uses the character tables of the
normalisers of radical subgroups of orders $5$ and $25$ in $Fi'_{24}$
and $3.Fi'_{24}$.  It selects the characters inflated from characters of defect zero
of the quotient by the radical subgroup. The classification in
\cite[Proposition~4.1 and Table~1]{AnCannonOBrienUngerFi24} gives the
normalisers, whose named character tables are supplied by CTblLib.  For each entry returned by \texttt{PossibleClassFusions}
for the named tables, the program induces every selected character to
the corresponding group, $Fi'_{24}$ or $3.Fi'_{24}$, requires its central character to determine a unique block,
and obtains the same block distribution.  For each returned outer fusion entry, it verifies an involutory permutation
preserving the set of selected characters.  The resulting pairs of total and fixed weight counts are independent of the
candidate choices. The program obtains the pairs $(16,16)$, $(14,6)$, and
$(16,16)$ of the three invariant blocks and the counts in table positions
$45$--$48$.  The program \path{fi24blocks.g} gives the pairs of blocks interchanged by the outer involution,
$45\leftrightarrow46$ and $47\leftrightarrow48$.  Using CTblLib's identification of the tables and the outer action,
these are the two pairs of blocks exchanged between the faithful sectors,
with $16$ and $14$ weights per block, respectively.
For the radical subgroup of order $5$, the same selection of quotient
characters and central character test place every selected weight in a
block of cyclic defect.

\subsection{The Baby Monster and the Monster}
\label{app:baby-monster}

For the Baby Monster,
\path{baby.g} uses the ordinary character table in
CTblLib, whose correctness is proved in
\cite[Section~7]{BreuerMagaardWilsonBaby2020}.  It uses \texttt{PrimeBlocks} to determine the block containing each ordinary
irreducible character. For each block, it restricts its ordinary characters to
the $2$-regular classes and computes the rank of the resulting matrix.  The two blocks have defects
$41,3$, numbers of ordinary characters $179,5$, and ranks $25$ and $2$.
Both the matrix formed by combining bases for the row spaces of these two
matrices and the matrix of all the restricted characters have rank $27$.  The calculation does not use a Baby
Monster $2$-Brauer table from CTblLib.
The program \path{baby.g} also loads the ordinary
table of $2.\Baby$ at $7$ and verifies that it has $222$ $7$-regular
classes, rather than the $220$ printed in the proof of
\cite[Lemma~5.2]{AnWilsonBaby}, and that its four blocks of defect $2$ have
restriction ranks $24,24,21,24$.  Hence the blockwise values used in
\cite[Theorem~5.3]{AnWilsonBaby} are unchanged.

For the Monster, \path{monster.g} uses the ordinary
character table in CTblLib, whose correctness is proved in
\cite[Section~5]{BreuerMagaardWilsonMonster2026}.  It verifies that the
ordinary irreducible characters form five blocks with respective defects
$(46,0,0,4,0)$ and numbers of ordinary irreducible characters
$(183,1,1,8,1)$.
It also checks the three characters of defect zero, all heights in the block
of defect $4$, and that there are $61$ $2$-regular classes.  The total number
of Brauer characters is obtained from these $61$ classes, rather than from the
number of ordinary characters.  The verified block defects and numbers of
ordinary characters
are used in Proposition~\ref{prop:four-boundary}, together with the total
number of Brauer characters and the result for the block of defect $4$.

\section{Proofs of Lemmas~\ref{lem:type-b-isolated-coverage} and~\ref{lem:symplectic-gggr-selection}}
\label{app:gggr-proofs}

\begin{proof}[Proof of Lemma~\ref{lem:type-b-isolated-coverage}]
For sufficiently large $p$ and $q$, Geck and H\'ezard prove the
conclusion in \cite[Proposition~4.3]{GeckHezard2008}. We adapt their
proof to the present hypotheses.
We first prove that the restrictions of
$D_H(\Gamma_{u_1}),\ldots,D_H(\Gamma_{u_d})$ to $\mathcal C^F$
are linearly independent. We then show that the restrictions of the
characters in $\Irr_{s,\mathcal F}(H)$ span the $H$-invariant
functions on $\mathcal C^F$. We obtain characters satisfying
\eqref{eq:type-b-isolated-pairings} by combining these facts with the
multiplicity sum formula~\eqref{eq:gggr-family-multiplicity-sum}.

Fix a prime $\ell_0\ne p$. For each $F$-stable pair
$\iota=(\mathcal C,\mathscr E_\iota)$, where $\mathscr E_\iota$ is an
irreducible $\mathbf H$-equivariant $\overline{\mathbb Q}_{\ell_0}$-local
system on $\mathcal C$, choose a Frobenius isomorphism with the
normalisation of \cite[10.1--10.8]{Taylor2016}.
Let $Y_\iota$ be the trace function associated with $\mathscr E_\iota$,
and let $X_\iota$ be the trace function associated with the corresponding
intersection cohomology complex. Extend them by zero outside
$\mathcal C^F$ and $\overline{\mathcal C}^{\,F}$, respectively.
Following \cite[11.12]{Taylor2016}, we obtain these functions from those
of \cite[10.8]{Taylor2016} by composition with the Springer isomorphism
chosen for the construction of the generalised Gelfand--Graev characters
\cite[Proposition~4.6 and Section~5]{Taylor2016}.
This is an $\mathbf H$-equivariant isomorphism from the variety of
unipotent elements of $\mathbf H$ to the variety of nilpotent elements
in its Lie algebra, chosen to commute with $F$.
Let $\Gamma_\iota^{\mathrm{mod}}$ be the modified class function defined in
\cite[(11.15)]{Taylor2016}, which is a linear combination of
$\Gamma_{u_1},\ldots,\Gamma_{u_d}$. We will prove the required linear
independence on $\mathcal C^F$ by calculating
$\langle D_H(\Gamma_\iota^{\mathrm{mod}}),X_\kappa\rangle_H$
for pairs $\iota,\kappa$ supported on $\mathcal C$, as in
\cite[Lemma~3.5]{Geck1994BasicIII}, and using the vanishing assertion in
\cite[Corollary~3.6(b)]{Geck1994BasicIII}.

To carry out Geck--H\'ezard's calculation for every odd $q$, we use
Taylor's expansions of the generalised Gelfand--Graev characters and the
modified functions in terms of intersection cohomology trace functions
\cite[Theorem~11.13 and Lemma~11.16]{Taylor2016}. These formulas assume
that $p$ is \emph{acceptable} for $\mathbf H$. Since $p$ is odd, it is
\emph{very good} for $\mathbf H$, and hence acceptable in Taylor's sense
\cite[Definition~6.1 and Lemma~6.3(i)]{Taylor2016}.
Taylor's acceptability condition also implies Letellier's condition
\cite[Assumption~5.0.14]{Letellier2005}, as explained in
\cite[6.2]{Taylor2016}.
We will use this implication later to justify the application of
Letellier's Frobenius equivariant induction isomorphism
\cite[(6.2.11)]{Letellier2005} in the comparison of Fourier constants.
We also use Lusztig's duality identities
\cite[Proposition~8.5 and Corollary~8.6]{Lusztig1992Support}.
Because $Z(\mathbf H)$ is connected, \cite[Corollary~13.6]{Taylor2016}
shows that Lusztig's duality identities hold without a restriction on $q$.

The equivariant local systems on $\mathcal C$ are parametrised by
$\Irr(A_{\mathbf H}(u_0))$ \cite[(4.2)]{Geck1999GGGR}.
Since $F$ acts trivially on $A_{\mathbf H}(u_0)$, the coefficient matrix
expressing the modified functions in terms of
$\Gamma_{u_1},\ldots,\Gamma_{u_d}$ is the character table of
$A_{\mathbf H}(u_0)$, up to nonzero scalar factors
\cite[Section~3(A)]{GeckMalle2000}. It is therefore invertible.
The orthogonality and triangularity relations and the inverse matrix
formula needed in the proof of \cite[Lemma~3.5]{Geck1994BasicIII} are
supplied by \cite[Lemma~10.9, (10.11), Theorem~10.14 and (10.20)]{Taylor2016},
without assuming that $F$ is split \cite[10.1]{Taylor2016}.
Using these identities with Taylor's expansions and Lusztig's duality
identities, the calculation of \cite[Lemma~3.5]{Geck1994BasicIII} gives
\begin{equation}\label{eq:gggr-dual-pairing}
 \bigl\langle D_H(\Gamma_\iota^{\mathrm{mod}}),X_\kappa\bigr\rangle_H
 =c_\iota\delta_{\iota\kappa},\qquad c_\iota\ne0,
\end{equation}
for pairs $\iota,\kappa$ supported on $\mathcal C$.

The proof of \cite[Corollary~3.6(b)]{Geck1994BasicIII}, using these
expansions, shows that
\begin{equation}\label{eq:gggr-dual-support}
 D_H(\Gamma_\iota^{\mathrm{mod}})(v)\ne0
 \quad\Longrightarrow\quad
 \mathcal C\subseteq\overline{\mathcal C_v}
\end{equation}
for every unipotent element $v\in H$, where $\mathcal C_v$ is its
$\mathbf H$-conjugacy class. The functions
$D_H(\Gamma_\iota^{\mathrm{mod}})$ also vanish on nonunipotent elements.
By \eqref{eq:gggr-dual-support}, they vanish on
$\overline{\mathcal C}^{\,F}\setminus\mathcal C^F$. Since $X_\kappa$
vanishes outside $\overline{\mathcal C}^{\,F}$, only values on
$\mathcal C^F$ contribute to the scalar product in
\eqref{eq:gggr-dual-pairing}. Pairing a linear relation between the
restricted modified functions with each $X_\kappa$ forces every
coefficient to vanish. The invertible coefficient matrix consequently
proves that
\[
 D_H(\Gamma_{u_1})|_{\mathcal C^F},\ldots,
 D_H(\Gamma_{u_d})|_{\mathcal C^F}
\]
are linearly independent. The same change of basis gives the vanishing property
\eqref{eq:gggr-dual-support} and vanishing on nonunipotent elements for
each $D_H(\Gamma_{u_j})$.

We next prove that the restrictions of the characters in
$\Irr_{s,\mathcal F}(H)$ span the $H$-invariant functions on
$\mathcal C^F$. For this purpose, we use the restriction theorem
\cite[Theorem~4.5]{Geck1999GGGR}, which identifies the relevant
restrictions of character sheaves with the irreducible equivariant local
systems on $\mathcal C$. To apply Geck's proof of
\cite[Theorem~4.5]{Geck1999GGGR} in odd characteristic, we first verify
the scalar product and vanishing formulas in \cite[(2.4)]{Geck1999GGGR}
and the compatibility of their constants with cuspidal induction in
\cite[(3.3)(b)]{Geck1999GGGR}. The large characteristic assumption in
\cite{Geck1999GGGR} enters through these character formulas, as
explained in \cite[Basic assumptions, p.~140]{Geck1999GGGR} and
\cite[(4.3)]{Geck1999GGGR}.

Let $\mathscr L$ be the finite collection of groups required in Geck's
reduction: $\mathbf H$, the Levi subgroups supporting the cuspidal pairs
used in \cite[(4.9)]{Geck1999GGGR}, and the adjoint quotients and further
supporting Levi subgroups occurring in the proofs of
\cite[Propositions~5.3 and 5.5]{Geck1999GGGR}.
Recall that a pair is cuspidal if its generalised Springer block consists
of that pair alone \cite[(3.3)]{Geck1999GGGR}. Every group in
$\mathscr L$ has connected centre and is either a torus or of type
$\mathsf B$ modulo its centre. Indeed, a Levi subgroup in type
$\mathsf B$ has components of type $\mathsf A$ apart from at most one
of type $\mathsf B$, and a nontrivial component of type $\mathsf A$
cannot support a cuspidal pair
when the centre is connected
\cite[Remark~3.4 and the proof of Theorem~3.8]{Geck1999GGGR}.
Passing to an adjoint quotient does not change the Dynkin type, and the
same restriction on components applies to its supporting Levi subgroups.

The restriction multiplicities are independent of the rational structure
\cite[p.~151]{Geck1999GGGR}. We therefore choose a Frobenius endomorphism
$F'$ defining a split $\mathbb F_{q'}$-structure, with $q'\equiv1\pmod4$.
This congruence will give the value $1$ for the cuspidal constants in
\cite[Theorem~3.8]{Geck1999GGGR}, once we have justified the use of its
calculation in odd characteristic. It can be ensured by taking an even
power of a split Frobenius endomorphism.
Choose $F'$ to satisfy \cite[(4.7)(a)--(c)]{Geck1999GGGR} for the chosen
semisimple elements in the dual groups and standard Levi subgroups
required in Geck's argument. As in the opening paragraph of the proof of
\cite[Proposition~5.3]{Geck1999GGGR}, we impose these conditions also for
the additional semisimple elements used in that proof.
Each chosen semisimple element is defined over a finite extension of
$\mathbb F_p$, so a common power of a split Frobenius endomorphism fixes
all these elements. This power permutes the finite sets of character
sheaves in the corresponding series. A further common power therefore
fixes every required character sheaf, as explained in
\cite[(4.7)]{Geck1999GGGR}.
Taking a sufficiently large further power makes $q'$ large enough for
geometric induction to coincide with Deligne--Lusztig induction
\cite[p.~115]{Letellier2005}.
The standard Levi subgroups and their adjoint quotients inherit split
structures from $F'$. Since $p$ is odd, it is very good for every group
in $\mathscr L$, and Taylor's formulas apply by
\cite[Lemma~6.3(i)]{Taylor2016}. Thus Letellier's characteristic hypothesis
also holds for every group in $\mathscr L$. With $q'$ chosen as above,
we may use the Frobenius equivariant isomorphism
\cite[(6.2.11)]{Letellier2005}.

To verify the formulas in \cite[(2.4)]{Geck1999GGGR} used in the proof
of Geck's restriction theorem \cite[Theorem~4.5]{Geck1999GGGR}, we need
the precise scalar product coefficients for every $\mathbf L\in\mathscr L$.
Fix $\mathbf L\in\mathscr L$, and set $L'=\mathbf L^{F'}$.
For each $F'$-stable pair $\iota=(\mathcal D,\mathscr E_\iota)$, choose
the Frobenius isomorphisms with the compatible normalisations of
\cite[10.1--10.8]{Taylor2016}. Write $Y_\iota$ for the trace function of
$\mathscr E_\iota$ and $X_\iota$ for the trace function of its
intersection cohomology complex, extended by zero outside
$\mathcal D^{F'}$ and $\overline{\mathcal D}^{\,F'}$, respectively.
As before, $\Gamma_\iota^{\mathrm{mod}}$ is the linear combination of
generalised Gelfand--Graev characters defined in \cite[(11.15)]{Taylor2016},
also given in \cite[(2.3)(a)]{Geck1999GGGR}.
Let $\mathcal I$ be the generalised Springer block containing $\iota$,
induced from a cuspidal pair $\iota_0$ of a standard Levi subgroup
$\mathbf M$. Set
\[
 a=|A_{\mathbf L}(v)|\quad(v\in\mathcal D),
 \qquad
 b_\iota=\frac{\dim\mathbf L-\dim\mathcal D-\dim Z(\mathbf M)}2.
\]
Let $\zeta_{\mathcal I}$ be the fourth root of unity in
\cite[Proposition~11.5]{Taylor2016}, and set
\[
 \delta_{\mathcal I}=(-1)^{\operatorname{rank}(\mathbf M/Z(\mathbf M))},
 \qquad
 \zeta'_{\mathcal I}=\delta_{\mathcal I}\zeta_{\mathcal I}^{-1},
\]
as in \cite[(8.4)(a)]{Lusztig1992Support} and \cite[(2.4)]{Geck1999GGGR}.

Taylor's expansions \cite[Theorem~11.13 and Lemma~11.16]{Taylor2016}
and Lusztig's duality identities
\cite[Proposition~8.5 and Corollary~8.6]{Lusztig1992Support}, valid here
by \cite[Corollary~13.6]{Taylor2016}, allow us to repeat the calculations in
\cite[Lemma~3.5 and Corollary~3.6(b)]{Geck1994BasicIII} for
$(\mathbf L,F')$. In particular, $D_{L'}(\Gamma_\iota^{\mathrm{mod}})$
vanishes on nonunipotent elements and satisfies
\begin{equation}\label{eq:gggr-auxiliary-support}
 D_{L'}(\Gamma_\iota^{\mathrm{mod}})(x)\ne0
 \quad\Longrightarrow\quad
 \mathcal D\subseteq\overline{\mathcal C_x}
\end{equation}
for unipotent $x\in L'$, where $\mathcal C_x$ is its
$\mathbf L$-conjugacy class. By \eqref{eq:gggr-auxiliary-support},
$D_{L'}(\Gamma_\iota^{\mathrm{mod}})$ vanishes on
$\overline{\mathcal D}^{\,F'}\setminus\mathcal D^{F'}$.
For a pair $\kappa$ supported on $\mathcal D$, we have
$X_\kappa=Y_\kappa$ on $\mathcal D^{F'}$, and $X_\kappa$ vanishes
outside $\overline{\mathcal D}^{\,F'}$. Replacing $X_\kappa$ by
$Y_\kappa$ therefore does not change its scalar product with
$D_{L'}(\Gamma_\iota^{\mathrm{mod}})$.
The calculation of \cite[Lemma~3.5]{Geck1994BasicIII} gives
\begin{equation}\label{eq:gggr-auxiliary-pairing}
 \bigl\langle D_{L'}(\Gamma_\iota^{\mathrm{mod}}),Y_\kappa\bigr\rangle_{L'}
 =a\,\zeta'_{\mathcal I}\,q'^{-b_\iota}\delta_{\iota\kappa}
\end{equation}
for $F'$-stable pairs $\iota,\kappa$ supported on $\mathcal D$.
Formula~\eqref{eq:gggr-auxiliary-pairing} verifies
\cite[(2.4)(a)]{Geck1999GGGR}, while
\eqref{eq:gggr-auxiliary-support}, together with vanishing on
nonunipotent elements, provides the required vanishing assertion in
\cite[(2.4)(c)]{Geck1999GGGR}. If $\mathbf M$ is a torus, then
$\zeta'_{\mathcal I}=1$, giving \cite[(2.4)(b)]{Geck1999GGGR}.

It remains to compare $\zeta'_{\mathcal I}$ with the constant attached
to the cuspidal pair $\iota_0$, computed in $\mathbf M$.
We prove the compatibility required in \cite[(3.3)(b)]{Geck1999GGGR}
using the Frobenius equivariant induction isomorphism
\cite[(6.2.11)]{Letellier2005}.
Let $\gamma_{\mathbf L}$ be the normalised Fourier constant in
\cite[Lemma~11.3]{Taylor2016} for the complex induced from $\iota_0$,
and let $\gamma_{\mathbf M}$ be the corresponding constant for the
complex on the Lie algebra of $\mathbf M$ constructed from $\iota_0$.
Thus $\gamma_{\mathbf L}$ is the scalar remaining after extracting
the factor $q'^{(\dim\mathbf L+\dim Z(\mathbf M))/2}$, with the
analogous convention for $\gamma_{\mathbf M}$.
Choose compatible Frobenius structures and an invariant form on the Lie
algebra of $\mathbf L$, restricted to that of $\mathbf M$.
Choose an $F'$-stable parabolic subgroup $\mathbf P$ of $\mathbf L$
with Levi subgroup $\mathbf M$ and unipotent radical $U_{\mathbf P}$.

For $w=1$, the induction isomorphism \cite[(6.2.11)]{Letellier2005}
includes a Tate twist by $\dim U_{\mathbf P}$, which multiplies the
Frobenius trace by $q'^{-\dim U_{\mathbf P}}$
\cite[Remark~4.4.6]{Letellier2005}.
Since $\dim\mathbf L-\dim\mathbf M=2\dim U_{\mathbf P}$, the powers
of $q'$ in these normalisations satisfy
\begin{equation}\label{eq:gggr-fourier-normalisation}
 q'^{-\dim U_{\mathbf P}}\,
 q'^{(\dim\mathbf L+\dim Z(\mathbf M))/2}
 =q'^{(\dim\mathbf M+\dim Z(\mathbf M))/2}.
\end{equation}
Taking Frobenius traces in \cite[(6.2.11)]{Letellier2005}, with the
normalisation of \cite[Lemma~11.3]{Taylor2016}, compares the Fourier
constants for $\mathbf L$ and $\mathbf M$.
Cancelling the powers of $q'$ by \eqref{eq:gggr-fourier-normalisation}
gives
\begin{equation}\label{eq:gggr-fourier-constants}
 \gamma_{\mathbf L}=\gamma_{\mathbf M}.
\end{equation}
This comparison also occurs in the proof of
\cite[Proposition~7.2]{Lusztig1992Support}, with the invariant form
restricted to the Lie algebra of the cuspidal Levi subgroup.
By \cite[Lemma~11.3]{Taylor2016},
$\gamma_{\mathbf L}^{\,2}=\nu_{\mathbf L}$, where
$\nu_{\mathbf L}$ is the sign in that lemma.
The proof of \cite[Proposition~11.5]{Taylor2016} identifies
$\zeta_{\mathcal I}=\nu_{\mathbf L}\gamma_{\mathbf L}^{-1}$.
Hence $\zeta_{\mathcal I}=\gamma_{\mathbf L}$.
The same argument in $\mathbf M$ gives
$\zeta_{\iota_0}=\gamma_{\mathbf M}$.
The signs attached to $\mathcal I$ in $\mathbf L$ and to $\iota_0$
in $\mathbf M$ satisfy
\[
 \delta_{\mathcal I}=\delta_{\iota_0}
 =(-1)^{\operatorname{rank}(\mathbf M/Z(\mathbf M))}
\]
by \cite[(8.4)(a)]{Lusztig1992Support}.
Combining these equalities with \eqref{eq:gggr-fourier-constants} and
the definition $\zeta'=\delta\zeta^{-1}$ gives
\begin{equation}\label{eq:gggr-cuspidal-compatibility}
 \zeta'_{\mathcal I}=\zeta'_{\iota_0}.
\end{equation}
This is the compatibility required in \cite[(3.3)(b)]{Geck1999GGGR}.

To determine the constants $\zeta'_{\iota_0}$ by Geck's induction, we first
establish the integrality conclusion of \cite[Corollary~3.2]{Geck1999GGGR}
for the groups in $\mathscr L$.
Let $\mathbf L\in\mathscr L$ be a group other than a torus.
For every $F'$-stable unipotent class $\mathcal D$ of $\mathbf L$,
the theorem of H\'ezard and Lusztig stated in
\cite[Theorem~1.1]{Taylor2013} gives a character $\chi\in\Irr(L')$
with unipotent support $\mathcal D$ and generic denominator
$n_\chi=|A_{\mathbf L}(v)|$ for $v\in\mathcal D$.
The degree polynomial relation for Alvis--Curtis duality gives
$n_{\chi^*}=n_\chi$ \cite[Proposition~3.4.21]{GeckMalle2020}. If $v_1,\ldots,v_r$
represent the $L'$-classes in $\mathcal D^{F'}$, then \cite[(11.15), Lemma~14.15 and
Proposition~15.4]{Taylor2016} give
\begin{equation}\label{eq:gggr-weighted-integrality-sum}
 \sum_{j=1}^{r}
 [A_{\mathbf L}(v_j):A_{\mathbf L}(v_j)^{F'}]
 \langle\chi^*,\Gamma_{v_j}\rangle_{L'}
 =\frac{|A_{\mathbf L}(v)|}{n_{\chi^*}}=1.
\end{equation}
The multiplicities are nonnegative integers and the indices are positive
integers, so exactly one multiplicity is nonzero, and it and its corresponding
index equal $1$. Expanding a modified
generalised Gelfand--Graev class function indexed by a local system on
$\mathcal D$ then expresses its scalar product with $\chi^*$ as the value
of that local system's trace function at the unique representative with
nonzero multiplicity. This is the
remaining assertion of \cite[Proposition~3.1]{Geck1999GGGR}.
For a torus, the only unipotent element is $1$ and every irreducible
character has generic denominator $1$, so the same multiplicity conclusions hold.
Together with the identities proved above, this establishes
\cite[Corollary~3.2]{Geck1999GGGR} for every group in $\mathscr L$.

The induction in the proof of \cite[Theorem~3.8]{Geck1999GGGR} therefore applies
to every group in $\mathscr L$. For a cuspidal pair $\iota_0$ on $\mathbf M$,
it gives
\begin{equation}\label{eq:gggr-cuspidal-constant}
 \zeta'_{\iota_0}
 =\varepsilon^{(\operatorname{rank}\mathbf M-\dim Z(\mathbf M))/2},
 \qquad q'\equiv\varepsilon\pmod4,\quad \varepsilon\in\{1,-1\}.
\end{equation}
Our choice $q'\equiv1\pmod4$ gives $\varepsilon=1$, so $\zeta'_{\iota_0}=1$.
By \eqref{eq:gggr-cuspidal-compatibility}, $\zeta'_{\mathcal I}=1$ for every
relevant generalised Springer block. The support properties in \cite[(4.3)]{Geck1999GGGR}
follow from \cite[Theorem~10.7]{Lusztig1992Support} by
\cite[Corollary~13.6]{Taylor2016}. For this $F'$,
\cite[Theorem~3.2]{Shoji1995II} gives the equality of spans stated in
\cite[Theorem~7.1]{Geck1999GGGR} for these classical groups with connected centre.
For a torus, each rational Lusztig series consists of one linear character,
and the trace function of the corresponding character sheaf is a nonzero
scalar multiple of it. Thus the same equality of spans holds. This verifies the formulas and comparisons used in Geck's
proof of the restriction theorem \cite[Theorem~4.5]{Geck1999GGGR}.

We now verify the conditions in \cite[(4.4)(a)--(b)]{Geck1999GGGR} for $(s,\mathcal F)$.
The elementary abelian groups $\mathcal G_{\mathcal F}$ and
$A_{\mathbf H}(u_0)$ have the same order and are therefore isomorphic,
giving \cite[(4.4)(a)]{Geck1999GGGR}. Isolation together with
this order equality gives \cite[(4.4)(b)]{Geck1999GGGR}, by H\'ezard's
result \cite[Proposition~2.3]{GeckHezard2008}. By the restriction theorem \cite[Theorem~4.5]{Geck1999GGGR}, restriction to
$\mathcal C$, up to shift, gives a bijection between the character sheaves in
the family indexed by $(s,\mathcal F)$ whose restriction to $\mathcal C$ is
nonzero and the irreducible $\mathbf H$-equivariant local systems on $\mathcal
C$. The multiplicities
in these geometric restrictions are independent of the rational structure,
as explained after \cite[Theorem~4.5]{Geck1999GGGR}.

We now return to the given Frobenius endomorphism $F$ and use this restriction
bijection to show that the restrictions of the characters in
$\Irr_{s,\mathcal F}(H)$ span the $H$-invariant functions on $\mathcal C^F$.
The choice $\mathbf T^*\subseteq\mathbf B^*$ and the absence of nontrivial
graph automorphisms in type $\mathsf C_n$ imply that $F$ acts trivially
on the Weyl group of $\mathbf H^*$ and on $W_s$. Since $s$ is fixed by $F$,
$F$ preserves the family of character sheaves indexed by $(s,\mathcal F)$.
Every irreducible $\mathbf H$-equivariant local system on $\mathcal C$ is
$F$-stable because $F$ acts trivially on $A_{\mathbf H}(u_0)$.
Restriction to $\mathcal C$ commutes with pullback by $F$.
The uniqueness in the restriction bijection therefore makes the corresponding
character sheaves $F$-stable, as in the proof of
\cite[Corollary~3.5]{GeckHezard2008}.

Since $F$ acts trivially on $A_{\mathbf H}(u_0)$, the values at
$u_1,\ldots,u_d$ of the trace functions of these local systems form the
character table of $A_{\mathbf H}(u_0)$, up to nonzero scalar multiples
of the rows. The same holds for the restrictions to $\mathcal C^F$ of
the trace functions of the corresponding character sheaves.
By Shoji's theorem \cite[Theorem~3.2]{Shoji1995II}, the trace functions
of these character sheaves are nonzero scalar multiples of almost
characters lying in the span of $\Irr_{s,\mathcal F}(H)$. It follows that there are characters
$\chi_1,\ldots,\chi_d\in\Irr_{s,\mathcal F}(H)$ such that
\[
 M=(\chi_i(u_j))_{1\leq i,j\leq d}
\]
is nonsingular.

To apply the multiplicity formula to every character in $\Irr_{s,\mathcal F}(H)$,
we identify its unipotent support and generic denominator.
Let $W$ be the Weyl group of $\mathbf H^*$ and let $E_0$ be the unique special
character in $\mathcal F$. Its truncated induction $j_{W_s}^W(E_0)$
is the unique constituent of
$\operatorname{Ind}_{W_s}^W(E_0)$ with $b$-invariant $b(E_0)$
\cite[p.~476]{Taylor2014Maximal}.
Under the Springer correspondence, $j_{W_s}^W(E_0)$ corresponds to
$\mathcal C$ \cite[Section~3(C)]{GeckMalle2000}. Consequently, every
$\rho\in\Irr_{s,\mathcal F}(H)$ has unipotent support $\mathcal C$
\cite[Theorem~3.7]{GeckMalle2000}.
Since $\mathcal G_{\mathcal F}$ is abelian, the denominator formula
\cite[(6.1)(c)]{Geck1999GGGR} and the degree polynomial relation for
Alvis--Curtis duality \cite[Proposition~3.4.21]{GeckMalle2020} give
\[
 n_{\rho^*}=n_\rho=|\mathcal G_{\mathcal F}|.
\]
By \cite[Lemma~14.15]{Taylor2016}, $\mathcal C$ is the wave front set of $\rho^*$.

Conjugation in $\mathbf H$ identifies the component groups of $u_j$ and $u_0$.
Under this identification, the actions of $F$ differ by an inner automorphism
of $A_{\mathbf H}(u_0)$. Since this group is abelian and $F$ acts trivially on it,
$F$ also acts trivially on $A_{\mathbf H}(u_j)$. Consequently,
\[
 [A_{\mathbf H}(u_j):A_{\mathbf H}(u_j)^F]=1
 \qquad(1\leq j\leq d).
\]
Apply the formula of Lusztig and Geck--Malle in \cite[Proposition~15.4]{Taylor2016}
to $\rho^*$, using the constant local system on $\mathcal C$ with trace function
identically $1$. In the expansion \cite[(11.15)]{Taylor2016}, the coefficient
of $\Gamma_{u_j}$ is the index just calculated. Hence
\begin{equation}\label{eq:gggr-family-multiplicity-sum}
 \sum_{j=1}^{d}\langle\rho^*,\Gamma_{u_j}\rangle_H
 =\frac{|A_{\mathbf H}(u_0)|}{n_{\rho^*}}=1.
\end{equation}
Each multiplicity is a nonnegative integer, so there is a unique index $j$
for which $\langle\rho^*,\Gamma_{u_j}\rangle_H=1$, and all the other
multiplicities vanish. For $1\leq j\leq d$, define
\[
 I_j=\{\rho\in\Irr_{s,\mathcal F}(H):
       \langle\rho^*,\Gamma_{u_j}\rangle_H=1\}.
\]
It remains to prove that every $I_j$ is nonempty.

Suppose that $I_r$ is empty. Then $\langle\rho^*,\Gamma_{u_r}\rangle_H=0$
for every $\rho\in\Irr_{s,\mathcal F}(H)$.
Since $D_H(\rho)=\pm\rho^*$, self-adjointness of Alvis--Curtis duality gives
\[
 \langle\rho,D_H(\Gamma_{u_r})\rangle_H
 =\langle D_H(\rho),\Gamma_{u_r}\rangle_H=0.
\]
As proved after \eqref{eq:gggr-dual-support}, $D_H(\Gamma_{u_r})$ vanishes on
nonunipotent elements. By \eqref{eq:gggr-dual-support}, its value at a
unipotent element $v$ can be nonzero only if
$\mathcal C\subseteq\overline{\mathcal C_v}$.
If $\mathcal C_v\ne\mathcal C$, this inclusion implies
$\dim\mathcal C_v>\dim\mathcal C$.
On such a class, $\rho$ vanishes by \cite[Theorem~11.2(iv)]{Lusztig1992Support},
applied to $\rho^*$. This theorem applies here by
\cite[Corollary~13.6]{Taylor2016}.
Thus only elements of $\mathcal C^F$ contribute to
$\langle\rho,D_H(\Gamma_{u_r})\rangle_H$, and
\[
 \sum_{j=1}^{d}
 \frac{\rho(u_j)\,
       \overline{D_H(\Gamma_{u_r})(u_j)}}{|C_H(u_j)|}=0
 \qquad(\rho\in\Irr_{s,\mathcal F}(H)).
\]
Apply these equalities to $\chi_1,\ldots,\chi_d$.
The nonsingularity of $M=(\chi_i(u_j))_{1\leq i,j\leq d}$ implies that
$D_H(\Gamma_{u_r})(u_j)=0$ for every $j$.
Its restriction to $\mathcal C^F$ is therefore zero, contradicting the linear
independence established after \eqref{eq:gggr-dual-pairing}.
Hence every $I_j$ is nonempty. Choosing $\rho_j\in I_j$ gives
\eqref{eq:type-b-isolated-pairings}.
\end{proof}

\begin{proof}[Proof of Lemma~\ref{lem:symplectic-gggr-selection}]
We adapt the proof of Lemma~\ref{lem:type-b-isolated-coverage} to the selected
type~$\mathsf C$ families. We first verify that the central multiple
$\widetilde s$ in the statement exists and is fixed by $F_0$.
Condition~\textup{(P4)} of \cite[Theorem~2.11]{Taylor2014Maximal} will replace
the isolation hypothesis used in Lemma~\ref{lem:type-b-isolated-coverage}.
By \cite[6.19]{Taylor2014Maximal}, the image of $s_0$ in
$\mathbf H^*/Z(\mathbf H^*)\cong\operatorname{SO}_{2m+1}$ is
quasi-isolated, so its square is one
\cite[Proposition~4.11(a)]{Bonnafe2005}. Thus $s_0^2\in Z(\mathbf H^*)$.
Since $\mathbf H$ has simply connected derived subgroup,
$Z(\mathbf H^*)$ is a torus \cite[Lemma~1.5.22]{GeckMalle2020}
and contains a square root of $s_0^{-2}$. This proves that $z$ can
be chosen as in the statement.
Multiplication by $z$ leaves the centraliser unchanged, so
$W_{\widetilde s}=W_{s_0}$ and the selected family is still $\mathcal F$.
The Frobenius endomorphism $F_0$ acts on $\mathbf T^*$ by the
$q_0$-power map. Since $q_0$ is odd and $\widetilde s^2=1$,
$F_0(\widetilde s)=\widetilde s$.
The torus $\mathbf T^*$ is split, so $F_0$ acts trivially on the
Weyl group of $\mathbf H^*$ and hence on $W_{\widetilde s}$.

We next verify the conditions in \cite[(4.4)(a)--(b)]{Geck1999GGGR}
for $(\widetilde s,\mathcal F)$.
By \cite[10.3]{Taylor2014Maximal}, we identify
$W_{\widetilde s}$ with $W(\mathsf D_a\mathsf B_b)$, where $a+b=m$.
When $a=1$, the corresponding factor is a torus with trivial Weyl group,
so $W_{\widetilde s}=W(\mathsf B_{m-1})$.
Let $E_0$ be the unique special character in $\mathcal F$,
using the terminology of \cite[p.~476]{Taylor2014Maximal}, and let
$W$ be the Weyl group of $\mathbf H^*$.
H\'ezard's construction \cite[Proposition~10.2]{Taylor2014Maximal} gives
$j_{W_{\widetilde s}}^W(E_0)=E_{\mathcal C}$, where
$E_{\mathcal C}$ is the Springer character corresponding to
$\mathcal C$ and its trivial local system.
The construction gives
$|\mathcal G_{\mathcal F}|=|A_{\mathbf H}(u)|$ for $u\in\mathcal C$
by \cite[7.1 and Proposition~10.2]{Taylor2014Maximal}.
The component group $A_{\mathbf H}(u)$ is an elementary abelian
$2$-group by \cite[proof of Proposition~2.4]{Taylor2013}, and the
same holds for $\mathcal G_{\mathcal F}$ by
\cite[Remark~4.2.17]{GeckMalle2020}.
These two groups are therefore isomorphic, giving
\cite[(4.4)(a)]{Geck1999GGGR}.
The selected pair satisfies condition~\textup{(P4)} of
\cite[Theorem~2.11]{Taylor2014Maximal}. For these type~$\mathsf C$
pairs, the verification in \cite[7.4]{Taylor2014Maximal} uses the equality
$\widetilde s^{\,2}=1$ and does not require $\widetilde s$ to be isolated.
By \cite[Definition~2.7]{Taylor2014Maximal}, condition~\textup{(P4)} gives
\[
 \operatorname{Ind}_{W(\mathsf D_a\mathsf B_b)}^W(E_0)
 =E_{\mathcal C}
  +\sum_{\substack{E\in\Irr(W)\\d(E)>b(E_0)}}m_EE,
 \qquad m_E\in\mathbb Z_{\geq0},
\]
where $d(E)$ is the dimension of the Springer fibre, the variety of
Borel subgroups containing a representative of the unipotent class
corresponding to $E$. This is the induction condition
\cite[(4.4)(b)]{Geck1999GGGR}.

To apply Geck's restriction theorem \cite[Theorem~4.5]{Geck1999GGGR}
in odd characteristic, we verify analogues of \eqref{eq:gggr-auxiliary-support}
and \eqref{eq:gggr-auxiliary-pairing} for $\mathbf H$, its standard Levi subgroups
supporting cuspidal pairs, their adjoint quotients, and the further supporting
Levi subgroups used in \cite[(4.9)]{Geck1999GGGR}.
Let $\mathbf L$ be any one of these algebraic groups.
These groups have connected centre and are either tori or simple modulo
their centre of type~$\mathsf C$. Indeed, a nontrivial component of type
$\mathsf A$ cannot support a cuspidal pair when the centre is connected
\cite[Remark~3.4 and the proof of Theorem~3.8]{Geck1999GGGR}.
Odd characteristic is very good for these groups and hence acceptable in
Taylor's sense by \cite[Lemma~6.3(i)]{Taylor2016}.
Choose a sufficiently large and divisible positive integer $N$ such that
$F'=F_0^{2N}$ satisfies \cite[(4.7)(a)--(c)]{Geck1999GGGR} and the
conditions on the additional objects in the proofs of
\cite[Propositions~5.3 and 5.5]{Geck1999GGGR}.
As in the proof of Lemma~\ref{lem:type-b-isolated-coverage}, take $N$ large
enough for geometric induction to coincide with Deligne--Lusztig induction
\cite[p.~115]{Letellier2005} for every group used in the proofs of
\cite[Propositions~5.3 and~5.5]{Geck1999GGGR}.
The field size satisfies $q'=q_0^{2N}\equiv1\pmod4$.
The calculations in the proof of Lemma~\ref{lem:type-b-isolated-coverage},
using \cite[Theorem~11.13 and Lemma~11.16]{Taylor2016}, now give
the analogues of \eqref{eq:gggr-auxiliary-support} and
\eqref{eq:gggr-auxiliary-pairing} for these groups.
The comparison of Fourier constants giving
\eqref{eq:gggr-fourier-constants} and
\eqref{eq:gggr-cuspidal-compatibility} also applies.
These calculations establish the formulas in \cite[(2.4)]{Geck1999GGGR} and
the compatibility of the constants $\zeta'_{\mathcal I}$ with cuspidal
induction in \cite[(3.3)(b)]{Geck1999GGGR}.

If $\mathbf L$ is not a torus,
\cite[Theorem~1.1]{Taylor2013} again supplies a character with unipotent
support $\mathcal D$ and generic denominator $|A_{\mathbf L}(v)|$ for
every $F'$-stable unipotent class $\mathcal D$ and $v\in\mathcal D$.
The calculation of \eqref{eq:gggr-weighted-integrality-sum} therefore applies
to these groups. In this sum, the unique nonzero multiplicity and its
corresponding index both equal $1$, as in the proof of
Lemma~\ref{lem:type-b-isolated-coverage}. Expanding the modified class
functions then gives the scalar products in
\cite[Proposition~3.1]{Geck1999GGGR}.
The torus case was included in that proof. Together with the analogues of
\eqref{eq:gggr-auxiliary-support} and \eqref{eq:gggr-auxiliary-pairing}
obtained above, this gives the integrality conclusion of
\cite[Corollary~3.2]{Geck1999GGGR} for all these groups.
Formula~\eqref{eq:gggr-cuspidal-constant} consequently applies and gives
$\zeta'_{\iota_0}=1$, since $q'\equiv1\pmod4$.
By \eqref{eq:gggr-cuspidal-compatibility}, $\zeta'_{\mathcal I}=1$
for every relevant generalised Springer block.

It remains to verify the support properties and compare character sheaf trace
functions with ordinary characters.
The support properties in \cite[(4.3)]{Geck1999GGGR} follow from
\cite[Theorem~10.7]{Lusztig1992Support} by
\cite[Corollary~13.6]{Taylor2016}.
For the groups above other than tori, Shoji's theorem
\cite[Theorem~3.2]{Shoji1995II} gives the equality of spans in
\cite[Theorem~7.1]{Geck1999GGGR}, since their centres are connected and
their quotients by their centres are simple of classical type.
The equality of spans for tori was established in the proof of
Lemma~\ref{lem:type-b-isolated-coverage}.
Having verified the conditions in \cite[(4.4)(a)--(b)]{Geck1999GGGR}, we may
therefore apply the proof of \cite[Theorem~4.5]{Geck1999GGGR}.
Restriction to $\mathcal C$, up to shift, gives a bijection between the
character sheaves in the family indexed by $(\widetilde s,\mathcal F)$
with nonzero restriction to $\mathcal C$ and the irreducible
$\mathbf H$-equivariant local systems on $\mathcal C$.

We now return to $F_0$ and use this restriction bijection to show that the
restrictions of the characters in $\Irr_{\widetilde s,\mathcal F}(H)$ span
the $H$-invariant functions on $\mathcal C^{F_0}$.
The multiplicities in the restrictions of character sheaves are independent
of the rational structure, as explained after
\cite[Theorem~4.5]{Geck1999GGGR}.
Choose $u_0\in\mathcal C^{F_0}$ as in \cite[Proposition~2.4]{Taylor2013},
with $F_0$ acting trivially on $A_{\mathbf H}(u_0)$.
Since $F_0$ fixes $\widetilde s$ and acts trivially on
$W_{\widetilde s}$, it preserves the chosen family of character sheaves.
The trivial action on $A_{\mathbf H}(u_0)$ makes the local systems on
$\mathcal C$ stable under $F_0$. Since restriction commutes with pullback
by $F_0$, the restriction bijection makes the corresponding character
sheaves stable under $F_0$ as well. The values at $u_1,\ldots,u_d$ of
the local system trace functions form the character table of
$A_{\mathbf H}(u_0)$, up to multiplication of each row by a nonzero scalar
\cite[Section~3(A)]{GeckMalle2000}. The restriction bijection therefore
shows that the corresponding character sheaf trace functions restrict to
a basis of the $H$-invariant functions on $\mathcal C^{F_0}$.
By Shoji's theorem \cite[Theorem~3.2]{Shoji1995II}, these trace functions
are nonzero scalar multiples of almost characters in the span of
$\Irr_{\widetilde s,\mathcal F}(H)$. Thus there are
$\chi_1,\ldots,\chi_d\in\Irr_{\widetilde s,\mathcal F}(H)$ for which
$M=(\chi_i(u_j))_{1\leq i,j\leq d}$ is nonsingular, as in the proof of
\cite[Corollary~3.5]{GeckHezard2008}.

It remains to select the characters with the required multiplicities.
The calculations giving \eqref{eq:gggr-dual-pairing} and
\eqref{eq:gggr-dual-support} in the proof of
Lemma~\ref{lem:type-b-isolated-coverage} apply to $\mathbf H$ with $F_0$.
They give linear independence of
$D_H(\Gamma_{u_1})|_{\mathcal C^{F_0}},\ldots,
D_H(\Gamma_{u_d})|_{\mathcal C^{F_0}}$ by the same change from the modified
class functions to the individual generalised Gelfand--Graev characters.
Since $j_{W_{\widetilde s}}^W(E_0)=E_{\mathcal C}$, the description in
\cite[Section~3(C) and Theorem~3.7]{GeckMalle2000} shows that every
$\rho\in\Irr_{\widetilde s,\mathcal F}(H)$ has unipotent support $\mathcal C$.
The action of $F_0$ on $W_{\widetilde s}$ is trivial and
$\mathcal G_{\mathcal F}$ is abelian, so the generic denominator formula
\cite[(6.1)(c)]{Geck1999GGGR} and the degree polynomial relation
for Alvis--Curtis duality \cite[Proposition~3.4.21]{GeckMalle2020} give
\[
 n_{\rho^*}=n_\rho
 =|\mathcal G_{\mathcal F}|=|A_{\mathbf H}(u_0)|.
\]
Changing the rational representative from $u_0$ to $u_j$ changes the
action of $F_0$ on the component group by an inner automorphism of
$A_{\mathbf H}(u_0)$. This group is abelian, so $F_0$ acts trivially
on every $A_{\mathbf H}(u_j)$ and
$[A_{\mathbf H}(u_j):A_{\mathbf H}(u_j)^{F_0}]=1$.
By \cite[Lemma~14.15]{Taylor2016}, $\mathcal C$ is the wave front set of
$\rho^*$. Apply \cite[Proposition~15.4]{Taylor2016} to $\rho^*$ and the
constant local system on $\mathcal C$ with trace function identically $1$.
The expansion \cite[(11.15)]{Taylor2016} now gives
\[
 \sum_{j=1}^{d}\langle\rho^*,\Gamma_{u_j}\rangle_H
 =\frac{|A_{\mathbf H}(u_0)|}{n_{\rho^*}}=1.
\]
The multiplicities are nonnegative integers. Thus there is a unique index
$j$ for which $\langle\rho^*,\Gamma_{u_j}\rangle_H=1$, and all the other
multiplicities vanish. We must show that every index occurs for some
$\rho\in\Irr_{\widetilde s,\mathcal F}(H)$.

Suppose that $\langle\rho^*,\Gamma_{u_r}\rangle_H=0$ for every
$\rho\in\Irr_{\widetilde s,\mathcal F}(H)$.
Self-adjointness of $D_H$ and $D_H(\rho)=\pm\rho^*$ give
$\langle\rho,D_H(\Gamma_{u_r})\rangle_H=0$.
As in the final argument in the proof of
Lemma~\ref{lem:type-b-isolated-coverage}, the support calculation
corresponding to \eqref{eq:gggr-dual-support} and the pointwise vanishing of
$\rho$ on larger unipotent classes show that only elements of $\mathcal
C^{F_0}$ contribute to $\langle\rho,D_H(\Gamma_{u_r})\rangle_H$.
Taking $\rho=\chi_i$ for $1\leq i\leq d$ in these equalities and using the
nonsingularity of $M=(\chi_i(u_j))_{1\leq i,j\leq d}$ gives
$D_H(\Gamma_{u_r})|_{\mathcal C^{F_0}}=0$, as in the final matrix calculation
in the proof of Lemma~\ref{lem:type-b-isolated-coverage}.
This contradicts the linear independence of the restricted functions.
For each $j$, we may therefore choose $\rho_j$ with
$\langle\rho_j^*,\Gamma_{u_j}\rangle_H=1$.
All its other multiplicities vanish, proving the stated identities.
\end{proof}

\section*{Acknowledgements}

OpenAI Codex was used to assist with the preparation of this manuscript
(using GPT 5.6 Sol) and the development of the accompanying Lean
formalisation (using GPT 5.6 Sol and GPT 6 Astra). The text generated by AI
has been reviewed and revised. The formalisation contains more than two
thousand Lean source files. This work has also produced a reusable Lean
library for the ordinary and modular representation theory of finite groups,
which continues to be expanded. The Lean sources and a report describing the
formalised results, external assumptions and verification procedures will
be made available through the
\href{https://github.com/BaoyuZhang747}{\textcolor{blue!60!black}{\mbox{GitHub page}}}
once the manuscript appears on arXiv. The Claude models Opus 5, Fable 5 and
Fable 5.1 were also used to assist with the review and revision of the
Lean files.

I would like to express my sincere appreciation to those who laid the
foundations of the modern representation theory of finite groups and
contributed to its development. The advances and corrections presented here
build largely on methods and results established before the widespread use
of AI. The Lean project initially aimed to develop, from basic foundations, the
representation theory needed for arguments concerning Alperin's weight
conjecture. A further motivation was to clarify the assumptions and
dependencies of commonly used results and to help identify and resolve
inaccuracies in the literature. It has been known for decades that inaccuracies in earlier work have
caused further errors and misunderstandings in subsequent work. However,
it remains unclear how many arguments in the foundational literature
might benefit from further clarification or correction. Systematic investigation demands
substantial time and effort and can be difficult to sustain alongside
other research. We believe that AI assistance now makes such work more
practicable and can deepen our understanding of the mathematics involved.
AI assistance may also help resolve misunderstandings and correct
misattributions in the literature by tracing results and arguments to
their original sources.
This is a valuable use of AI in mathematics, although the cost of sustained
AI assistance remains a significant constraint.
After weeks of cumulative working time recorded by Codex for the
formalisation project, it became clear that pursuing this aim would still
require considerable time to identify and formalise a substantial body of
prerequisite theory, even with AI assistance. For example, we estimate that
formalising Deligne--Lusztig theory alone would require several people
working for at least several months. This illustrates the scale of the work involved in
formalising the background needed for modern modular representation theory.
The present formalisation therefore checks selected arguments under
explicitly stated external assumptions, while work continues on the
underlying foundations and on extending the Lean library. The reusability of the library's constructions has been tested in ongoing work
on formalising
the much longer proofs concerning the iBAW condition for the remaining families of
types $\mathsf D$ and $\mathsf E$, for which the problem was previously
largely open. The Lean library is also being used to formalise and check
proofs of new results on other problems of current interest in the
modular representation theory of finite groups.

Using AI in the modular representation theory of finite groups has also
provided insights into how to approach mathematical problems.
I regard this as a meaningful exercise because modern modular
representation theory rests on many layers of prerequisite theory,
with very long chains of arguments connecting its results to the basic
foundations of mathematics. In my experience, these very long chains of dependencies
make it harder for current AI models to develop and test new theory here
than in areas whose arguments lie closer to the basic foundations of
logic. In experiments involving hundreds of billions of tokens, I have found that AI
models, when tackling some problems directly, repeatedly return to the same
approaches without making progress, whereas
developing a suitable mathematical architecture, with the intermediate
results and their dependencies organised in an effective order, can lead
to rapid advances. These experiments suggest that the order in which
related conjectures are approached matters. Establishing one key
conjecture may make problems that previously resisted AI assistance
substantially more accessible to it. When appropriate, I hope to share
these observations and strategies for approaching proofs of the remaining
major conjectures in modular representation theory after Alperin's weight
conjecture. Writing out all these
arguments in full detail could require hundreds of billions of additional
tokens and result in more than a thousand pages of proofs, making it an
impractical undertaking for one person. Sharing the proposed strategies would allow mathematicians with greater
expertise in these conjectures than I have to assess and develop them.

\section*{Declaration of competing interest}

The author declares no competing interests. This research is not funded
by any institution or organisation.

\end{document}